\documentclass[reqno,11pt]{amsart}

\usepackage[english]{babel} 						
\usepackage[T1]{fontenc}							
\usepackage[utf8]{inputenx}

\usepackage{mathtools}								
\usepackage{empheq}									
\usepackage{enumitem}								

\usepackage{amssymb}								
\usepackage[sc]{mathpazo}							
\usepackage{mathrsfs}								

\usepackage{todonotes}								
\usepackage{aliascnt}								
\usepackage{comment}

\usepackage[normalem]{ulem}
\usepackage{epsfig}
\usepackage{a4wide,color,eucal,mathrsfs}
\usepackage[normalem]{ulem}
\usepackage{amsmath,amssymb,epsfig,amsthm} 

\usepackage[babel]{csquotes}							

\usepackage{nameref}									
\usepackage[colorlinks	=	true,						
          	 	linkcolor	=	red,
            	urlcolor	=	red,
            	citecolor	=	red]%
            	{hyperref}
\usepackage{bookmark}									
\usepackage{cleveref}									

\newcommand{\apmd}[2][]{							
	\ifthenelse{\equal{#1}{}}%
					{ \operatorname{N}_{#2}	}%
					{ \operatorname{N}_{#1,#2} 	}}

\newcommand{\LIP}{{\rm LIP}}

\newcommand{\loc}{{\rm loc}}

\newcommand{\aint}[2][]{
	\ifthenelse{\equal{#1}{}}%
					{%
\mathchoice%
      {\mathop{\kern 0.2em\vrule width 0.6em height 0.69678ex depth -0.58065ex
              \kern -0.8em \intop}\nolimits_{\kern -0.45em#2}^{#1}}%
      {\mathop{\kern 0.1em\vrule width 0.5em height 0.69678ex depth -0.60387ex
              \kern -0.6em \intop}\nolimits_{#2}^{#1}}%
      {\mathop{\kern 0.1em\vrule width 0.5em height 0.69678ex depth -0.60387ex
              \kern -0.6em \intop}\nolimits_{#2}^{#1}}%
      {\mathop{\kern 0.1em\vrule width 0.5em height 0.69678ex depth -0.60387ex
              \kern -0.6em \intop}\nolimits_{#2}^{#1}}}%
					{%
\mathchoice%
      {\mathop{\kern 0.2em\vrule width 0.6em height 0.69678ex depth -0.58065ex
              \kern -0.8em \intop}\nolimits_{\kern -0.45em#1}^{#2}}%
      {\mathop{\kern 0.1em\vrule width 0.5em height 0.69678ex depth -0.60387ex
              \kern -0.6em \intop}\nolimits_{#1}^{#2}}%
      {\mathop{\kern 0.1em\vrule width 0.5em height 0.69678ex depth -0.60387ex
              \kern -0.6em \intop}\nolimits_{#1}^{#2}}%
      {\mathop{\kern 0.1em\vrule width 0.5em height 0.69678ex depth -0.60387ex
              \kern -0.6em \intop}\nolimits_{#1}^{#2}}}}

\numberwithin{equation}{section}

\newtheorem{theorem}{Theorem}[section]

\newtheorem{corollary}[theorem]{Corollary}
\newtheorem{lemma}[theorem]{Lemma}
\newtheorem{proposition}[theorem]{Proposition}

\newtheorem{definition}[theorem]{Definition}

\newtheorem{example}[theorem]{Example}
\newtheorem{question}[theorem]{Question}

\theoremstyle{remark}
\newtheorem{remark}[theorem]{Remark}

\DeclareMathOperator{\diam}{diam}

\newcommand{\N}{\mathbb{N}}

\newcommand{\R}{\mathbb{R}}

\def\rr{{\mathbb R}}
\def\rn{{{\rr}^n}}

\def\fz{\infty}

\def\loc{{\mathop\mathrm{\,loc\,}}}

\def\boz{{\Omega}}

\definecolor{emerald}{rgb}{0.31, 0.78, 0.47}

\def\bint{{\ifinner\rlap{\bf\kern.35em--}
\int\else\rlap{\bf\kern.45em--}\int\fi}\ignorespaces}

\def\bbint{{\ifinner\rlap{\bf\kern.35em--}
\hspace{0.078cm}\int\else\rlap{\bf\kern.45em--}\int\fi}\ignorespaces}

\def\diam{{\mathop\mathrm{\,diam\,}}}

\def\dfrac{\displaystyle\frac}

\def\r{\right}
\def\lf{\left}

\def\bint{{\ifinner\rlap{\bf\kern.35em--}
\int\else\rlap{\bf\kern.45em--}\int\fi}\ignorespaces}

\begin{document}

\title[]
{Extensions and approximations of Banach-valued Sobolev functions}

\author{Miguel Garc\'ia-Bravo}

\address{Departamento de An\'alisis Matem\'atico y Matem\'atica Aplicada, Facultad de Ciencias Matem\'aticas, Universidad Complutense, 28040, Madrid, Spain}

\email{miguel05@ucm.es}

\author{Toni Ikonen} 
\address{University of Jyvaskyla \\
         Department of Mathematics and Statistics \\
         P.O. Box 35 (MaD) \\
         FI-40014 University of Jyvaaskyla \\
         Finland}
         
\email{toni.m.h.ikonen@jyu.fi}

\author{Zheng Zhu}

\address{ Beihang University\\
        School of Mathematical science\\
        Changping District Shahe Higher Education Park South Third Street No. 9\\
        Beijing 102206,\\
        P. R. China.\\
and University of Jyvaskyla \\
         Department of Mathematics and Statistics \\
         P.O. Box 35 (MaD) \\
         FI-40014 University of Jyvaaskyla \\
         Finland}
         
\email{zheng.z.zhu@jyu.fi}

\thanks{The first named author was partly supported by the Academy of Finland, project number 314789. The second named author was supported by the Academy of Finland, project number 308659 and by the Vilho, Yrjö and Kalle Väisälä Foundation. The third named author was supported by the Academy of Finland, project number 323960.
\newline {\it 2020 Mathematics Subject Classification.} Primary: 46E35. Secondary: 46E36, 30L99}
\keywords{}

\maketitle

\begin{abstract}
In complete metric measure spaces equipped with a doubling measure and supporting a weak Poincaré inequality, we investigate when a given Banach--valued Sobolev function defined on a subset satisfying a measure-density condition is the restriction of a Banach--valued Sobolev function defined on the whole space. We investigate the problem for Haj{\l}asz-- and Newton--Sobolev spaces, respectively.

First, we show that Haj\l{}asz--Sobolev extendability is independent of the target Banach spaces. We also show that every $c_0$-valued Newton--Sobolev extension set is a Banach-valued Newton--Sobolev extension set for every Banach space. We also prove that any measurable set satisfying a measure-density condition and a weak Poincaré inequality up to some scale is a Banach-valued Newton--Sobolev extension set for every Banach space. Conversely, we verify a folklore result stating that when $n\leq p<\infty$, every $W^{1,p}$-extension domain $\Omega\subset\R^n$ supports a weak $(1,p)$-Poincaré inequality up to some scale.

As a related result of independent interest, we prove that in any metric measure space when $1 \leq p < \infty$ and real-valued Lipschitz functions with bounded support are norm-dense in the real-valued $W^{1,p}$-space, then Banach-valued Lipschitz functions with bounded support are energy-dense in every Banach-valued $W^{1,p}$-space whenever the Banach space has the so-called metric approximation property.

\end{abstract}

\section{Introduction}
A metric measure space is a triple $( Z, d, \mu )$, where $(Z,d)$ is a separable metric space and $\mu$ is a Borel regular outer measure satisfying $0 < \mu( B ) < \infty$ for all open balls $B$.

We say that $\mu$ is \emph{doubling} if there exists a constant $c_\mu > 0$ such that $\mu( 2B ) \leq c_\mu \mu( B )$ for every open ball $B \subset Z$ where $2B$ is a ball with the same center but twice the radius. If the inequality holds for all balls up to radius $r_0$, we say that $\mu$ is doubling \emph{up to scale $r_0$}.

A metric measure space $Z$ supports a \emph{weak $(1,p)$-Poincaré inequality},  for some $1\leq p<\infty$, if there exist constants $C>0,\lambda\geq 1$ such that 
\begin{equation}\label{eq:PI:inequality}
    \aint{B} |u-u_B|\,d\mu\leq C\diam (B)\left(\aint{\lambda B} \rho^p\,d\mu  \right)^{1/p}
\end{equation}
for every ball $B\subset Z$, every integrable function $u \colon \lambda B \to \R$ and every $L^p$-integrable $p$-weak upper gradient $\rho \colon \lambda B \rightarrow [0,\infty]$ of $u$. Here we define the \emph{integral average} of $u$ over a measurable set $F$ as $$u_F \coloneqq\aint{F} u \,d\mu \coloneqq \frac{1}{ \mu(F) } \int_F u \,d\mu$$  whenever $u$ is integrable on $F$ and $0 < \mu(F) < \infty$. If the inequality \eqref{eq:PI:inequality} holds for all balls $B$ whose radius is at most $r_0$, we say that $Z$ supports a weak $(1,p)$-Poincaré inequality \emph{up to scale $r_0$}. A metric measure space is a \emph{$p$-PI space} if $Z$ is a complete metric measure space with $\mu$ doubling and supporting a weak $(1,p)$-Poincaré inequality.

A measurable set $\boz\subset Z$ is called an $\rr$-valued \emph{$W^{1, p}$-extension set} for some $1\leq p\leq\fz$ if for every Sobolev function $u\in W^{1, p}(\boz; \rr)$,  there exists a function $E(u)\in W^{1, p}(Z; \rr)$ such that $E(u)\big|_\boz\equiv u$ and the inequality
\begin{equation}
    \label{eq:boundedoperator}
    \|E(u)\|_{W^{1, p}(Z; \rr)}\leq C\|u\|_{W^{1, p}(\boz; \rr)}
\end{equation}
holds with a constant $C > 0$ independent of $u$. We call 
$$E \colon W^{1,p}( \boz; \rr ) \rightarrow W^{1,p}( Z; \rr )$$ 
an $\R$-valued \emph{$W^{1,p}$-extension operator}. The smallest constant $C > 0$ for which \eqref{eq:boundedoperator} holds is denoted by $\|E\|$ and called the \emph{operator norm} of $E$. To simplify notation we typically omit the $\R$. Note that we do not require the operator $E$ to be linear.

For us, the notation $W^{1,p}( \Omega; \mathbb{R} )$ refers to the Lebesgue equivalence classes of Newton--Sobolev functions introduced by Shanmugalingam in \cite{Shanmugalingam} for real-valued mappings, based on earlier works by Koskela--MacManus \cite{Ko:Mac:98} and Heinonen--Koskela \cite{Hei:Ko:98}. It is now understood that in every complete metric measure space, the Newton--Sobolev approach is equivalent to the approaches by Cheeger \cite{Cheeger,Shanmugalingam} (when $1<p<\infty$) or the minimal relaxed slope arising from the relaxation of the slope of Lipschitz functions \cite{Amb:Gig:Sav:13,EB:20:energy} (when $1\leq p<\infty $); see \cite{Amb:Gig:Sav:13} for further reading. Moreover, when $\boz$ is an open set in the Euclidean space $\rn$, the definition coincides with the classical definition based on integration by parts formulation \cite{Shanmugalingam}. See \cite{Cre:Evs:21} and references therein for related Banach-valued results. 

On the Euclidean space $\mathbb{R}^n$, Calder\'on and Stein proved that Lipschitz domains are $W^{1, p}$-extension domains for arbitrary $1\leq p\leq\fz$, see \cite{Calderon, Stein:book}. Later, Jones \cite{Jones} generalized this result to the class of the so-called $(\varepsilon, \delta)$-domains (containing the class of uniform domains), strictly containing the class of Lipschitz domains. Similar results have now been proved in the metric measure space setting, for example, in \cite{S2007,HKT2008:B,Bjo:Sha:07,Bj:Bj:Lah:21}. Based on the results in \cite{K1998JFA,S2005, Ko:Ra:Zh:planar, Ko:Ra:Zh:W11}, a geometric characterization of planar simply connected $W^{1, p}$-extension domains is well-established. However, basic questions about the structure of more general $W^{1,p}$-extension domains remain open. A standard example of a domain $\Omega \subset \mathbb{R}^2$ that is not a $W^{1,p}$-extension domain is the slit disk: the unit disk minus a radial segment. This and other examples are discussed in \Cref{Section_LAST}. When $\Omega$ is not necessarily open but its complement has zero measure, there are known necessary and sufficient geometric conditions for $\Omega$ to be a $W^{1,p}$-extension set \cite{Bj:Bj:Lah:21}. The results in \cite{Bj:Bj:Lah:21} partially motivated us to work with measurable sets in place of domains.

The main subject of this paper is $W^{1,p}( \Omega; \mathbb{V} )$-extension sets, where $W^{1,p}( \Omega; \mathbb{V} )$ refers to the Lebesgue equivalence classes of $\mathbb{V}$-valued Newton--Sobolev functions, see \Cref{subsec_sob} for the precise definition. Now, given a Banach space $\mathbb{V}$, one can define the $\mathbb V$-valued extension operator as in \eqref{eq:boundedoperator} simply by replacing $\rr$ with $\mathbb{V}$. We are only considering the cases $1 < p<\infty$, the limiting cases $p=1$ and $p = \infty$ being outside the scope of this work. As far as the authors are aware, this project is the first time the Banach-valued Sobolev extension sets are considered in print.

The following question was the starting point of this manuscript.
\begin{question}\label{ques:startingpoint}
Let $1 <  p < \infty$ and $Z$ be a $p$-PI space. Are $W^{1,p}$-extension sets $\Omega \subset Z$ also $\mathbb{V}$-valued $W^{1,p}$-extension sets for all Banach spaces $\mathbb{V}$?

More generally, does there exist a Banach space $\mathbb{V}_0$ such that every $\mathbb{V}_0$-valued $W^{1,p}$-extension set is a $\mathbb{V}$-valued $W^{1,p}$-extension set for all Banach spaces $\mathbb{V}$?
\end{question}
Many of the known real-valued extension results generalize to the Banach-valued setting with relative ease when the real-valued assumption is replaced by the corresponding Banach-valued one. The main difficulty in \Cref{ques:startingpoint} arises in connecting the extensions properties for different Banach spaces: For example, if $\Omega$ is a real-valued $W^{1,p}$-extension set, an obvious approach to proving the existence of a, e.g. $\ell^{\infty}$-valued $W^{1,p}$-extension operator is to extend componentwise. However, without knowing that the extension operator commutes with projections, it is not obvious why the extension obtained in this way is even an $L^{p}( Z; \ell^{\infty} )$-function. Even if $L^{p}( Z; \ell^{\infty} )$-boundedness of the extension could be deduced (as one can do in many important cases, cf. \Cref{thm:W1p} below), the Sobolev regularity of the extension is more difficult to deduce. We were able to make this idea work only in some special cases, cf. \Cref{sec:partialanswer} below.

While the authors find \Cref{ques:startingpoint} itself interesting,  understanding Banach-valued extension sets has some applications. For example, by the Kuratowski embedding theorem, every separable metric space $X$ can be isometrically embedded into $\ell^{\infty}$. Then if every $u \in W^{1,p}( \Omega; X )$ is the restriction of some $h \in W^{1,p}( Z; \ell^{\infty} )$, this allows one to establish the existence of minimizers for several types of energy minimization problems, at least when $Z$ is a $p$-PI space. This is of interest, for example, when studying minimal surfaces using the so-called $Q$-valued mappings, see \cite{Q-valued}. Moreover, under suitable geometric assumptions on $Z$ and $X$, knowing that every $u \in W^{1,p}( \Omega; X )$ satisfies $u = h|_{\Omega}$ for some $h \in W^{1,p}( Z; X )$ is closely related to the density of Lipschitz functions in $W^{1,p}( \Omega; X )$, cf. \cite{La:Sch:05,Haj:Sch:14} and Theorems \ref{thm:M=W} and \ref{thm:V-extension:PI} below. We believe that understanding Banach-valued extension sets sheds some light to these metric-valued problems.

\subsection{Extending a given function}
Before stating our first result, we say that a measurable set $\Omega \subset Z$ satisfies the \emph{measure-density condition (with constant $c_{\Omega}>0$}) if
\begin{equation}\label{eq:measuredensitcondition}
    \mu( B(x,r) \cap \Omega ) \geq c_{\Omega}\mu( B(x,r) ) \quad\text{for every $(x,r) \in \Omega \times (0,1]$.}
\end{equation}
Our first result concerns recognizing the functions $u \in W^{1,p}( \Omega; \mathbb{V} )$ for which there exists some extension $h \in W^{1,p}( Z; \mathbb{V} )$.
\begin{theorem}\label{thm:extensionresults}
Let $1 < p< \infty$, $Z = (Z,d,\mu)$ be a $p$-PI space, $\Omega \subset Z$ be a measurable set satisfying the measure-density condition  and $\mathbb V$ be a Banach space. Then, given $u \in L^{p}( \Omega; \mathbb{V} )$, there exists $h \in W^{1,p}( Z; \mathbb{V} )$ with $u = h|_{\Omega}$ if and only if there exists some $s = s(u) > 0$ such that
\begin{equation}\label{eq:ext_condition}
    x
    \to
    u^{\sharp}_{s}( x )
    \coloneqq
    \sup_{ 0 < t < s }
        \frac{1}{t}
        \aint{ \Omega \cap B(x,t) }
        \aint{ \Omega \cap B(x,t) }
            | u(y)-u(z) |
        \,d\mu(y) \,d\mu(z)
    \in
    L^{p}( \Omega ).
\end{equation}
\end{theorem}
In the real-valued case and $s = \infty$, \Cref{thm:extensionresults} was considered in \cite[Theorem 1.2]{S2007} which partially motivated our local formulation of the aforementioned theorem; see also \cite{HKT2008:B,HKT2008}. The main novelty of \Cref{thm:extensionresults} is that the radius $s = s(u)$ is allowed to be arbitrarily small and depend on $u$, and this is a new result even in the real-valued case.

Note that if $\mu$ is doubling, \eqref{eq:measuredensitcondition} implies that the restriction of $\mu$ to $\Omega$ is a doubling measure up to scale one, so necessarily $\mu( \overline{\Omega} \setminus \Omega ) = 0$ by the Lebesgue differentiation theorem. The assumption \eqref{eq:measuredensitcondition} holds whenever $\Omega$ is an $( \varepsilon, \delta )$-domain in a $p$-PI space (see, e.g., \cite{Bjo:Sha:07}). This fact is especially of interest as bounded domains in PI spaces can be approximated from inside by $( \varepsilon, \delta )$-domains, cf. \cite{Raj:21}. Furthermore, if there are constants $C > 0$, $Q \geq 1$ for which
\begin{equation}\label{eq:Ahlfors}
    C^{-1} r^{Q} \leq \mu( B(x,r) ) \leq C r^Q
    \quad\text{for all $( x, r ) \in Z \times (0, 1]$},
\end{equation}
condition \eqref{eq:measuredensitcondition} is necessary for a domain $\Omega \subset Z$ to be a $W^{1,p}$-extension domain \cite[Theorem 2]{HKT2008:B}. The slit disk example illustrates that the measure-density condition is not sufficient. Typical examples of $p$-PI spaces satisfying \eqref{eq:Ahlfors} include every Carnot group (e.g. the Heisenberg groups), finite-dimensional Banach spaces, or compact Riemannian manifolds, each of these examples endowed with an appropriate dimensional Hausdorff measure, see e.g. \cite{Hei:Ko:98,Cheeger,KZ2008} for further reading.

\subsection{Extension sets}
\begin{theorem}\label{thm:W1p}
Let $1 < p< \infty$, $Z = (Z,d,\mu)$ be a $p$-PI space, $\Omega \subset Z$ be a measurable set satisfying the measure-density condition and $\mathbb V$ be a Banach space. Then the following are equivalent:
\begin{enumerate}
    \item There exists a linear extension operator $$E_{\mathbb{V}} \colon W^{1,p}( \Omega; \mathbb{V} ) \rightarrow W^{1,p}( Z; \mathbb{V} ).$$
    \item For every $u \in W^{1,p}( \Omega; \mathbb{V} )$, there exists $h \in W^{1,p}( Z; \mathbb{V} )$ with $h|_{\Omega} = u$.
\end{enumerate}
\end{theorem}
Even without assuming the measure-density condition, a simple argument using the bounded inverse theorem shows that condition (2) implies the existence of an extension operator $E_{\mathbb{V}} \colon W^{1,p}( \Omega; \mathbb{V} ) \rightarrow W^{1,p}( Z; \mathbb{V} )$. However, the measure-density condition guarantees the \emph{linearity} of \emph{some} extension operator.

We obtain a linear extension operator in \Cref{thm:W1p} using an idea based on \cite{HKT2008:B}, where Haj\l{}asz, Koskela and Tuominen considered real-valued $W^{1,p}$-extension domains. They observed that if $\Omega \subset Z$ is a domain and $Z$ is a $p$-PI space satisfying \eqref{eq:Ahlfors}, then $\boz$ is a real-valued $W^{1,p}$-extension domain if and only if it satisfies the measure-density condition and $W^{1,p}( \boz )$ coincides with the Haj\l{}asz--Sobolev space $M^{1,p}( \boz )$ (see \cite{M=W} or \Cref{preliminaries} for the definitions). The authors proved this equivalence by first considering $M^{1,p}$-extension domains. We now state our related $\mathbb{V}$-valued result.
\begin{theorem}\label{th:Mexten}
Let $Z = (Z,d,\mu)$ be a metric measure space, with $\mu$ doubling, and $\Omega \subset Z$ a measurable set satisfying the measure-density condition. Then, for every $1 \leq p < \fz$, there exist a constant $C = C(c_\mu,c_\Omega,p) > 0$ and a family of linear extension operators, one for each Banach space $\mathbb{V}$, $E_{ \mathbb{V} } \colon M^{1,p}( \Omega; \mathbb{V} ) \rightarrow M^{1,p}( Z; \mathbb{V} )$ with
\begin{itemize}
    \item $\| E_{\mathbb{V} } \| \leq C$;
    \item $T \circ E_{\mathbb{V}} = E_{\mathbb{W}} \circ T$ for every continuous linear map $T \colon \mathbb{V} \rightarrow \mathbb{W}$,
\end{itemize}
where $\mathbb{V}$ and $\mathbb{W}$ are arbitrary Banach spaces.
\end{theorem}
Here and elsewhere the notation $C = C( a, b, c )$ means that the constant $C$ depends only on the parameters $a$, $b$ and $c$.

Having \Cref{th:Mexten} in mind, now the proof of "(2) $\Rightarrow$ (1)" in \Cref{thm:W1p} is quite straightforward. On the one hand, we always have a $4$-Lipschitz linear embedding $ M^{1,p}( \Omega; \mathbb{V} ) \subset W^{1,p}( \Omega; \mathbb{V} )$, see \Cref{lemma:M:in:W}. On the other hand, by the assumptions on $Z$ and since $1<p<\infty$, we have $M^{1,p}( Z; \mathbb{V} ) = W^{1,p}( Z; \mathbb{V} )$ as sets with bi-Lipschitz equivalent norms (see \cite{M=W,HKST2015} or \Cref{Section_Sob.ext.dom.}). Then the bounded inverse theorem leads to the equality $M^{1,p}( \Omega; \mathbb{V} ) = W^{1,p}( \Omega; \mathbb{V} )$ as sets and with bi-Lipschitz equivalent norms. This implies that the operator $E_{\mathbb{V}}$ in \Cref{th:Mexten} is a bounded linear operator between the corresponding $W^{1,p}$-spaces.

This argument also establishes the following result, thereby extending the corresponding real-valued result from \cite{HKT2008:B} to the Banach-valued setting.
\begin{corollary}\label{cor:measuredensity:Sob}
Let $1 < p< \infty$, $Z = (Z,d,\mu)$ be a $p$-PI space and $\Omega \subset Z$ be  a measurable set satisfying the measure-density condition. Then $\Omega$ is a $\mathbb{V}$-valued $W^{1,p}$-extension set if and only if $W^{1,p}( \Omega; \mathbb{V} ) = M^{1,p}( \Omega; \mathbb{V} )$ as sets.
\end{corollary}
Next, we provide a positive answer to the second part of \Cref{ques:startingpoint}. For this purpose, we consider the Banach space
$$c_0 \coloneqq \left\{ x=( x_{i} )_{ i = 1 }^{ \infty } \colon x_i \in \mathbb{R},\; \lim_{ i \to\infty } x_i = 0\right\},$$
endowed with the supremum norm $| x | = \sup_{ i \in \N } | x_i |$. Then $c_0$ is a closed linear subspace of $\ell^\infty$, the space of bounded sequences. We prove the following result.
\begin{theorem}\label{thm:M=W}
Let $1 < p < \infty$ and $Z$ be a metric measure space. Then the following are equivalent:
\begin{enumerate}
    \item $W^{1,p}( Z; c_0 ) = M^{1,p}( Z; c_0 )$ as sets;
    \item $W^{1,p}( Z; \ell^{\infty} ) = M^{1,p}( Z; \ell^{\infty} )$ as sets;
    \item $W^{1,p}( Z; \mathbb{V} ) = M^{1,p}( Z; \mathbb{V} )$ as sets for every Banach space $\mathbb{V}$.
\end{enumerate}
Moreover, under any one of these equivalent conditions, there exists a constant $C = C( Z ) > 0$ so that for every Banach space $\mathbb{V}$, we have
\begin{equation*}
    C^{-1} \| u \|_{ M^{1,p}( Z; \mathbb{V} ) }
    \leq
    \| u \|_{ W^{1,p}( Z; \mathbb{V} ) }
    \leq
    4 \| u \|_{ M^{1,p}( Z; \mathbb{V} ) }
    \quad\text{for every $u \in W^{1,p}( Z; \mathbb{V} )$.}
\end{equation*}
\end{theorem}
\Cref{thm:M=W} and \Cref{cor:measuredensity:Sob} yield the following consequence.
\begin{corollary}\label{cor:c_0:all}
Let $1 < p< \infty$, $Z = (Z,d,\mu)$ be a $p$-PI space and $\Omega \subset Z$ be  a measurable set  satisfying the measure-density condition. If $\Omega$ is a $c_0$-valued $W^{1,p}$-extension set, then $\Omega$ is a $\mathbb{V}$-valued $W^{1,p}$-extension set for all Banach spaces $\mathbb{V}$.
\end{corollary}
While \Cref{cor:c_0:all} provides a positive answer to the second part of \Cref{ques:startingpoint}, the first part remains open. The following results investigate the density of Lipschitz functions in $W^{1,p}(\Omega;\mathbb V)$, closely related to the properties of the embedding $M^{1,p}( \Omega; \mathbb{V} ) \subset W^{1,p}( \Omega; \mathbb{V} )$.
\subsection{Density of Lipschitz functions}
If the equality $W^{1,p}( Z ) = M^{1,p}( Z )$ holds as sets, Lipschitz functions with bounded support are norm-dense in $W^{1,p}( Z )$ as the norm-density is known in $M^{1,p}(Z)$, cf. \cite{M=W}. Given the norm-density, it is of interest to know how well any given $u \in W^{1,p}( Z; c_0 )$ can be approximated by Lipschitz functions with bounded support, or whether such approximations even exist.

We show the following theorem.
\begin{theorem}\label{thm:density:Lips}
Let $1 \leq p < \infty$ and $Z = ( Z, d, \mu)$ be a metric measure space for which Lipschitz functions with bounded support are dense in norm in $W^{1,p}( Z )$. Then Lipschitz functions with bounded support are dense in energy in $W^{1,p}( Z; \mathbb{V} )$ whenever $\mathbb{V}$ is a Banach space with the metric approximation property.
\end{theorem}
We refer the reader to \Cref{sec:M=W} for the definitions and the proof. However, we note that energy-density is weaker than norm density, see \Cref{sec:M=W}, and we do not know if norm-density holds in $W^{1,p}( Z; \mathbb{V} )$ under the stated assumptions. We do expect that the conclusion of the theorem holds under the weaker assumption of energy-density of Lipschitz functions with bounded support in $W^{1,p}(Z)$.

We also recall that many classical Banach spaces, such as all finite-dimensional Banach spaces, $\ell^{q}$ space for $1 \leq q \leq \infty$, $c_0$ and every Hilbert space, have the metric approximation property, while not all Banach spaces have the property. We do not know if the density in energy holds for all Banach spaces under the assumptions of \Cref{thm:density:Lips}.

Typically the Lipschitz energy-density for Banach-valued Sobolev functions is deduced from the equality $M^{1,p}( Z; \mathbb{V} ) = W^{1,p}( Z; \mathbb{V} )$ as sets, valid in $p$-PI spaces for $p > 1$, cf. \cite{M=W,HKST2001,Haj:Sch:14,HKST2015,Haj:Sch:14} and the references therein. In fact, under the equality $M^{1,p}( Z; \mathbb{V} ) = W^{1,p}( Z; \mathbb{V} )$, the density in energy improves to density in {norm}, e.g. when the measure of $Z$ is doubling, cf. \cite[Lemma 10.2.7]{HKST2015}. See also \cite{La:Sch:05,Haj:Sch:14}. \Cref{thm:density:Lips} implies that to answer the first part of \Cref{ques:startingpoint} positively, it is necessary and sufficient to prove that the $4$-Lipschitz embedding $M^{1,p}( \Omega; c_0 ) \subset W^{1,p}( \Omega; c_0 )$ has a closed image whenever $\Omega$ is a $W^{1,p}$-extension set, cf. \Cref{prop:coincidence} and \Cref{rem:necessaryandsufficient}.

There is now a wealth of literature which guarantees the norm-density of Lipschitz functions in $W^{1,p}( Z )$, e.g. under the completeness \cite{Cheeger,HKST2015,Amb:Col:DiMa:15,Iko:Pas:Sou:22,EB:Sou:21} and noncompleteness \cite{Lew:87,Kos:Zh:16,Ko:Ra:Zh:17} assumptions, respectively. In particular, when $Z$ is complete and $Z$ is metrically doubling \cite{Amb:Col:DiMa:15} or has finite Hausdorff dimension \cite{EB:Sou:21}, the norm-density holds. Based on the known results in the complete setting, we prove the following.
\begin{proposition}\label{prop:density}
Let $Z$ be a $p$-PI space and $\Omega \subset Z$ a measurable set with $\mu( \overline{\Omega} \setminus \Omega ) = 0$. Then, for any $1 \leq p < \infty$, the following properties hold. 
\begin{enumerate}
    \item If every $u \in W^{1,p}( \Omega )$ is the restriction of some $h \in W^{1,p}( \overline{\Omega} )$, then Lipschitz functions with bounded support are norm-dense in $W^{1,p}( \Omega )$.
    \item If $\Omega$ is $p$-path almost open in $\overline{\Omega}$ and Lipschitz functions with bounded support are energy-dense in $W^{1,p}( \Omega )$, then every $u \in W^{1,p}( \Omega )$ is the restriction of some $h \in W^{1,p}( \overline{\Omega} )$ with equal norm.
\end{enumerate}
\end{proposition}
We refer the reader to \Cref{Section_LAST} for the definitions and the proof. The interested reader can readily relax the $p$-PI assumption from \Cref{prop:density}, and we expect \Cref{prop:density} to be known to experts in the field. More interestingly for us, combining the proposition with \Cref{thm:density:Lips} yields the energy density of Banach-valued Lipschitz functions for many Banach-valued $W^{1,p}$-spaces of interest.

Getting back to the Sobolev extension problem, we have been able to answer \Cref{ques:startingpoint} positively under further geometric assumptions on $\Omega$. This is the content of the results below.

\subsection{Sufficient geometric conditions for extendability}\label{sec:partialanswer}
A consequence of \Cref{thm:extensionresults}, based on \eqref{eq:ext_condition}, is the following.
\begin{proposition}\label{thm:positive}
Let $1 < p< \infty$, $Z = (Z,d,\mu)$ be a $p$-PI space and $\Omega \subset Z$ be a measurable set satisfying the measure-density condition. If $\Omega$ supports a weak $(1,p)$-Poincaré inequality up to scale $r_0$, for some $0 < r_0$, then $\Omega$ is a $\mathbb{V}$-valued $W^{1,p}$-extension set for every Banach space $\mathbb{V}$.
\end{proposition}
When we say that $\Omega$ supports a weak $(1,p)$-Poincaré inequality up to scale $r_0$, we mean that when the distance and measure of $Z$ are restricted to $\Omega$, the corresponding triple $( \Omega, d_{\Omega}, \mu_{\Omega} )$ supports a weak $(1,p)$-Poincaré inequality up to scale $r_0$. \Cref{thm:positive} is an expected result, especially when $r_0 = \infty$, cf. \cite[Section 5]{Bjo:Sha:07}, but we provide a proof.

The proof of \Cref{thm:positive} relies on the local self-improvement phenomenon of the weak Poincaré inequality established in \cite[Theorem 5.1]{Bj:Bj:2019}, based on earlier work in \cite{Bj:Bj:18} and by Keith and Zhong in \cite{KZ2008}. The noncompleteness of $\Omega$ makes the self-improvement a subtle problem, cf. \cite{Kos:99,Bj:Bj:Lah:21}.

Observe that the assumptions of \Cref{thm:positive} hold for every $(\varepsilon,\delta)$-domain in a $p$-PI space: the measure-density condition follows from \cite[Lemma 4.2]{Bjo:Sha:07} and weak $( 1, p )$-Poincaré inequality from \cite[Theorem 4.4]{Bjo:Sha:07}. See also \cite[Theorem 3.2]{Bj:Bj:Lah:21} for a related local result. 

Recently, \cite{Bj:Bj:Lah:21} studied the special case of $\mu( Z \setminus \Omega ) = 0$ partially motivated by removable sets for Sobolev functions and related problems. Among many other results, the authors prove that under a further technical condition on $\Omega$ as in \Cref{prop:density} (2), a set $\Omega \subset \mathbb{R}^n$ whose complement has negligible measure is a $W^{1,p}$-extension set if and only if $\Omega$ supports a weak $(1,p)$-Poincaré inequality \cite[Theorem 1.1]{Bj:Bj:Lah:21}.  See \cite{EB:Go:21,EB:Go:21:borderline} for further (fractal-like) examples of $\Omega \subset \R^n$ satisfying the assumptions of \Cref{thm:positive}.

The techniques developed for the proof of \Cref{thm:positive} led us to prove the following folklore result, providing a positive answer to \Cref{ques:startingpoint} in case $\Omega \subset \mathbb{R}^n$ is a $W^{1,p}$-extension domain for $n\leq p<\infty$.
\begin{theorem}\label{thm:extension:topoincare}
If $\Omega\subset \mathbb{R}^n$ is a $W^{1,p}$-extension domain for some $1<n\leq p<\infty$, then $\Omega$ satisfies the measure-density condition and supports a weak $( 1, p )$-Poincaré inequality up to some scale $r = r( \Omega ) > 0$.
\end{theorem} 
We recall that every $W^{1,p}$-extension domain $\Omega \subset \mathbb{R}^n$ satisfies the measure-density condition due to \cite{HKT2008}. The fact that $\Omega$ supports a weak $(1,p)$-Poincaré inequality up to some scale follows, with some work, from \cite{K1998JFA}, combined with \cite[Theorem 5.1]{Bj:Bj:2019}. For completeness sake, we provide a proof in \Cref{section:Corollary}. It would be interesting to extend \Cref{thm:extension:topoincare} in two ways. First, to $p$-PI spaces satisfying \eqref{eq:Ahlfors} for some $1 < Q \leq p$. Second, relaxing, as much as possible, the domain assumption on $\Omega$ to measurability and $p$-path almost openness and to the measure-density condition.

When $p < n$ there are $W^{1,p}$-extension domains that do not support a weak $( 1, p )$-Poincaré inequality up to any scale $r_0$ (see \cite[Example 2.5]{Koskela}, \cite[Remark 5.2]{Bjo:Sha:07}, or \Cref{ex:3}). \Cref{ex:3} considers a domain $\Omega \subset \mathbb{R}^2$ which is a $W^{1,p}$-extension domain for every $1 < p < 2$ and Lipschitz functions are dense in $W^{1,p}( \Omega )$ for every $1 < p \leq 2$, but still $\Omega$ is not a $W^{1,2}$-extension domain. In particular, $\Omega$ cannot support a weak $(1,p)$-Poincaré inequality up to any scale for any $1 \leq p \leq 2$. Although for this example we cannot use  \Cref{thm:extension:topoincare} to conclude from \Cref{thm:positive} that the extension property still holds for Banach-valued functions, an explicit extension operator can be built, hence giving further support that the answer to \Cref{ques:startingpoint} is yes for the whole range $1<p<\infty$.

We also answer affirmatively \Cref{ques:startingpoint} for the case of planar Jordan $W^{1, p}$-extension domains. This follows by applying a key theorem from \cite{PPZ}. The authors observed that for all planar Jordan $W^{1,p}$-extension domains there exists an extension operator induced by a reflection map. We show that the reflection map allows us to also extend Banach-valued functions.
\begin{proposition}\label{prop:planarjordan}
Let $\boz\subset\rr^2$ be a Jordan $W^{1, p}$-extension domain for $1<p<\fz$. Then $\Omega$ is a $\mathbb V$-valued $W^{1, p}$-extension domain for every Banach space $\mathbb V$.
\end{proposition}

To conclude the introduction we outline the paper. In \Cref{preliminaries}, we state some preliminary results and fix some notation. We introduce the Sobolev spaces we use in \Cref{sec:Sob}. In \Cref{sec:M=W}, we establish Theorems \ref{thm:M=W} and \ref{thm:density:Lips}. In \Cref{sec:proof:main}, we construct the $M^{1,p}$-extension operator following the construction from \cite{HKT2008:B} and prove \Cref{th:Mexten}. In \Cref{Section_Sob.ext.dom.}, we prove \Cref{thm:extensionresults}. In \Cref{Section_LAST}, we investigate the connection of being a $W^{1,p}$-extension set to the Banach property of the so-called local Haj\l{}asz--Sobolev space and to the density of Lipschitz functions with bounded support, under the further assumptions of the measure-density condition and $p$-path almost openness, cf. \cite{Bj:Bj:2015,Bj:Bj:2019,Bj:Bj:2019:corrig,Bj:Bj:Lah:21}. We also consider some example domains in this section. Lastly, we prove \Cref{thm:positive} and \Cref{thm:extension:topoincare} in \Cref{section:Corollary} and \Cref{prop:planarjordan} in \Cref{section:LAST_LAST}, respectively.

\section{Preliminaries}\label{preliminaries}

As mentioned previously, in what follows, we use the notation $C,C'$ or $C(\cdot)$ for a computable positive constant depending only on the parameters listed in the parenthesis. The constant may differ between appearances, even within a chain of inequalities.

We say that a metric measure space $Z=( Z, d, \mu )$ is a \emph{doubling metric measure space} whenever $\mu$ is doubling.

Often we consider restrictions of the distance and measure to measurable subsets $\Omega \subset Z$. The distance is obtained simply by restriction, while the measure is obtained in the following way. For every $K \subset \Omega$, we denote
\begin{equation*}
    \mu_\Omega( K )
    \coloneqq
    \inf\left\{
        \mu( K' )
        \colon
        K \subset K' \text{ is Borel in $\Omega$}
    \right\}.
\end{equation*}
Recall that $K'$ is Borel in $\Omega$ if and only if $K'=B \cap \Omega$ for some Borel set $B \subset Z$, cf. \cite[Lemma 3.3.4]{HKST2015}. Since $B \cap \Omega$ is $\mu$-measurable whenever $B$ is Borel in $Z$, we therefore obtain
\begin{equation}\label{eq:considerations:measure}
    \mu_\Omega( K )
    =
    \inf\left\{
        \mu( K' )
        \colon
        K \cap \Omega \subset K' \text{ is Borel in $Z$}
    \right\}
    =
    \mu( K ).
\end{equation}
In particular, $\mu_\Omega$ is a Borel regular outer measure. These considerations imply, in particular, that $K \subset \Omega$ is $\mu_\Omega$-measurable if and only if it is $\mu$-measurable. Moreover, if $K$ is Borel measurable in $Z$, $K$ is Borel measurable in $\Omega$. Since $\mu_\Omega \equiv \mu$ for subsets of $\Omega$, we typically omit the subscript from $\mu_\Omega$.

We say that $N \subset Z$ is \emph{negligible} if $\mu( N ) = 0$. A property holds \emph{almost everywhere} in $F \subset Z$,  or for almost every $z \in F$, if there exists a negligible set $N \subset F$ such that the property holds for every point in $F \setminus N$.

We denote the open and closed balls of $Z$ by $$B(z,r) \coloneqq \left\{ y \in Z \colon d(z,y) < r \right\} \quad\text{and}\quad \overline{B}(z,r) \coloneqq \left\{ y \in Z \colon d(z,y) \leq r \right\},$$
respectively. We sometimes omit the $(z,r)$, and use the notation $\lambda B$ to denote the ball $B( z, \lambda r )$  for $0<\lambda<\fz$. For any set $F \subset Z$, we define the open $r$-neighbourhood as
\begin{equation*}
    B(F,r) = \bigcup_{z \in F} B(z,r)=\{ z\in Z:\, d(z,F)<r \},
\end{equation*}
where $d(z,F) \coloneqq \inf\left\{ d(z,x) \colon x \in F \right\}$ is the distance from a point $z\in Z$ to a subset $F \subset Z$. Similarly, if $E \subset Z$, then $d( E, F ) = \inf\left\{ d(z, F) \colon z \in E \right\}$ is the distance between the sets $E$ and $F$.

\subsection{Lipschitz maps}
Let $Z$ and $X$ be metric spaces. Given $L \geq 0$, a map $f \colon Z\to X$ is \emph{$L$-Lipschitz} if $$d( f(x), f(y) ) \leq Ld(x,y) \quad\text{for every $x,y\in Z$.}$$ We say that $f$ is \emph{$L$-bi-Lipschitz} if
$$L^{-1} d(x,y) \leq d( f(x), f(y) ) \leq Ld(x,y) \quad\text{for every $x,y\in Z$.}$$
A map $f \colon Z\to X$ is \emph{(bi-)Lipschitz} if it is $L$-(bi-)Lipschitz for some $L \geq 0$.

In case $X$ is a Banach space, we say that map $f \colon Z \to X$ has \emph{bounded support} if the set $\left\{ f \neq 0 \right\}$ is bounded.

\subsection{Functional analysis}
Given a Banach space $\mathbb{V}$, we typically denote its elements by $v$ and the norm by $|v|$. The \emph{dual of $\mathbb{V}$}, denoted by $\mathbb{V}^{*}$, refers to the collection of all continuous linear maps $w \colon \mathbb{V} \rightarrow \mathbb{R}$, where the norm is defined by setting 
$$|w| \coloneqq \sup_{ |v| \leq 1 } w(v).$$ 
We reserve the notation $\| \cdot \|$ for norms on function spaces, or for extension operators.

We recall here an important functional analytic tool, see \cite[Section 12]{Conway}.
\begin{theorem}[Bounded inverse theorem]\label{thm:bddinverse}
Let $\mathbb{V}_1$ and $\mathbb{V}_2$ be Banach spaces and $\iota \colon \mathbb{V}_1 \rightarrow \mathbb{V}_2$ be linear and continuous.
\begin{enumerate}
    \item If $\iota$ is surjective, then there exists a constant $c> 0$ such that for every $v_2\in \mathbb V_2$ there exists $v_1\in \mathbb V_1$ with $\iota(v_1)=v_2$ and $| \iota(v_1) | \geq c| v_1|$.
    \item If $\iota$ is injective and $\iota(\mathbb V_1)$ is a closed subspace of $\mathbb{V}_2$, then there exists a constant $c> 0$ such that $c | v | \leq | \iota(v) |$ for every $v \in \mathbb{V}_1$. In particular $\iota \colon \mathbb{V}_1 \to \iota( \mathbb{V}_1 )$ is bi-Lipschitz.
\end{enumerate} 
\end{theorem}

\subsection{Bochner integration}
Given a Banach space $\mathbb{V}$, we say that $u \colon Z \rightarrow \mathbb{V}$ is \emph{measurable} if for every open set $V \subset \mathbb{V}$, $u^{-1}( V )$ is $\mu$-measurable and there exists a negligible set $N \subset Z$ such that $u( Z \setminus N )$ is separable. We refer to \cite[Pettis measurability]{HKST2015} for equivalent definitions of measurability. In particular, a map $u \colon Z \rightarrow \mathbb{V}$ is measurable if and only if there exists a sequence of \emph{simple maps} $u_{i} = \sum_{ j = 1 }^{ N_{i} } c_{i,j} \chi_{ E_{i,j} }$ converging to $u$ pointwise outside a negligible set. Here each $E_{i,j} \subset Z$ is a measurable subset, $c_{i,j} \in \mathbb{V}$, and $\chi_{ E_{i,j} }$ is the \emph{characteristic function} of the set $E_{i,j}$.

The approximation by simple maps lends itself to a natural extension of the Lebesgue integration of real-valued maps. The approximation defines the so-called \emph{Bochner integration} for all measurable $u \colon Z \rightarrow \mathbb{V}$. Indeed, a measurable function $u \colon Z \to\mathbb V$ is Bochner integrable if and only if $|u| \in \mathcal{L}^{1}( Z )$ in the usual Lebesgue sense; see \cite[Section 3]{HKST2015}. We define $\mathcal{L}^{p}( Z; \mathbb{V} )$, for $1 \leq p \leq \infty$, as the collection of all measurable $u \colon Z \rightarrow \mathbb{V}$ for which $|u| \in \mathcal{L}^{p}( Z )$. Per usual, we identify $u_1, u_2 \in \mathcal{L}^{p}( Z; \mathbb{V} )$ if $u_1 = u_2$ almost everywhere, thereby obtaining a quotient space $L^{p}( Z; \mathbb{V} )$ equipped with the norm 
\[\|u\|_{L^p(Z;\mathbb V)}:=\lf(\int_Z|u|^pd\mu\r)^{\frac{1}{p}}.\]

We also say that a function $u:Z\to\mathbb V$ is {\em locally integrable}, denoted by $u \in \mathcal{L}^{1}_{\loc}( Z; \mathbb{V} )$, if $u$ is measurable and $Z$ can be covered by open sets in which the norm $|u|$ is integrable.

The following result shows how the image of measurable functions can be isometrically embedded into $\ell^{\infty}$. Recall that 
$$\ell^{\infty}\coloneqq\{(z_i)^{\infty}_{i=1}:\, z_i\in\R,\, \sup_{ i } |z_i|<\infty\}$$  endowed with the supremum norm $|z| \coloneqq \sup_{ i } |z_i|$. The simple lemma below is one of the key ideas in the proof of \Cref{thm:M=W}.
\begin{lemma}\label{lemm:representation}
Let $u \colon Z \rightarrow \mathbb{V}$ be measurable and $N \subset Z$ be negligible so that $u( Z \setminus N ) \subset \mathbb{V}$ is separable. Then there exists a separable closed subspace $\mathbb{V}_{0} \subset \mathbb{V}$ containing $u( Z \setminus N )$ and a linear $1$-Lipschitz mapping $\iota \colon \mathbb{V} \rightarrow \ell^{\infty}$ such that $\iota|_{ \mathbb{V}_0 }$ is an isometric embedding.
\end{lemma}
\begin{proof}
Given that $u( Z \setminus N )$ is separable, it contains a countable dense set $\{u(z_i)\}_{i\in\mathbb N}$ with $z_i\in Z \setminus N$. The set
\begin{align*}
    \mathbb{V}_{0}
    =
    \overline{
        \left\{
            \sum^{n}_{i=1} \lambda_i u(z_i) \colon
            \lambda_i \in \mathbb{Q},\, 
            n\in\mathbb N
            \right\}}
\end{align*}
is a separable and closed subspace of $\mathbb{V}$. In particular, there exists a dense subset $\{v_{i} \}_{ i\in\mathbb N}\subset \mathbb{V}_{0}$. By the Hahn--Banach theorem, for every $i$, there exists $w_{i} \in \mathbb{V}^{*}$ for which $w_{i}( v_i ) = | v_i |$ and $|w_{i}| = 1$. 

We define $\iota \colon \mathbb V\to\ell^{\infty}$ by setting $\iota(v) := ( w_{i}(v) )_{ i = 1 }^{ \infty }$ for every $v \in \mathbb{V}$. This defines a $1$-Lipschitz linear mapping. By the density of $\{ v_{i}\}_{ i \in\mathbb N }$ in $\mathbb{V}_{0}$, for every $\varepsilon > 0$ and every $v \in \mathbb{V}_{0}$, there exists some $j \in \mathbb{N}$ for which $|v|-\varepsilon\leq |w_{j}( v )| \leq |v|$. This implies that $|\iota(v)| =|v|$ for every $v\in \mathbb{V}_{0}$. In other words, the restriction of $\iota$ to $\mathbb{V}_{0}$ is a linear isometry.
\end{proof}

It is sometimes convenient to consider zero extensions of measurable functions defined on measurable sets, the extension being measurable, cf. \eqref{eq:considerations:measure}. To this end, given a measurable $g \colon F \rightarrow \mathbb{V}$, we denote by $\widehat{g} \colon Z \rightarrow \mathbb{V}$ the \emph{zero extension} of $g$ by setting $g = \widehat{g}|_{F}$ and zero elsewhere.

For every locally integrable function $u \colon Z \rightarrow \R$, we define the {\em Hardy-Littlewood maximal operator} $\mathcal M(u):Z\to \R$ by setting 
\begin{equation}
    \label{eq:maximalfunction}
    \mathcal{M}( u )(x)
    \coloneqq
    \sup_{r>0}
    \aint{ B(x,r) }
        u(y)
    \,d\mu(y)
    \quad\text{for every $x \in Z$}.
\end{equation}
 We also use the following \emph{restricted maximal function} for a given $R>0$:
\begin{equation}
    \label{eq:maximalfunction:restr}
    \mathcal{M}_{R}( u )(x)
    \coloneqq
    \sup_{R>r>0}
    \aint{ B(x,r) }
        u(y)
    \,d\mu(y)
    \quad\text{for every $x \in Z$}.
\end{equation}
In the case $u\colon Z\rightarrow \mathbb V$ we set  $\mathcal{M}( u ) = \mathcal{M}( |u| )$ for the norm $|u| \colon Z \rightarrow \mathbb{R}$ of $u$. 
The following lemma recalls basic properties of the maximal operator defined above.
\begin{lemma}[Theorem 3.5.6 \cite{HKST2015}]\label{lemm:maximalfunction:bounded}
Let $Z$ be a doubling metric measure space and $\mathbb{V}$ a Banach space. For each $1 < p < \infty$ and $R>0$, the mappings $$\mathcal{M} \colon L^{p}( Z; \mathbb{V} ) \to L^{p}( Z ) \quad\text{and}\quad  \mathcal{M}_{R} \colon L^{p}( Z; \mathbb{V} ) \rightarrow L^{p}( Z )$$ are bounded operators, with operator norms bounded from above by a constant $C( p, c_\mu ) > 0$.
\end{lemma}

\subsection{Lebesgue differentiation theorem}
In this section, we consider a doubling metric measure space $Z$. In fact, the results below work in any metric measure space $Z$ for which
\begin{equation}\label{eq:pointwisedoubling}
    \limsup_{ r \rightarrow 0^+ }
    \frac{ \mu( B(z,2r) ) }{ \mu(B(z,r)) }
    <
    \infty
    \quad\text{for $\mu$-almost every $z \in Z$.}
\end{equation}
We recall the Lebesgue differentiation theorem: If $u \in L^{p}( Z; \mathbb{V} )$ for some Banach space $\mathbb{V}$, then
\begin{equation}\label{eq:Lebesguedensity}
    \lim_{ r \rightarrow 0^{+} }
    \aint{ \overline{B}(z,r) } | u(y) - u(z) |^{p} \,d \mu(y)
    =
    0
    \quad\text{for almost every $z \in Z$}.
\end{equation}

We refer the reader to \cite[Section 3.4]{HKST2015} for a proof of  \eqref{eq:Lebesguedensity}. Any $z \in Z$ for which \eqref{eq:Lebesguedensity} holds is called a \emph{Lebesgue point} of $u$. The property \eqref{eq:pointwisedoubling} implies that we may replace the closed balls by open balls in \eqref{eq:Lebesguedensity}, while the triangle inequality yields
\begin{equation}\label{eq:densitypoint}
    u(z)
    =
    \lim_{ r \rightarrow 0^{+} }
    u_{ \overline{B}(z,r) }
    =
    \lim_{ r \rightarrow 0^{+} }
    u_{ B(z,r) }
    \quad\text{at Lebesgue points of $u$}.
\end{equation}
In particular,
\begin{equation*}
    |u|(z)
    =
    \lim_{ R \rightarrow 0^+ }
    \mathcal{M}_R( |u| )
    \quad\text{at Lebesgue points of $u$}.
\end{equation*}

\section{Sobolev spaces}\label{sec:Sob}

\subsection[Hajlasz-Sobolev spaces]{Haj\l{}asz--Sobolev spaces}
Let $1 \leq p \leq \infty$ and $Z$ be a metric measure space. Given a map $u \colon Z \rightarrow \mathbb{V}$, we say that a function $g \colon Z \rightarrow \left[0, \infty\right]$ is a \emph{Haj\l{}asz upper gradient} of $u$ if $g$ is measurable and there exists $N \subset Z$ with $\mu( N ) = 0$ such that
\begin{equation}\label{eq:Hineq}
    |u(x)-u(y)|
    \leq d(x,y)\lf(g(x)+g(y)\r) \quad \text{for every $x,y \in Z \setminus N$}.
\end{equation}
The class of all Haj\l{}asz upper gradients of $u$ is denoted by $\mathcal{D}(u)$, and we set $\mathcal{D}_{p}(u) \coloneqq \mathcal{D}(u) \cap \mathcal{L}^{p}(Z)$.

Observe that if $u_1, u_{2} \colon Z\to\mathbb V$, $g \in \mathcal{D}( u_1 )$ and $u_1 = u_2$ almost everywhere, then $g \in \mathcal{D}( u_2 )$. In particular, if $u \in L^{p}( Z; \mathbb{V} )$, the collection $\mathcal{D}( u )$ can be unambigously defined as those Haj\l{}asz gradients of \emph{some} (an arbitrary) representative of $u$. Similarly, if $g \in \mathcal{D}( u )$ all representatives of the Lebesgue equivalence classes of $g$ are elements of $\mathcal{D}(u)$, justifying the abuse of notation $\mathcal{D}_{p}(u) \subset L^{p}( Z )$.
\begin{definition}\label{de:Mspace}
Let $1\leq p\leq\fz$ and $\mathbb V$ be a Banach space. The Haj\l{}asz-Sobolev space $M^{1,p}(Z;\mathbb V)$ is defined by setting
\[  M^{1, p}(Z; \mathbb{V})\coloneqq\lf\{u\in L^p(Z; \mathbb{V}) \colon \mathcal D_p(u)\neq\emptyset\r\},  \]
equipped with the norm 
\[  \|u\|_{M^{1, p}(Z; \mathbb{V})}\coloneqq\|u\|_{L^p(Z; \mathbb{V})}+\inf_{g\in\mathcal D_p(u)}\|g\|_{L^p(Z)}.\]
In case $\mathbb{V}=\R$, the short-hand $M^{1,p}(Z) = M^{1,p}( Z; \R )$ is used.
\end{definition}

\begin{lemma}\label{lemm:basic}
Let $u \colon  Z\to\mathbb V$ be measurable and $g \in \mathcal{D}(u)$.
\begin{enumerate}
    \item If $\mathbb W$ is a Banach space and $\phi \colon \mathbb{V} \rightarrow \mathbb W$ is $L$-Lipschitz, then $Lg \in \mathcal{D}( \phi \circ u )$. Moreover, if $\phi(0) = 0$, then $\| \phi \circ u \|_{ L^{p}( Z;\mathbb W ) } \leq L \| u \|_{ L^{p}( Z; \mathbb{V} ) }$ for every $1 \leq p \leq \infty$.
    \item If $w \in \mathbb{V}^{*}$, then $|w|g \in \mathcal{D}( w(u) )$.
    \item If $\phi \colon Z \rightarrow \mathbb{R}$ is $L$-Lipschitz and $K = Z \setminus \left\{ \phi = 0 \right\}$, then $\phi u$ has a Haj\l{}asz upper gradient \[\widetilde{g} = ( L |u| + \| \phi \|_{\infty}g )\chi_{ K }.\]
\end{enumerate}
\end{lemma}
\begin{proof}
We show $(1)$ first. For every $x, y \in Z$, $$| \phi( u(x) ) - \phi( u(y) ) | \leq L | u(x) - u(y) | < \infty.$$ Also, there exists a negligible set $N \subset Z$ so that $$| u(x)-u(y) | \leq d(x,y)( g(x) + g(y) )$$ for every $x, y \in Z \setminus N$. Hence we have  $Lg \in \mathcal{D}( \phi \circ u )$. If in addition $\phi(0)=0$, we have $|\phi\circ u(x)|\leq L|u(x)|$ for every $x\in Z$, and the norm inequality $\| \phi \circ u \|_{ L^{p}( Z;\mathbb W ) } \leq L \| u \|_{ L^{p}( Z; \mathbb{V} ) }$ follows for every $1 \leq p \leq \infty$. Observe that (2) is a special case of (1), given that each $w \in \mathbb{V}^{*}$ is linear and $|w|$-Lipschitz.

The statement (3) is basically \cite[Lemma 5.20]{Haj:Kin:98}, where the authors formulated the result for real-valued mappings. First, in case $x\in K,y\notin K$, one has
$$|\phi(x)u(x)|\leq d(x,y) L |u(x)|\leq d(x,y) \widetilde g(x). $$
 The case $y\in K,x\notin K$ is analogous. If $x,y\in K$, then
\begin{align*}
|\phi(x)u(x) - \phi(y)u(y)|&\leq d(x,y)\left( (g(x)+g(y))|\phi(x)| + L|u(y)|\right) \\
    &\leq   d(x,y)\left(
        \|\phi\|_{\infty}g(x) +L|u(x)|+\|\phi\|_{\infty}g(y) +L|u(y)|
    \right)\\
    &\leq d(x,y)(\widetilde g(x)+\widetilde g(y)).
 \end{align*}
If $x, y \notin K$ the result is clear as well.
\end{proof}

\begin{lemma}\label{lemm:projection}
Suppose that $u \colon Z \rightarrow \mathbb{V}$ is measurable and $g  \colon Z \rightarrow \left[0,\infty\right]$ is measurable.

Then the following are equivalent:
\begin{enumerate}
    \item  $g$ is a Haj\l{}asz upper gradient of $u$;
    \item  $g$ is a Haj\l{}asz upper gradient of $w(u)$ for every $|w| \leq 1$, $w \in \mathbb{V}^{*}$;
\end{enumerate}
Let $E \subset Z$ be negligible so that $u( Z \setminus E ) \subset \mathbb{V}_{0}$ for some separable closed subspace $\mathbb{V}_{0}$ of $\mathbb{V}$, and let $\iota \colon \mathbb{V} \rightarrow \ell^{\infty}$ be a $1$-Lipschitz linear map with $\iota|_{ \mathbb{V}_0 }$ being an isometric embedding.

Then (1) and (2) are also equivalent to the following:
\begin{enumerate}
    \item[(3)]  $g$ is a Haj\l{}asz upper gradient of the mapping $U \coloneqq \iota \circ u \colon Z \rightarrow \ell^{\infty}$;
    \item[(4)]  $g$ is a Haj\l{}asz upper gradient of $u_{i}$, for each $i \in \mathbb{N}$, where $U = ( u_{i} )_{ i = 1 }^{ \infty }$;
    \item[(5)]  $g$ is a Haj\l{}asz upper gradient of $U^{N} \coloneqq ( u_{i} )_{ i = 1 }^{N} \colon Z\to \ell^\infty$ for every $N \in \mathbb{N}$. 
\end{enumerate}
\end{lemma}
\begin{proof}
The existence of the negligible set $E \subset Z$, the subspace $\mathbb{V}_0$ and the mapping $\iota$ is established in \Cref{lemm:representation}, so we lose no generality in proving the claim as follows: we first prove $(1) \Rightarrow (2) \Rightarrow (4) \Rightarrow (5)$ and then we establish $(5) \Rightarrow (3)$ and $(1) \Leftrightarrow (3)$.

The implication from $(1)$ to $(2)$ is established by \Cref{lemm:basic}. Next, if  $e_{j} \colon \ell^{\infty} \rightarrow \mathbb R$ is defined by $( x_i )_{ i = 1 }^{ \infty } \mapsto x_j $, then $e_{j} \circ \iota \in \mathbb{V}^{*}$ with $| e_{j} \circ \iota | \leq 1$. In particular, $(4)$ follows from $(2)$.

We denote $\pi_{N} \colon \ell^{\infty} \rightarrow \ell^{\infty}$ taking $( x_i )_{ i = 1 }^{ \infty }$ to $( x_i )_{ i = 1 }^{N}$. We also let $u_i=e_i\circ \iota\circ u$. Assuming $(4)$, we find negligible sets $N_i \subset Z$ such that $$| u_i(x) - u_i(y) | \leq d( x, y )( g(x)+  g(y) )$$ for every $i $ and every $x, y \in Z \setminus N_i$. Hence for every $N\in\mathbb N$ and $x, y \in Z\setminus (\bigcup^{N}_{i=1}N_i)$, $$| \pi_{N}(U(x)) - \pi_{N}(U(y)) | = \sup_{1\leq i\leq N}| u_{i}(x) - u_{i}(y) |\leq d(x,y)(g(x)+g(y)),$$
which implies (5).

By noting the equality $$\sup_{N} | \pi_{N}(U)(x) - \pi_{N}(U)(y) | = |U(x)-U(y)| \quad\text{for every $x,y \in Z$},$$ we see that $(5)$ implies $(3)$. The equivalence between $(1)$ and $(3)$ holds since $| U(x) - U(y) | = |u(x) - u(y)|$ for every $x, y \in Z \setminus E$ outside the given negligible set $E$.
\end{proof}

\subsection{Path integrals}
We consider an arbitrary metric space $Z$. A \emph{path} is a continuous function $\gamma \colon \left[a, b\right] \rightarrow Z$ whose domain $\left[a,b\right]$ is denoted by $I_{\gamma}$. A path $\gamma$ is \emph{rectifiable} if
\begin{equation}
    \label{eq:rectifiable}
    \ell( \gamma )
    =
    \sup
    \sum_{ i = 1 }^{ N } d( \gamma(t_i), \gamma( t_{i+1} ) )
    <
    \infty,
\end{equation}
where the supremum is taken over all finite partitions $t_1 = a,$ $t_{i} \leq t_{i+1}$, $t_{N+1} = b$.

A path $\gamma$ is rectifiable if and only if the {\em length function} $h(t) = \ell( \gamma|_{[ a, t ]} )$, $t \in I_{\gamma}$ is non-decreasing and has bounded variation; recall that $h$ is continuous when $\ell( \gamma ) < \infty$. We let $d s_\gamma$ denote the total variation measure of $h$.

Given a Borel function $\rho \colon Z \to [0, \infty]$, the \emph{path integral} of $\rho$ over $\gamma$ is the following number:
\begin{equation}
    \label{eq:pathintegral}
    \int_{\gamma} \rho \,ds
    \coloneqq
    \begin{cases}
    \infty &\quad\text{if $\gamma$ is not rectifiable, } \vspace{0.2cm}
    \\
   \displaystyle{ \int_{ I_{\gamma} }
        \rho( \gamma(t) )
    \,ds_{\gamma}(t)}
    &\quad\text{otherwise}.
\end{cases}    
\end{equation}

For a rectifiable path, we denote
\begin{equation}
    \label{eq:metricspeed:pointwise}
    \| \gamma' \|(t)
    \coloneqq
    \lim_{ t \neq s \rightarrow t }
        \frac{ d( \gamma(t), \gamma(s) ) }{ |t-s| },
\end{equation}
whenever the limit exists. The limit exists for almost every $t \in \left[a,b\right]$ for every rectifiable path. The mapping $t \mapsto \| \gamma' \|(t)$ is called the \emph{metric speed} of $\gamma$.

Every rectifiable path $\gamma\colon [a, b]\to Z$ has a \emph{unit speed reparametrization} $\widetilde{\gamma} \colon [0, l(\gamma)]\to Z$ with $\| \widetilde{\gamma}'\|(t)  \equiv 1$ almost everywhere. Namely, $\widetilde{\gamma}$ is the unique continuous map for which $\widetilde{\gamma}( h(t) ) = \gamma( t )$ for the length function $h$ above, see \cite[Eq. (5.1.6)]{HKST2015}. Since $\widetilde{\gamma} \circ h = \gamma$ and $h$ is (right-)continuous, nondecreasing and of bounded variation, the equality
\begin{equation*}
    \int_{\gamma} \rho \,ds
    =
    \int_{\widetilde{\gamma}} \rho \,ds
\end{equation*}
holds for every Borel $\rho \colon Z \rightarrow \left[0,\infty\right]$, see, e.g., \cite[2.5.18.]{Fed:69}.

A rectifiable path $\gamma$ is \emph{not absolutely continuous} if there exists a negligible set $I \subset I_\gamma$ and a Borel set $B \subset \gamma( I )$ with $\int_{\gamma} \chi_{B} \,ds > 0$. If no such $I$ and $B$ exist, we say that $\gamma$ is \emph{absolutely continuous}. For an absolutely continuous path $\gamma \colon \left[a,b\right] \rightarrow Z$, we have
\begin{equation*}
    \label{eq:pathintegral:ABS}
    \int_{\gamma} \rho \,ds
    =
    \int_{ a }^{ b }
        \rho( \gamma(s) ) \| \gamma' \|(s)
    \,d\mathcal{L}^{1}(s),
    \quad\text{see \cite{Dud:07},}
\end{equation*}
where $\mathcal{L}^{1}$ is the Lebesgue measure on $[a,b]$.

\subsection{Negligible path families}
We consider an arbitrary metric measure space $Z$ here. Any collection of paths $\gamma \colon I_{\gamma} \rightarrow Z$ is referred to as a \emph{path family}, typically denoted by $\Gamma$.
\begin{definition}
Let $1 \leq p \leq \infty$. A path family $\Gamma$ is \emph{$p$-negligible} if there exists an $L^{p}$-integrable Borel $\rho \colon Z \rightarrow [0, \infty]$ for which
\begin{equation*}
    \Gamma \subset \left\{ \gamma \colon \int_{ \gamma } \rho \,ds = \infty \right\}.
\end{equation*}
A property holds for \emph{$p$-almost every} $\gamma$ if the family of paths where the property fails is $p$-negligible.
\end{definition}
Observe that the collection of nonrectifiable paths is $p$-negligible, \emph{by definition}. Notice that any path family containing a constant path is \emph{not} $p$-negligible. Our notion of "$p$-negligible" coincides with the usual one defined using the so-called \emph{$p$-modulus}; for the case $1 \leq p < \infty$ we refer the reader to \cite[Lemma 5.2.8]{HKST2015} and for the $p = \infty$ case to \cite[Lemma 5.7]{Durand}.

\subsection{Weak upper gradients}\label{subsec_sob:grad}
We consider an arbitrary metric measure space $Z$ in this section. Let $u \colon Z \rightarrow \mathbb{V}$ and $1 \leq p \leq \infty$. We say that a mapping $\rho \colon Z \rightarrow [0,\infty]$ is a \emph{$p$-weak upper gradient} of $u$ if $\rho$ is Borel measurable and for $p$-almost every $\gamma \colon \left[a,b\right] \rightarrow Z$,
\begin{equation}\label{eq:boundaryinequality}
    |u( \gamma(a) ) - u( \gamma(b) )|
    \leq
    \int_{\gamma} \rho \,ds.
\end{equation}
If, in addition, $\rho$ is $L^{p}$-integrable, we denote $\rho \in \mathcal{D}_{N,p}( u )$.

\begin{lemma}\label{lemm:gradient:absolutecontinuity}
Let $u \colon Z \rightarrow \mathbb{V}$ and $\rho\in \mathcal{D}_{N,p}( u )$. There exists a $p$-negligible family
\begin{equation*}
    \Gamma_{u,\rho}
    \supset
    \left\{ \gamma \colon \ell( \gamma ) + \int_{\gamma} \rho \, ds = \infty \right\}
\end{equation*}
such that every $\gamma \not\in \Gamma_{u,\rho}$ satisfies the following: $u \circ \gamma$ is continuous and rectifiable and every Borel $a \colon I_\gamma \rightarrow [0, \infty]$ satisfies
\begin{equation}\label{eq:upper:abs}
    \int_{ I_{\gamma} }
        a(t)
    \,ds_{ u \circ \gamma }(t)
    \leq 
    \int_{ I_{\gamma} }
        \rho( \gamma(t) ) a(t)
    \,ds_{ \gamma }(t).
\end{equation}
\end{lemma}
\begin{proof}
Let $\Gamma$ denote the collection of all paths with at least one of the following properties: $\gamma$ is not rectifiable, $\int_{\gamma} \rho \, ds = \infty$, or there exists a compact subinterval $I \subset I_{\gamma}$ such that the path $\gamma|_{I}$ fails \eqref{eq:boundaryinequality}. Such a family is $p$-negligible as a union of three such families.

Fix now a rectifiable path $\gamma \not\in \Gamma$; assume also that $\gamma$ is not constant. Let $s, s' \in I_{\gamma}$ with $s \leq s'$. Then
\begin{equation*}
    | u( \gamma(s) ) - u( \gamma(s') ) |
    \leq
    \int_{\gamma|_{ \left[s,s'\right] } } \rho \,ds
    =
    \int_{ [s,s'] } \rho( \gamma(s) ) \,d s_{\gamma}
    \leq
    \int_{ \gamma } \rho \,d s
    <
    \infty.
\end{equation*}
Since $s$ and $s'$ were arbitrary, $s_\gamma( \left\{s\right\} ) = 0$ for every $s \in I_\gamma$ and $\rho \circ \gamma$ is $L^{1}$-integrable with respect to the measure $s_\gamma$, the continuity of $u \circ \gamma$ follows. In fact, the (sub)additivity of integration implies
\begin{equation}\label{eq:dominatedvariation}
    s_{ u \circ \gamma }( I ) = \ell( u \circ \gamma|_{I} ) \leq \int_{ \gamma|_{I} } \rho \,d s
    \quad\text{for every closed interval $I \subset I_\gamma$}.
\end{equation}
Hence $u \circ \gamma$ is rectifiable. As \eqref{eq:dominatedvariation} implies the corresponding inequality for all open intervals, the inequality holds for all open sets $I \subset I_\gamma$. The outer regularity of the measures $s_{ u \circ \gamma }$ and $s_\gamma$ implies the corresponding inequality for all Borel sets $B \subset I_\gamma$. Then \eqref{eq:upper:abs} follows by first establishing the inequality for simple Borel functions, while the general claim follows from monotone convergence and standard approximation results.
\end{proof}

\begin{lemma}\label{lemm:representative}
Let $u \colon Z \rightarrow \mathbb{V}$ be a map and $\rho_{0} \in \mathcal{D}_{N,p}( u )$. If $\rho_{1} \colon Z \rightarrow \left[0, \infty\right]$ is Borel measurable and satisfies $\rho_{1} = \rho_{0}$ almost everywhere, then $\rho_{1} \in \mathcal{D}_{N,p}(u)$. 
\end{lemma}
\begin{proof}
The set $E = \left\{ \rho_{0} \neq \rho_{1} \right\}$ is Borel and satisfies $\mu( E ) = 0$. Consider the $p$-negligible family
\begin{equation*}
    \Gamma_{1}
    =
    \left\{ \gamma \colon \int_{ \gamma } \infty \cdot \chi_{E} \, ds = \infty \right\}
    \cup
    \left\{ \gamma \colon \int_{ \gamma } \rho_0 \, ds = \infty \right\}
    \cup
    \Gamma_{u,\rho_0},
\end{equation*}
where $\Gamma_{u, \rho_0}$ is any $p$-negligible family satisfying the conclusions of \Cref{lemm:gradient:absolutecontinuity}.

By construction, $g_0 \circ \gamma = g_1 \circ \gamma$ $s_{\gamma}$-almost everywhere, so it follows from \eqref{eq:upper:abs} that every $\gamma \not\in \Gamma_1$ satisfies
\begin{equation*}
    \ell( u \circ \gamma )
    \leq
    \int_{ \gamma } \rho_0 \,ds
    =
    \int_{ \gamma } \rho_1 \,ds.
\end{equation*}
This implies $\rho_{1} \in \mathcal{D}_{N,p}( u )$.
\end{proof}

\begin{proposition}\label{lemm:minimal}
If $u \colon Z \rightarrow \mathbb{V}$ is measurable and $\mathcal{D}_{N,p}( u ) \neq \emptyset$, then there exists $\rho \in \mathcal{D}_{N,p}( u )$ such that for every $\widetilde\rho \in \mathcal{D}_{N,p}(u)$, the inequality $\rho(x) \leq\widetilde \rho(x)$ holds for almost every $x\in Z$.
\end{proposition}
\begin{proof}
For the case $1 \leq p < \infty$, see Theorem 6.3.20 \cite{HKST2015}. The $p = \infty$ appears in \cite{M2013} in the real-valued case; the Banach-valued case requires only minor modifications.
\end{proof}
For a function $\rho\in\mathcal D_{N, p}(u)$ from \Cref{lemm:minimal}, its $L^{p}(Z)$-equivalence class is denoted by $\rho_{u}$. We refer to $\rho_{u}$ as a \emph{minimal $p$-weak upper gradient} of $u$. Note that every Borel representative of $\rho_{u}$ satisfies the conclusion of \Cref{lemm:minimal}, as we see from \Cref{lemm:representative}.

Next, we recall the following variants of Lemmas \ref{lemm:basic} and \ref{lemm:projection}.
\begin{lemma}\label{lemm:basic:new}
Let $u \in L^{p}( Z; \mathbb{V} )$, for $1 \leq p \leq \infty$, with $\rho \in \mathcal{D}_{N,p}(u)$.
\begin{enumerate}
    \item If $\mathbb{W}$ is a Banach space and $\phi \colon \mathbb{V} \rightarrow\mathbb{W}$ is $L$-Lipschitz, then $L\rho \in \mathcal{D}_{N,p}( \phi \circ u )$. Moreover, if $\phi(0) = 0$, then $$\| \phi \circ u \|_{ L^{p}( Z; \mathbb{W} ) } \leq L \| u \|_{ L^{p}( Z; \mathbb{V} ) }.$$
    \item If $w \in \mathbb{V}^{*}$, then $w(u) \in L^{p}( Z )$ and $|w|\rho \in \mathcal{D}_{N,p}( w(u) )$.
    \item If $\phi \colon Z \rightarrow \mathbb{R}$ is $L$-Lipschitz, and $K = Z \setminus \left\{ \phi = 0 \right\}$ is bounded, then $$\widetilde{\rho} = ( L |u| + \| \phi \|_{L^{\infty}(Z)}\rho )\chi_{ K } \in \mathcal{D}_{N,p}( \phi u ).$$
    \item If $U \subset Z$ is open, then $\rho_{ u|_{ U } } = (\rho_{u})|_{ U }$ almost everywhere in $U$.
\end{enumerate}
\end{lemma}
\begin{proof}
The claims were already known for $1\leq p<\fz$. Indeed, claim (1) follows from \cite[Theorem 7.1.20]{HKST2015}, with (2) being a special case of (1). Claim (3) follows from \cite[Proposition 6.3.28]{HKST2015}. Claim (4) follows from \cite[Corollary 3.9]{Wil:12}. Similar argument also works for the case $p = \infty$.
\end{proof}

\begin{lemma}\label{lemm:projection:new}
Suppose that $u \colon Z \rightarrow \mathbb{V}$ is measurable, $\rho  \colon Z \rightarrow \left[0,\infty\right]$ is Borel measurable and $1 \leq p \leq \infty$. Then the following are equivalent:
\begin{enumerate}
    \item $u$ has a representative $\widetilde{u}$ with $\rho \in \mathcal{D}_{N,p}( \widetilde{u} )$;
    \item $w(u)$ has a representative $u_{w}$ with $\rho \in \mathcal{D}_{N,p}( u_w )$ for every $w \in \mathbb{V}^{*}$, $|w| \leq 1$.
\end{enumerate}
\end{lemma}
\begin{proof}

The implication "$(1) \Rightarrow (2)$" already follows from \Cref{lemm:basic:new}, since $w( \widetilde{u} )$ is a representative of $w(u)$ for every $w \in \mathbb{V}^{*}$. The direction "$(2) \Rightarrow (1)"$ is known when $1 \leq p < \infty$, but the same proof also works for $p = \infty$; the key ingredients are \Cref{lemm:gradient:absolutecontinuity} and \Cref{lemm:minimal}. We refer the interested reader to the proof of \cite[Theorem 7.1.20, implication (iii) $\Rightarrow$ (i)]{HKST2015}.
\end{proof}

\begin{lemma}\label{lemma:Hajlasz:into:New}
Let $1 \leq p \leq \infty$, $u \in \mathcal{L}^{1}_{\loc}( Z; \mathbb{V} )$, and $g \in \mathcal{D}_{p}( u )$. Then there exists a representative $\widetilde{u}$ of $u$ and a representative $\widetilde{g}$ of $g$ such that $4\widetilde{g} \in \mathcal{D}_{N,p}( \widetilde{u} )$.
\end{lemma}
\begin{proof}
Theorem 1.3 in \cite{Ji:Sha:Ya:2015} proves the following: Suppose that $u, g \in \mathcal{L}^{1}_{\loc}( Z )$ and $g \in \mathcal{D}_p( u )$. Then $u$ and $g$ have representatives $\widetilde{u}$ and $\widetilde{g}$, respectively, for which $4\widetilde{g} \in \mathcal{D}_{N,p}( \widetilde{u} )$.

We use this as follows: For every $w \in \mathbb{V}^{*}$ with $|w| \leq 1$, note that $g \in \mathcal{D}_{p}( w(u) )$ for the real-valued $w(u)$. By applying \cite[Theorem 1.3]{Ji:Sha:Ya:2015}, we conclude that there exist representatives $u_w$ of $w(u)$ and $\widetilde{g}_{w}$ of $g$ with $4\widetilde{g}_{w} \in \mathcal{D}_{N,p}( u_w )$.

We fix now an arbitrary Borel representative $\widetilde{g}$ of $g$. It follows from \Cref{lemm:representative} that $4\widetilde{g} \in \mathcal{D}_{N,p}( u_w )$ for every $w \in \mathbb{V}^{*}$ with $|w| \leq 1$. The implication "(2) $\Rightarrow$ (1)" in \Cref{lemm:projection:new} shows that $4\widetilde{g} \in \mathcal{D}_{N,p}( \widetilde{u} )$ for some representative $\widetilde{u}$ of $u$.
\end{proof}

\subsection[Newton-Sobolev spaces]{Newton--Sobolev spaces}\label{subsec_sob}
We continue working with an arbitrary metric measure space. Following \cite{Shanmugalingam,HKST2001}, we say $\widetilde{u} \in \widetilde{N}^{1,p}( Z; \mathbb{V} )$ if $\widetilde{u} \in \mathcal{L}^p(Z; \mathbb V)$ and $\mathcal{D}_{N,p}( \widetilde{u} ) \neq \emptyset$. We denote
\begin{equation*}
    \| \widetilde{u} \|_{ \widetilde{N}^{1,p}(Z; \mathbb{V}) }
    \coloneqq
    \| \widetilde{u} \|_{L^p(Z; \mathbb{V})}
    +
    \inf_{ \rho \in \mathcal{D}_{N,p}( \widetilde{u} ) }\| \rho \|_{L^p(Z)}.
\end{equation*}
Two elements $\widetilde{u}_{1}, \widetilde{u}_{2} \in \widetilde{N}^{1,p}( Z; \mathbb{V} )$ are identified if $\widetilde{u}_{1} - \widetilde{u}_{2} = 0$ and $\rho_{ \widetilde{u}_1 - \widetilde{u}_{2} } = 0$ almost everywhere. With this convention, we obtain from $\widetilde{N}^{1,p}( Z; \mathbb{V} )$ a Banach space $N^{1,p}( Z; \mathbb{V} )$ with the norm
\begin{equation*}
    \| u \|_{ N^{1,p}(Z; \mathbb{V} ) }
    \coloneqq
    \| \widetilde{u} \|_{ \widetilde{N}^{1,p}(Z; \mathbb{V}) },
\end{equation*}
where $\widetilde{u}$ is any $\widetilde{N}^{1,p}( Z; \mathbb{V} )$-representative of ${u} \in N^{1,p}(Z; \mathbb{V} ) $; see \cite[Section 7.1]{HKST2015}.

We consider the following variant of the $N^{1,p}$-space.
\begin{definition}\label{def:Dir}
Let $1 \leq p \leq \infty$ and $\mathbb{V}$ be a Banach space. We denote $u \in W^{1,p}( Z; \mathbb{V} )$ if $u \in L^{p}( Z; \mathbb{V} )$ and there exists $\widetilde{u} \in \widetilde{N}^{1,p}( Z; \mathbb{V} )$ for which $\left\{ u \neq \widetilde{u} \right\}$ is negligible. The space $W^{1,p}( Z; \mathbb{V} )$ is equipped with the norm
\begin{equation*}
    \|u\|_{W^{1, p}(Z; \mathbb{V})}
    \coloneqq
    \| \widetilde{u} \|_{ \widetilde{N}^{1,p}(Z; \mathbb{V}) }.
\end{equation*}
In case $\mathbb{V} = \mathbb{R}$, the letter $\mathbb{R}$ is omitted.
\end{definition}
For $1 \leq p <\infty$, Proposition 6.3.22 in \cite{HKST2015} states that given two mappings $\widetilde{u}_{1}, \widetilde{u}_2 \in \widetilde{N}^{1,p}( Z; \mathbb{V} )$ agreeing almost everywhere, we have $\rho_{ \widetilde{u}_{1} - \widetilde{u}_{2} } = 0$ almost everywhere. For $p = \infty$, Lemma 5.13 in \cite{Durand} implies the corresponding result. The equality $\rho_{ \widetilde{u}_{1} } = \rho_{ \widetilde{u}_{2} }$ almost everywhere follows from the observations $\rho_{ \widetilde{u}_{1} - \widetilde{u}_{2} } + \rho_{ \widetilde{u}_{2} } \in \mathcal{D}_{N,p}( \widetilde{u}_1 )$ and $\rho_{ \widetilde{u}_{1} - \widetilde{u}_{2} } + \rho_{ \widetilde{u}_{1} } \in \mathcal{D}_{N,p}( \widetilde{u}_2 )$. Hence every $u \in W^{1,p}( Z; \mathbb{V} )$ defines a unique equivalence class in $N^{1,p}( Z; \mathbb{V} )$ with equal norm. This argument also implies that we may define $\rho_u \in L^{p}(Z)$ as the Lebesgue equivalence class of a minimal $p$-weak upper gradient of some (all) Newton--Sobolev representative of $u \in W^{1,p}(Z; \mathbb{V})$.

We conclude from \Cref{lemma:Hajlasz:into:New} the following result of independent interest valid for every metric measure space $Z$.
\begin{proposition}\label{lemma:M:in:W}
Let $Z$ be a metric measure space, $1 \leq p \leq \infty$ and $\mathbb{V}$ be a Banach space. Then there exists a linear $4$-Lipschitz mapping $\iota \colon M^{1,p}( Z; \mathbb{V} ) \rightarrow W^{1,p}( Z; \mathbb{V} )$ where $u \mapsto u$. In particular, every $u \in M^{1,p}( Z; \mathbb{V} )$ induces a unique equivalence class $\widetilde{u} \in N^{1,p}( Z; \mathbb{V} )$.
\end{proposition}

We next prove a slightly technical statement involving $u \in \widetilde{N}^{1,p}( Z; \mathbb{V} )$ and the $\ell^{\infty}$-space. We recall the existence of a negligible set $E \subset Z$ such that $u( Z \setminus E )$ is contained in a closed and separable subspace of $\mathbb{V}$. Recalling \Cref{lemm:representation}, there exists a linear $1$-Lipschitz mapping
\begin{equation*}
    \iota \colon \mathbb{V} \rightarrow \ell^{\infty}
\end{equation*}
for which $\iota|_{ \mathbb{V}_0 }$ is a linear isometric embedding. We also denote $U \coloneqq \iota \circ u$.
\begin{proposition}\label{lemma:usefulfact}
Let $u$ and $U$ be as above. Then $U \in \widetilde{N}^{1,p}( Z; \ell^{\infty} )$. Moreover, $\rho_u(x) = \rho_{U}(x)$ for almost every $x \in Z$, $|u(x)-u(y)| = |U(x)-U(y)|$ for every $x, y \in Z \setminus E$, and $|u(x)| = |U(x)|$ for every $x \in Z \setminus E$.
\end{proposition}
\begin{proof}
It follows directly from the definitions that $|u(x)-u(y)| = |U(x)-U(y)|$ for every $x, y \in Z \setminus E$ and $|u(x)| = |U(x)|$ for every $x \in Z \setminus E$.

Since $\iota$ is $1$-Lipschitz, we have $\rho_{u} \in \mathcal{D}_{N,p}( U )$, thereby implying $U \in \widetilde{N}^{1,p}( Z; \ell^{\infty} )$ and $\rho_{U} \leq \rho_{u}$ almost everywhere.

It remains to show $\rho_{U} \in \mathcal{D}_{N,p}( u )$. We consider the $p$-negligible path family
\begin{equation*}
    \Gamma_{0}
    =
    \Gamma_{U, \rho_{U}}
    \cup
    \Gamma_{u,\rho_u}
    \cup
    \left\{ \gamma \colon \int_{\gamma} \infty \cdot \chi_{E} \,ds = \infty \right\},
\end{equation*}
where $\Gamma_{U, \rho_{U}}$ and $\Gamma_{u,\rho_u}$ are obtained from \Cref{lemm:gradient:absolutecontinuity}. The key property for us is the following: For every constant speed $\gamma \not\in \Gamma_{0}$, the paths $u \circ \gamma$ and $U \circ \gamma$ are absolutely continuous, and
\begin{gather}\label{eq:upper:abs:later}
    \| ( u \circ \gamma )' \|(t)
    \leq 
    \rho_u( \gamma(t) ) \| \gamma' \|(t)
    \quad\text{for almost every $t$,}
\end{gather}
with the corresponding inequality holding for the pair $( U, \rho_{U} )$ instead of $(u, \rho_{u} )$; here we simply apply \eqref{eq:upper:abs}. Since $\int_{ \gamma } \chi_{E} \,ds = 0$ and $\gamma$ has constant speed, we conclude that $\gamma^{-1}( E )$ has negligible Lebesgue measure. The inequality \eqref{eq:upper:abs:later} (resp. the corresponding inequality for $U$) implies that $u \circ \gamma$ (resp. $U \circ \gamma$) has zero length in the complement of $\mathbb{V}_{0}$ (resp. $\iota( \mathbb{V}_{0} )$). The continuity of $u \circ \gamma$ (resp. $U \circ \gamma$) actually yields that the image of $u \circ \gamma$ lies in $\mathbb{V}_{0}$ (resp. the image of $U \circ \gamma$ in $\iota( \mathbb{V}_{0} )$).

Given that $\iota$ is an isometry in $\mathbb{V}_{0}$, we have $| U \circ \gamma(s) - U \circ \gamma(s') | =  | u \circ \gamma(s) - u \circ \gamma(s') |$ for every $s,s'\in I_{\gamma}$. In particular, the metric speeds of the absolutely continuous paths $u \circ \gamma$ and $U \circ \gamma$ coincide. Hence we may replace in the right-hand side of \eqref{eq:upper:abs:later} the $\rho_u( \gamma(t) ) \| \gamma' \|(t)$ by $\rho_U( \gamma(t) ) \| \gamma' \|(t)$ for almost every $t$. This implies $\rho_{U} \in \mathcal{D}_{N,p}( u )$. We have now established the claim.
\end{proof}

\subsection[Local Hajlasz-Sobolev spaces]{Local Haj\l{}asz-Sobolev spaces}\label{sec:locHaj}
In preparation for the proof of \Cref{thm:extensionresults}, we consider in this section a doubling metric measure space $Z$ and a measurable set $\Omega \subset Z$ satisfying the measure-density condition \eqref{eq:measuredensitcondition}. We define an intermediate space between the Haj\l{}asz--Sobolev and the Newton--Sobolev spaces.
\begin{definition}
Let $Z$ be a metric measure space and $\mathbb V$ be a Banach space. Fix a function $u \colon Z \to  \mathbb{V}$. We say that $g \colon Z \rightarrow \left[0,\infty\right]$ is a \emph{ Haj\l{}asz gradient} of $u$ up to scale $r > 0$ if $g|_{ B( z, r) } \in \mathcal{D}( u|_{ B(z,r) } )$ for every $z \in Z$. The notation $\mathcal{D}( u, r )$ refers to the class of all Haj\l{}asz gradients of $u$ up to scale $r$ and we set
\begin{equation*}
    \mathcal{D}_{\loc}( u ) \coloneqq \bigcup_{ r > 0 } \mathcal{D}( u, r ).
\end{equation*}
We also denote
\begin{equation*}
    \mathcal{D}_{p}( u, r ) \coloneqq \mathcal{D}( u, r ) \cap L^{p}( Z )
    \quad\text{and}\quad
    \mathcal{D}_{\loc,p}( u ) \coloneqq \mathcal{D}_{\loc}( u ) \cap L^{p}( Z ).
\end{equation*}
We say that $u \in m^{1,p}( Z; \mathbb{V} )$ if $u \in L^{p}( Z; \mathbb{V} )$ and $\mathcal{D}_{\loc,p}(u) \neq \emptyset$, and we define the norm
\begin{equation*}
    \| u \|_{ m^{1,p}( Z; \mathbb{V} ) }
    \coloneqq
    \| u \|_{ L^{p}( Z; \mathbb{V} ) }
    +
    \inf_{ g \in  \mathcal{D}_{\loc, p}( u ) }
    \| g \|_{ L^{p}( Z ) }.
\end{equation*}
In case $\mathbb{V} = \mathbb{R}$, we denote $m^{1,p}( Z ) \coloneqq m^{1,p}( Z; \mathbb{R} )$.
\end{definition}
\begin{remark}\label{rem:Banach}
Even when $\Omega \subset \mathbb{R}^n$ is an open or measurable set, it is not clear when $m^{1,p}( \Omega; \mathbb{V} )$ is a Banach space. In fact, we show examples in \Cref{sec:Banach:extension} illustrating that the Banach property of $m^{1,p}( \Omega )$ is closely connected to $W^{1,p}$-extension sets.

We also note that, in general, $\mathcal{D}_\loc(u)$ is different from the collection of all $g \colon Z \to [0,\infty]$ for which for all $z \in Z$ there exists \emph{some} $r(z) > 0$ with $g|_{ B(z,r) } \in \mathcal{D}( u|_{B(z,r)} )$. An illustrative example is given by the slit disk. As far as the authors are aware, for real-valued mappings this slight variant of the space $m^{1,p}$ was first defined in \cite{Ji:Sha:Ya:2015}, the variant containing the space $m^{1, p}$ defined here.
\end{remark}
Observe that since $\mathcal{D}_{p}( u ) \subset \mathcal{D}_{\loc,p}( u )$, the space $M^{1,p}( Z; \mathbb{V} )$ $1$-Lipschitz linearly embedds into $m^{1,p}( Z; \mathbb{V} )$. Moreover, the following fact shows that there is a $4$-Lipschitz embedding from $m^{1,p}( Z;\mathbb V )$ into $W^{1,p}( Z; \mathbb V )$.

\begin{lemma}\label{lemma:m:into:N}
Let $Z$ be a metric measure space, $\mathbb V$ be a Banach space and $1 < p < \infty$. Then $m^{1,p}( Z; \mathbb{V} )$ $4$-Lipschitz linearly embedds into $W^{1,p}( Z; \mathbb{V} )$. More precisely, given any $u \in m^{1,p}( Z; \mathbb{V} )$ and  $g \in \mathcal{D}_{\loc, p}( u )$, there exists a representative $\widetilde{u}$ of $u$ and a Borel representative $\widetilde{g}$ of $g$ such that $4\widetilde{g} \in \mathcal{D}_{N,p}( \widetilde{u} )$.
\end{lemma}
\begin{proof}
The following is a variant of the proof of \Cref{lemma:Hajlasz:into:New}. First, \cite[Theorem 3.2]{Ji:Sha:Ya:2015} proves the following: Suppose that $u, g \in \mathcal{L}^{1}_{\loc}( Z )$ and $g \in \mathcal{D}_{\loc, p}( u )$. Then $u$ and $g$ have representatives $\widetilde{u}$ and $\widetilde{g}$, respectively, for which $4\widetilde{g} \in \mathcal{D}_{N,p}( \widetilde{u} )$.

For every $w \in \mathbb{V}^{*}$ with $|w| \leq 1$, note that $g \in \mathcal{D}_{\loc, p}( w(u) )$ for the real-valued function $w(u)$. By applying \cite[Theorem 3.2]{Ji:Sha:Ya:2015}, we conclude that there exist representatives $\widetilde{u}_w$ of $w(u)$ and $\widetilde{g}_{w}$ of $g$ with $4\widetilde{g}_{w} \in \mathcal{D}_{N,p}( \widetilde{u}_w )$.

We fix now an arbitrary Borel representative $\widetilde{g}$ of $g$. It follows from \Cref{lemm:representative} that $4\widetilde{g} \in \mathcal{D}_{N,p}( u_w )$ for every $w \in \mathbb{V}^{*}$ with $|w| \leq 1$. The implication "(2) $\Rightarrow$ (1)" in \Cref{lemm:projection:new} shows that $4\widetilde{g} \in \mathcal{D}_{N,p}( \widetilde{u} )$ for some representative $\widetilde{u}$ of $u$, and the proof of the lemma is finished.
\end{proof}

For the next result we recall the following definition:
\begin{equation*}
    u^{\sharp}_{s}(z)
    =
    \sup_{ 0 < t < s }
    \frac{1}{t}
    \aint{ \Omega \cap B(z,t) }
    \aint{ \Omega \cap B(z,t) }
        | u(x)-u(y) |
    \,d\mu(x)\,d\mu(y)
\end{equation*}
Here $u \colon \Omega \rightarrow \mathbb{V}$ is locally integrable and $\Omega \subset Z$ is a measurable subset satisfying the measure-density condition.

\begin{proposition}\label{lemm:localHaj}
Whenever $Z$ is a doubling metric measure space, $1 < p < \infty$, $\Omega\subset Z$ satisfies the measure-density condition, and $\mathbb V$ is a Banach space, we have
$$ m^{1,p}(\Omega;\mathbb V)=\{u\in L^p(\Omega;\mathbb V):\,u^{\sharp}_{s} \in L^{p}( \Omega)\;\text{for some}\; s>0\}. $$
Moreover, there exists a constant $C=C(c_\mu,c_\Omega,p) \geq  1$, such that
\begin{equation*}
    \frac{ 1 }{ C}
    \left(
        \| u \|_{ L^{p}(\Omega; \mathbb{V}) }
        +
        \inf_{ s > 0 }\| u^{\sharp}_{s}(z) \|_{ L^{p}( \Omega ) }
    \right)
    \leq
    \| u \|_{ m^{1,p}( \Omega; \mathbb{V} ) }
    \leq
    C
    \left(
        \| u \|_{ L^{p}(\Omega; \mathbb{V}) }
        +
        \inf_{ s > 0 }\| u^{\sharp}_{s}(z) \|_{ L^{p}( \Omega ) }
    \right).
\end{equation*}
In fact, whenever $0 < s \leq 1/2$, there exists another constant $C=C(c_\mu,c_\Omega) > 0$ such that for every locally integrable $u \colon \Omega \rightarrow \mathbb{V}$, $C u^{\sharp}_{2s} \in \mathcal{D}( u, s )$.

\end{proposition}
\begin{proof}
This is a standard argument but we include the details for the convenience of the reader.

Assume first that $u\in m^{1,p}(\Omega;\mathbb V)$. Obviously by definition we have $u\in L^p(\Omega;\mathbb V)$. Let us prove that $u^{\sharp}_{s} \in L^{p}( \Omega)$. Suppose that $g \in \mathcal{D}_{\loc,p}( u )$  and take $s\in(0,1]$ so that $g|_{ B( z, s)\cap \Omega } \in \mathcal{D}( u|_{ B(z,s)\cap\Omega } )$ for every $z \in \Omega$. Then, at every $z \in \Omega$,
\begin{align*}
    \sup_{ 0 < t < s }
    \frac{1}{t}
    \aint{ \Omega \cap B(z,t) }
    \aint{ \Omega \cap B(z,t) }
        | u(x)-u(y) |
    \,d\mu(x)\,d\mu(y)
    &\leq
    \sup_{ 0 < t < s }
    4\aint{ \Omega \cap B(z,t) } g(x) \,d\mu(x)\\
    &\leq  C(c_\Omega)  \mathcal{M}( \widehat{g} )(z).
\end{align*}
In the last line we merely use the measure-density condition of $\Omega$. Recall that  $\widehat{g}$ denotes the zero extension of $g$ to $Z$ and $\mathcal{M}$ the maximal function on $Z$. Since the maximal function $\mathcal{M}$ is a bounded operator by \Cref{lemm:maximalfunction:bounded}, we conclude from these inequalities that $u^{\sharp}_{s} \in L^{p}( \Omega )$. In fact, these estimates depend only on  the measure-density constant $c_\Omega$, the doubling constant $c_\mu$ and $p$, so there is $C = C(c_\Omega, c_\mu, p )\geq 1$ so that
\begin{equation}\label{eq:lowerbound:locHaj}
    \frac{ 1 }{ C }
    \left(
        \| u \|_{ L^{p}(\Omega; \mathbb{V}) }
        +
        \inf_{ s > 0 }\| u^{\sharp}_{s}(z) \|_{ L^{p}( \Omega ) }
    \right)
    \leq
    \| u \|_{ m^{1,p}( \Omega; \mathbb{V} ) }.
\end{equation}
Let us prove now that if $u^{\sharp}_{s} \in L^{p}( \Omega)$ for some $s > 0$ and $u \in L^{p}( \Omega; \mathbb{V} )$, then $u \in m^{1,p}( \Omega; \mathbb{V} )$. We may always assume $1 \geq s > 0$. To this end, it is enough to verify the following claim: For every locally integrable $u:\Omega\to\mathbb V$ there exists a constant $C=C(c_\mu,c_\Omega)>0$ such that for each $1/2 \geq s > 0$,
\begin{equation}\label{eq:localHaj:localgradient}
    C u^{\sharp}_{2s}|_{ B(z,s) \cap \Omega }
    \in
    \mathcal{D}( u|_{ B(z,s) \cap \Omega } )
    \quad\text{for every $z \in \Omega$}.
\end{equation}
Observe that this immediately implies
$$  \| u \|_{ m^{1,p}( \Omega; \mathbb{V} ) }
    \leq \| u \|_{ L^{p}(\Omega; \mathbb{V}) }
        +
        C\inf_{ s > 0 }\| u^{\sharp}_{s}(z) \|_{ L^{p}( \Omega ) }. $$
For the purpose of proving \eqref{eq:localHaj:localgradient}, let $N$ denote the collection of non-Lebesgue points of $u$, i.e., those for which 
$$N=\left\{ x\in\Omega:\, \limsup_{ r \rightarrow 0^{+} }
    \aint{ B(x,r) \cap \Omega }
        | u(w) - u_{ B(x,r) \cap \Omega } |
    \,d\mu(w)
    >
    0  \right\}.$$
The measure-density condition and the local integrability of $u$ imply $\mu( N ) = 0$.

Fix $z \in \Omega$. We claim that there exists $C=C(c_\mu,c_\Omega)>0$ for which
\begin{equation*}
    | u(x) - u(y) | \leq d(x,y)( C u^{\sharp}_{2s}(x) + C u^{\sharp}_{2s}(y) )
    \quad\text{for every $x, y \in \Omega \cap B(z,s) \setminus N$}.
\end{equation*}

Fix $1 \geq s' > 0$ and consider $B_i := B( x, 2^{-i} s' ) \cap \Omega$ for $x \in \Omega \cap B(z,s') \setminus N$ and each $i = 0,1,2,\dots$. Since $x \in \Omega \setminus N$,
\begin{equation}\label{eq:chainargument:1}
    | u(x) - u_{ B_{0} } |
    \leq
    \sum_{ i = 0 }^{ \infty }
    | u_{ B_i } - u_{ B_{i+1} } |
    \leq
    \sum_{ i = 0 }^{ \infty }
    \aint{ B_{i+1}  }
        | u(w) - u_{ B_{i}  } |
    \,d\mu(w).
\end{equation}
As $B_{i+1} \subset B_{i}$ and $\mu( B_{i} ) \leq C(c_\mu,c_\Omega)\mu( B_{i+1} )$, we obtain
\begin{align}\label{eq:chainargument:2}
    \sum_{ i = 0 }^{ \infty }
    \aint{ B_{i+1} }
        | u(w) - u_{ B_{i} } |
    \,d\mu(w)
    &\leq
   C(c_\mu,c_\Omega)
    \sum_{ i = 0 }^{ \infty }
    \aint{ B_{i} }
        | u(w) - u_{ B_{i} } |
    \,d\mu(w)
    \\ \notag
    &\leq
  C(c_\mu,c_\Omega)
    \sum_{ i = 0 }^{ \infty }
    ( 2^{ -i } s' )
    \frac{ 1 }{ 2^{-i}s' }
    \aint{ B_{i} }
    \aint{ B_{i} }  | u(w) - u(w') | \,d\mu(w')
    \,d\mu(w)
    \\ \notag
    &\leq
   C(c_\mu,c_\Omega)
    \sum_{ i = 0 }^{ \infty }
    ( 2^{ -i } s' )
    u^{\sharp}_{s'}(x)
    \leq
    2 C(c_\mu,c_\Omega) s' u^{\sharp}_{s'}(x).
\end{align}
We consider now $1/2 \geq s > 0$. Observe for each $y \in B( x, s ) \cap \Omega \setminus N$ that $B_0 = B( x, s ) \cap \Omega$ is contained in $B( y, 2s ) \cap \Omega$. Then, by applying \eqref{eq:chainargument:1} and \eqref{eq:chainargument:2} for $y$ in place of $x$, we conclude
\begin{align*}
    | u(y) - u_{ B_0 } |
    &\leq
    | u(y) - u_{ B(y,2s) \cap \Omega } |
    +
    | u_{ B(y,2s) \cap \Omega } - u_{B_0} |
    \\
    &\leq
    2 C(c_\mu,\Omega) 2s  u^{\sharp}_{2s}(y)
    +
    \aint{ B_0 }
        | u(w) - u_{ B(y,2s) \cap \Omega } |
    \,d\mu(w).
\end{align*}
Since $\mu( B(y,2s) \cap \Omega )\leq C(c_\mu,c_\Omega) \mu( B_0 )$ and $B_0 \subset B( y, 2s ) \cap \Omega$, we have
\begin{align*}
    \aint{ B_0 }
        | u(w) - u_{ B(y,2s) \cap \Omega } |
    \,d\mu(w)
    &\leq
    C(c_\mu,c_\Omega)
    \aint{ B(y,2s) \cap \Omega }
        | u(w) - u_{ B(y,2s) \cap \Omega } |
    \,d\mu(w)
    \\
    &\leq
    C(c_\mu,c_\Omega)
    \aint{ B(y,2s) \cap \Omega }
    \aint{ B(y,2s) \cap \Omega }
        | u(w) - u(w') |
    \,d\mu(w')
    \,d\mu(w).
\end{align*}
This implies the inequality
\begin{equation}\label{eq:chainargument:y}
    | u(y) - u_{ B_0 } |
    \leq
    3C(c_\mu,c_\Omega) ( 2s ) u^{\sharp}_{2s}(y).
\end{equation}
Claim \eqref{eq:localHaj:localgradient} follows by combining the inequalities \eqref{eq:chainargument:1}, \eqref{eq:chainargument:2} and \eqref{eq:chainargument:y}.
\end{proof}

\section[Lipschitz density and Hajlasz-Sobolev spaces]{Lipschitz density and Haj\l{}asz-Sobolev spaces}\label{sec:M=W}
In this section, we prove Theorems \ref{thm:M=W} and \ref{thm:density:Lips}. Recall that $Z$ is just a metric measure space without further assumptions.

\begin{proof}[Proof of \Cref{thm:M=W}]
It is clear that "$(3)\Rightarrow (2), (1)$". Recall also that by \Cref{lemma:M:in:W} we always have a $4$-Lipschitz linear embedding $M^{1,p}(Z,\mathbb V)\subset W^{1,p}(Z;\mathbb V)$ for every Banach space $\mathbb V$. 

Let us prove now "$(2)\Rightarrow (3)$". Assume that $M^{1,p}( Z; \ell^{\infty} ) = W^{1,p}( Z; \ell^{\infty} )$ as sets. Then the bounded inverse theorem yields the existence of $C > 0$ with
\begin{equation}\label{eq:boundedinverse}
    C^{-1}\| u \|_{ M^{1,p}( Z; \ell^{\infty} ) }
    \leq
    \| u \|_{ W^{1,p}( Z; \ell^{\infty} ) }
    \leq
    4\| u \|_{ M^{1,p}( Z; \ell^{\infty} ) }.
\end{equation}
In this implication, we also prove that \eqref{eq:boundedinverse} holds with $\ell^{\infty}$ replaced by an arbitrary Banach space $\mathbb{V}$, with the same constant $C = C( \ell^{\infty} )$.

Next, for a given Banach space $\mathbb V$, take $u\in W^{1,p}(Z; \mathbb V)$ and fix a representative in $ \widetilde{N}^{1,p}( Z; \mathbb{V} )$ denoted the same way. By \Cref{lemma:usefulfact}, there exists a $1$-Lipschitz linear $\iota \colon \mathbb{V} \rightarrow \ell^{\infty}$ for which $U \coloneqq \iota \circ u \in \widetilde{N}^{1,p}( Z; \ell^{\infty} )$ and a negligible set $E \subset Z$ so that
\begin{equation*}
    u( Z \setminus E ) \subset \mathbb{V}_0
\end{equation*}
for a closed separable subspace $\mathbb{V}_0 \subset \mathbb{V}$ with $\iota|_{ \mathbb{V}_0 }$ being an isometry. We also have that $\rho_u \equiv \rho_U$ and
\begin{equation}\label{eq:isometryinequality}
    | u(x) - u(y) | = | U(x) - U(y) | \quad\text{for every $x, y \in Z \setminus E$.}
\end{equation}
Since $U \in \widetilde{N}^{1,p}( Z; \ell^{\infty} )$, the Lebesgue equivalence class of $U$ is contained in the space $W^{1,p}( Z; \ell^{\infty} ) = M^{1,p}( Z; \ell^{\infty} )$. Then \eqref{eq:isometryinequality} implies that $u$ and $U$ have the same Haj{\l}asz upper gradients. Thus we have $u\in M^{1,p}(Z;\mathbb V)$ and, by \eqref{eq:boundedinverse}, we have
\begin{equation*}
    C^{-1}\| u \|_{ M^{1,p}( Z; \mathbb{V} ) }
    \leq
    \| u \|_{ W^{1,p}( Z; \mathbb{V} ) }
    \leq
    4\| u \|_{ M^{1,p}( Z; \mathbb{V} ) }.
\end{equation*}

Let us show "$(1)\Rightarrow (2)$" to finish the proof. Assume $M^{1,p}( Z; c_0 ) = W^{1,p}( Z; c_0 )$ as sets. Then, by the bounded inverse theorem, there exists $C = C( c_0 ) > 0$ with
\begin{equation}\label{eq:boundedinverse:c0}
    C^{-1}\| u \|_{ M^{1,p}( Z; c_0) }
    \leq
    \| u \|_{ W^{1,p}( Z; c_0 ) }
    \leq
    4\| u \|_{ M^{1,p}( Z; c_0 ) }.
\end{equation}
Fix $u \in W^{1,p}( Z; \ell^\infty )$ and let us verify $u\in M^{1,p}( Z; \ell^{\infty} ) $. For each $N \in \mathbb{N}$, we consider the $1$-Lipschitz linear mapping
\begin{equation*}
    \pi^N \colon \ell^\infty \rightarrow c_0, \quad ( x_i )_{ i = 1 }^{ \infty } \mapsto ( x_i )_{ i = 1 }^{ N }.
\end{equation*}
Note that $u^N = \pi^N \circ u \in W^{1,p}( Z; c_0 )$. Therefore, by \eqref{eq:boundedinverse:c0}
\begin{equation*}
    \sup_{N}
    \inf_{ g \in \mathcal{D}_{p}( u^{N} ) }
        \|g\|_{ L^{p}( Z ) }
    \leq \sup_N 
        \|u^{N}\|_{ M^{1,p}( Z;c_0 ) }\leq
        C\sup_N \|u^N\|_{W^{1,p}(Z;c_0)}\leq
    C \| u \|_{ W^{1,p}( Z; \ell^{\infty} ) }.
\end{equation*}
We prove next that $\mathcal{D}_{p}( u ) \neq \emptyset$, and consider, for each $N \in \mathbb{N}$, the set
\begin{equation*}
    \mathcal{K}_N
    \coloneqq
    \mathcal{D}_{p}( u^{N} )
    \cap
    \left\{ g \in L^{p}(Z) \colon \|g\|_{ L^{p}( Z ) } \leq 1+ C \| u \|_{ W^{1,p}( Z; \ell^{\infty} ) } \right\}.
\end{equation*}
The set $\mathcal{K}_N$ is nonempty, convex, bounded, and closed under $L^{p}( Z )$-convergence. Hence, by the reflexivity of $L^{p}( Z )$, the set is compact in the weak topology. Since also $\mathcal{K}_{N+1} \subset \mathcal{K}_{N}$, the intersection $\bigcap_N \mathcal K_N $ is non-empty by the Cantor's intersection theorem (or the \v{S}mulian theorem \cite{Smulian}). The equivalence between (3) and (5) in \Cref{lemm:projection} yields that any element in the intersection is also in $\mathcal{D}_p( u )$. Thus $\mathcal{D}_{p}( u ) \neq \emptyset$. We have now verified that $u \in M^{1,p}( Z; \ell^{\infty} )$.
\end{proof}

Next, we study the density of Lipschitz functions in $W^{1,p}(Z;\mathbb V)$. In particular, we aim to prove \Cref{thm:density:Lips}. As a consequence we are able to deduce sufficient conditions for the equality $M^{1,p}( Z ) = W^{1,p}( Z )$ to imply the equality $M^{1,p}( Z; \mathbb{V} ) = W^{1,p}( Z; \mathbb{V} )$ for $\mathbb{V} \in \left\{ c_0, \ell^\infty \right\}$, see \Cref{prop:coincidence}.

Before delving further, we state some definitions. Given a Banach space $\mathbb{V}$ and $1 \leq p \leq \infty$, we say that Lipschitz functions with bounded support are \emph{energy-dense} (or \emph{dense in energy})  in $W^{1,p}( Z; \mathbb{V} )$ if for every $u \in W^{1,p}( Z; \mathbb{V} )$, there exists a sequence of $\mathbb{V}$-valued Lipschitz functions with bounded support $( u_n )_{ n = 1 }^{ \infty }$ such that
\begin{equation}\label{eq:energydensity:definition}
    \| u - u_n \|_{ L^{p}(Z;\mathbb{V}) }
    +
    \| \rho_u - \rho_{ u_n } \|_{ L^{p}(Z) }
    \rightarrow
    0
    \quad\text{as $n \rightarrow \infty$}.
\end{equation}
This convergence is weaker than norm-convergence, cf. \cite[Introduction]{EB:20:energy}. Moreover, it is known that in any complete and separable metric measure space, Lipschitz functions with bounded support are always dense in energy in $W^{1,p}( Z )$, cf. \cite[Theorem 1]{EB:20:energy}, while the norm-density fails, e.g. for the slit disk, cf. \Cref{ex:slitdisk}. Observe that \cite[Theorem 1]{EB:20:energy} yields the existence of a sequence for which
\begin{equation*}
    \| u - u_n \|_{ L^{p}(Z) }
    +
    \| \rho_u - {\rm lip}_a( u_n ) \|_{ L^{p}(Z) }
    \rightarrow
    0
    \quad\text{as $n \rightarrow \infty$},
\end{equation*}
where
\begin{equation*}
    {\rm lip}_a( v )(z) \coloneqq \lim_{ r \rightarrow 0^+ } \sup_{ y \neq x, x,y\in B( z, r ) } \frac{ |v(x)-v(y)| }{ d(x,y) }
\end{equation*}
is the \emph{asymptotic Lipschitz constant} of $v \in \LIP( Z )$ at $z \in Z$. The same sequence satisfies \eqref{eq:energydensity:definition} for $\mathbb{V} = \mathbb{R}$, see e.g. \cite[Lemma 2.2]{EB:20:energy}.

We state now an important definition in the subsequent results.
\begin{definition}\label{defi:keyfact:bounded}
Let $\lambda \geq 1$. A Banach space $\mathbb{V}$ is said to have the \emph{$\lambda$-bounded approximation  property} if for every compact set $K \subset \mathbb{V}$ and every $\epsilon > 0$, there exists a finite rank operator $T \colon \mathbb{V} \rightarrow \mathbb{V}$, with $\| T \| \leq \lambda$ and
\begin{equation*}
    | v - T(v) | \leq \epsilon
    \quad\text{for every $v \in K$}.
\end{equation*}
We say that $\mathbb{V}$ has the \emph{bounded approximation property} if it has the $\lambda$-bounded approximation property for some $\lambda \geq 1$. When we may take $\lambda = 1$, the Banach space has the \emph{metric approximation property}.
\end{definition}
The following definition gives a sufficient condition for the bounded approximation property.
\begin{definition}\label{def:Schauder}
Given a Banach space $\mathbb{V}$, we say that a sequence $( v_n )_{ n = 1 }^{ \infty }$ is a \emph{Schauder basis} on $\mathbb{V}$ if each $v \in \mathbb{V}$ can be \emph{uniquely} expressed as a convergent series $\sum_{ n = 1 }^{ \infty } c_n v_n$, $c_n \in \mathbb{R}$. We let
\begin{equation*}
    T_k \colon \mathbb{V} \to \mathbb{V},
    \sum_{ n = 1 }^{ \infty }c_n v_n
    \mapsto
    \sum_{ n = 1 }^{ k }c_n v_n
\end{equation*}
denote the \emph{$k$'th canonical projection}. We say that a Schauder basis is \emph{asymptotically $\lambda$-monotone} if $\limsup_{k}\| T_k \| \leq \lambda$.
\end{definition}
The uniform boundedness principle implies that $\sup_{ k } \| T_k \| \leq C < \infty$. In particular, $\mathbb{V}$ has the $C$-bounded approximation property. The $\ell^{q}$-spaces for $1 \leq q < \infty$ and $c_0$ are exemplary Banach spaces with Schauder basis with $\sup_{k} \| T_k \| = 1$. Moreover, each Hilbert space has the metric approximation property. For these and related facts, we refer the reader to \cite[Chapter 7]{Joh:Lin:01}.

Before establishing \Cref{thm:density:Lips}, we prove some preliminary results.
\begin{lemma}\label{lemm:real-to-finite:energy}
Let $Z$ be a metric measure space, $1 \leq p < \infty$, and $\mathbb{W}$ and $\mathbb{H}$ Banach spaces. Suppose the existence of a linear isometric embedding $\iota \colon \mathbb{W} \rightarrow \mathbb{H}$ and a linear $1$-Lipschitz $P \colon \mathbb{H} \rightarrow \mathbb{W}$ such that $P \circ \iota$ is the identity.

Then, given $\lambda \geq 1$, the following holds: If $u \in W^{1,p}(Z; \mathbb{W} )$ and $( u_i )_{ i = 1 }^{ \infty }$ is a sequence in $W^{1,p}( Z; \mathbb{H} )$ with
\begin{enumerate}
    \item $\| \iota(u) - u_{i} \|_{ L^{p}( Z; \mathbb{H} ) } \rightarrow 0$ as $i \rightarrow \infty$;
    \item $\sup_{ i } \rho_{ u_i } \in L^{p}(Z)$ and $\limsup_{ i \rightarrow \infty } \rho_{ u_i } \leq \lambda \rho_{\iota(u)}$ almost everywhere, 
\end{enumerate}
then we have the following:
\begin{enumerate}
    \item $\| u - P(u_{i}) \|_{ L^{p}( Z; \mathbb{W} ) } \rightarrow 0$ as $i \rightarrow \infty$;
    \item every subsequence of $( \rho_{P(u_i)} )_{ i = 1 }^{ \infty }$ has a weakly convergent subsequence. Any such weak limit $\rho'$ satisfies $\rho_u \leq \rho'$ almost everywhere, and $\limsup_{ i \rightarrow \infty } \rho_{ P(u_i) } \leq \lambda \rho_{u}$ almost everywhere.
\end{enumerate}
In particular, when $\lambda = 1$, the $\rho_{ P(u_i) }$ converge strongly to $\rho_u$ in $L^p( Z )$.
\end{lemma}
\begin{proof}
Since $P( \iota(u) ) = u$ and $P$ is $1$-Lipschitz, we have
\begin{equation*}
    \| u - P(u_i) \|_{ L^{p}( Z; \mathbb{W} ) }
    \leq
    \| \iota(u) - u_i \|_{ L^{p}( Z; \mathbb{H} ) }.
\end{equation*}
Therefore $P(u_i)$ converge to $u$ in $L^{p}( Z; \mathbb{W} )$. Moreover, $P$ being $1$-Lipschitz implies that $P( u_i ) \in W^{1,p}( Z; \mathbb{W} )$ and $\rho_{ P( u_i ) } \leq \rho_{ u_i }$, cf. \Cref{lemm:basic:new}.

We now show the weak compactness of the sequence $(\rho_{P(u_i)})^{\infty}_{ i = 1 } $ distinguishing between the cases $p> 1$ and $p=1$.

For $p > 1$, the weak compactness of $( \rho_{ u_i } )_{ i = 1 }^{ \infty }$ implies a norm bound for $( \rho_{ P(u_i) } )_{ i = 1 }^{ \infty }$, which is enough for the weak compactness of $(\rho_{P(u_i)})^{\infty}_{i=1} $ by reflexivity.

For $p = 1$, we consider the measure $\nu(E) = \int_{E} \sup_{ i } \rho_{ u_i} \,d\mu$, for each measurable $E \subset Z$, and the sequence of functions defined by 
\begin{equation}
\Phi_i(x) \coloneqq \begin{cases}
\frac{\rho_{ P(u_i) }(x)}{\sup_{ i } \rho_{ u_i}(x)},\ \ & {\rm if}\  \sup_{ i } \rho_{ u_i}(x) \neq 0,\\
0,\ \ & {\rm elsewhere}.
\end{cases}\nonumber
\end{equation}

Denote by $L^{2}( \nu )$ the space of $\nu$-measurable functions $h$ with $|h|^2$ being $\nu$-integrable. Then $\Phi_i \leq 1 \in L^{2}( \nu )$, so by the reflexivity of the $L^2(\nu)$-space defined by $\nu$, we obtain that every subsequence has a further subsequence weakly convergent in $L^2(\nu)$. That is, there exists $\Phi \in L^{2}( \nu )$ and a subsequence $( \Phi_{i_j} )_{ j = 1 }^{ \infty }$ such that
\begin{equation*}
    \int g \Phi \,d\nu
    =
    \lim_{ j \rightarrow \infty }
    \int g \Phi_{i_j} \,d\nu
    \quad\text{for every $g \in L^{2}( \nu )$.}
\end{equation*}
Fix a Borel representative of $\Phi$ and set $\rho'(z) \coloneqq \Phi(z) \sup_{ i } \rho_{ u_i }(z)$ for every $z \in Z$, the product set to zero everywhere in $\left\{ z \in Z \colon \sup_{ i } \rho_{ u_i }(z) = 0 \right\}$. Then, since every $g \in L^{\infty}(Z)$ defines a unique element in $L^{2}( \nu )$, we see that
\begin{equation*}
    \int g(z) \rho'(z) \,d\mu
    =
    \lim_{ j \rightarrow \infty }
    \int g(z) \rho_{i_j}(z) \,d\mu(z)
    \quad\text{for every $g \in L^{\infty}( Z )$.}
\end{equation*}
In particular, the weak compactness holds also for $p = 1$.

Now for all $1 \leq p < \infty$, suppose the existence of a weak limit $\rho'$ of some subsequence $( P( u_i ) )_{ i = 1 }^{\infty}$. The $L^{p}$-convergence of $P( u_i )$ to $u$ implies that every Borel representative of $\rho'$ is a $p$-weak upper gradient of some Newton--Sobolev representative of $u$, cf. \cite[Mazur's Lemma and Proposition 7.3.7]{HKST2015}. Therefore $\rho_u \leq \rho'$. Claim (2) follows.

When $\lambda = 1$, we have that every weak limit of any weakly convergent subsequence of $\rho_{ u_i }$ is equal to $\rho_u$ almost everywhere and a further subsequence converges pointwise almost everywhere to $\rho_u$. Dominated convergence implies that the convergence is in $L^{p}(Z)$ and, by the uniqueness of the limit, the original sequence itself converges to $\rho_u$ in $L^{p}(Z)$.
\end{proof}

\begin{proposition}\label{lemm:densityoffinitedim}
Let $Z$ be a metric measure space and $1 \leq p< \infty$. Suppose that $\mathbb{V}$ is a Banach space with the $\lambda$-bounded approximation property. Then for every $u \in W^{1,p}( Z; \mathbb{V} )$, there exists a sequence $( u_i )_{ i = 1 }^{ \infty } \subset W^{1,p}( Z; \mathbb{V} )$ with the following properties:
\begin{enumerate}
    \item up to a negligible set, the image of $u_i$ is contained in a finite-dimensional subspace $\mathbb{V}_i$ of $\mathbb{V}$ with $\mathbb{V}_i \subset \mathbb{V}_{i+1}$ for every $i$;
    \item $\| u - u_i \|_{ L^{p}(Z; \mathbb{V}) } \rightarrow 0$ as $i \rightarrow \infty$;
    \item $\sup_{ i }\rho_{ u_i } \in L^{p}( Z )$ and $\limsup_{ i \rightarrow \infty } \rho_{ u_i } \leq \lambda \rho_u$ almost everywhere.
\end{enumerate}
In particular, when $\mathbb{V}$ has the metric approximation property, the $\rho_{ u_i }$ converge to $\rho_u$ strongly in $L^{p}(Z)$.
\end{proposition}

\begin{proof}
We fix $\lambda \geq 1$ for the rest of the proof.

Claim (1): Proving the properties (1)--(3) for Banach spaces $\mathbb{H}$ with a Schauder basis that is asymptotically $\lambda$-monotone implies the same properties for all Banach spaces $\mathbb{V}$ that have the $\lambda$-bounded approximation property.

Fix $u \in W^{1,p}( Z; \mathbb{V} )$ with a Newton--Sobolev representative $\widetilde{u} \in \widetilde{N}^{1,p}( Z; \mathbb{V} )$. Then, by \Cref{lemm:representation}, there exists a negligible set $N \subset Z$ such that $\widetilde{u}( Z \setminus N )$ is contained in a separable subspace $\mathbb{V}_0 \subset \mathbb{V}$. By arguing as in the proof of \Cref{lemma:usefulfact}, $p$-almost every constant speed path $\gamma$ is such that $\widetilde{u} \circ \gamma$ has image in $\mathbb{V}_0$. Thus we redefine $\widetilde{u}$ to be zero in the set $\widetilde{u}^{-1}( \mathbb{V} \setminus \mathbb{V}_0 )$, while still being a Newton--Sobolev representative of $\widetilde{u}$. In conclusion, we may and do assume $N = \emptyset$.

Since $\mathbb{V}_0$ is a closed separable subspace of $\mathbb{V}$, with $\mathbb{V}$ having the $\lambda$-bounded approximation property, there exists a separable closed subspace $\mathbb{V}_0 \subset \mathbb{W}$ that has the $\lambda$-bounded approximation property with the same constant, cf. \cite[Chapter 7, Theorem 9.7]{Joh:Lin:01}. For every separable Banach space $\mathbb{W}$ with the $\lambda$-bounded approximation property, there exist a Banach space $\mathbb{H}$ having an asymptotically $\lambda$-monotone Schauder basis and a linear isometric embedding $\iota \colon \mathbb{W} \rightarrow \mathbb{H}$ together with a $1$-Lipschitz linear map $P \colon \mathbb{H} \rightarrow \mathbb{W}$ such that $P \circ \iota$ is the identity, cf. \cite[Theorem 1.2]{Mu:Vi:10}. Claim (1) follows by applying \Cref{lemm:real-to-finite:energy}.

Claim (2): The properties (1)--(3) are true in any Banach space $\mathbb{H}$ with an asymptotically $\lambda$-monotone Schauder basis.

To this end, let $u \in W^{1,p}( Z; \mathbb{H} )$ and set $u_k \coloneqq T_{k}(u)$, where $T_k$ are the canonical projections. By definition of the Schauder basis, we have $u_k \rightarrow u$ pointwise. Then, by dominated convergence, applicable due to $| u_k | \leq C |u|$, $u_k$ converges to $u$ in $L^{p}( Z; \mathbb{H} )$.

We have that $\rho_{ u_k } \leq C\rho_u$, so every subsequence has a further weakly convergent subsequence $( \rho_{ u_{k_j} } )_{ j = 1 }^{ \infty }$ in $L^{p}( Z )$, with limit, say $\rho$; when $p > 1$, we apply reflexivity of $L^{p}( Z )$ while for $p = 1$ we argue as in the proof of \Cref{lemm:real-to-finite:energy}. Any Borel representative of such a $\rho$ is a $p$-weak upper gradient of some Newton--Sobolev representative of $u$, cf. \cite[Mazur's Lemma and Proposition 7.3.7]{HKST2015} by the $L^{p}$-convergence of $u_k$ to $u$. Next, observe that, by weak convergence,
\begin{equation}\label{eq:limsupestimate}
    \rho_u
    \leq
    \rho
    \leq
    \limsup_{ j \rightarrow \infty } \rho_{ u_{ k_j } }
    \quad\text{pointwise almost everywhere.}
\end{equation}
The limsup estimate follows by considering the convex sets $K_m$ formed by convex combinations of $( \rho_{ u_{k_j} } )_{ j \geq m }$ observing that, by Mazur's Lemma, there exists a sequence $g_m \in K_m$ converging to $\rho$ in $L^{p}( Z )$ and pointwise almost everywhere. Furthermore, by definition of $K_m$, we have that
\begin{equation*}
    g_m \leq \sup_{ j \geq m }\rho_{ u_{k_j} } \quad\text{almost everywhere for every $m$}.
\end{equation*}
Then \eqref{eq:limsupestimate} follows by passing to the limit $m \rightarrow \infty$.

We also have that
\begin{equation*}
    \limsup_{ j \rightarrow \infty } \rho_{ u_{ k_j } }
    \leq
    \limsup_{ j \rightarrow \infty } \| T_{ k_{j} } \| \rho_{ u }
    \leq
    \lambda \rho_u
    \quad\text{pointwise almost everywhere}.
\end{equation*}
Claim (2) follows.

When $\lambda = 1$, the convergence of $\rho_{ u_i }$ to $\rho_u$ in $L^{p}(Z)$ is argued as in the proof of \Cref{lemm:real-to-finite:energy}.
\end{proof}

\begin{lemma}\label{lemm:finitedimension:approximation:weak}
Let $Z$ be a metric measure space, $1 \leq p < \infty$, and $\mathbb{V}$ a finite-dimensional Banach space. If Lipschitz functions with bounded support are norm-dense in $W^{1,p}( Z )$, then they are norm-dense in $W^{1,p}( Z; \mathbb{V} )$.
\end{lemma}
\begin{proof}
Let $n$ denote the dimension of $\mathbb{V}$. Since all $n$-dimensional Banach spaces are bi-Lipschitz isomorphic, there exists a $C$-bi-Lipschitz isomorphism $\phi \colon \mathbb{V} \to \mathbb{V}_n$, the latter space denoting $\mathbb{R}^n$ endowed with the supremum norm; in fact, by John's theorem \cite[Theorem 14.1]{Tom-Jae:89}, we may take $C = n$.

The mapping $\phi$ induces an $n$-bi-Lipschitz isomorphism from $W^{1,p}(Z; \mathbb{V} )$ into $W^{1,p}( Z; \mathbb{V}_n )$, by defining $h = \phi(u)$ for every $u \in W^{1,p}( Z; \mathbb{V} )$. Thus it suffices to prove the claim for $\mathbb{V}_n$.

Fix the standard basis $( e_i )_{ i = 1 }^{ n }$ for $\mathbb{V}_n$. Then each $h \in W^{1,p}( Z; \mathbb{V}_n )$ can be expressed as $u = \sum_{ i = 1 }^{ n } u_i e_i$ for some $u_i \in W^{1,p}( Z )$, $i = 1,2,\dots n$. By the assumed norm-density, we find Lipschitz functions with bounded support $h_{i,m}$ such that
\begin{equation*}
    \sum_{ i = 1 }^{ n }\| u_i - h_{i,m} \|_{ W^{1,p}(Z) } \leq 2^{-m}.
\end{equation*}
We denote $h_m \coloneqq \sum_{ i = 1 }^{ n } h_{i,m} e_i$. Then
\begin{equation}\label{eq:claimedequalities}
    \rho_{ (u - h_m) } = \max_{ i } \rho_{ (u_i - h_{i,m}) }
 \end{equation}
follows from \Cref{lemm:projection:new}: every $w \in ( \mathbb{V}_n )^{*}$ with $|w| \leq 1$ can be expressed as a convex combination of $( e_i^{*} )_{ i = 1 }^{ n }$ and $( -e_i^{*} )_{ i = 1 }^{ n }$, where $e_{i}^{*} \in ( \mathbb{V}_n )^{*}$ is the standard $i$'th coordinate projection, cf. \cite[Lemma 2.4.5]{Sch:14}, and $\rho_{ (u_i - h_{i,m} ) } = \rho_{ e_{i}^{*}( u - h_m ) } = \rho_{ - e_{i}^{*}( u - h_m ) }$. Then the convexity of $w \mapsto \rho_{w(u)}$ implies
\begin{equation*}
    \rho_{ w(u-h_m) } \leq \max_{ i } \rho_{ e_{i}^{*}(u-h_m) } \leq \rho_{u-h_m}.
\end{equation*}
Taking the supremum over such $w$ yields the claimed equality \eqref{eq:claimedequalities}. By triangle inequality, we have
\begin{equation*}
    \| u - h_{m} \|_{ W^{1,p}(Z;\mathbb{V}) } \leq 2^{-m}.
\end{equation*}
The claim follows by passing to the limit $m \rightarrow \infty$.
\end{proof}

\begin{proof}[Proof of \Cref{thm:density:Lips}]
We are given $1 \leq p < \infty$ and the  norm-density of Lipschitz functions with bounded support in $W^{1,p}( Z )$. We wish to show the energy-density of Lipschitz functions with bounded support in $W^{1,p}( Z ;\mathbb V)$ for every Banach space $\mathbb V$ with the metric approximation property. To this end, fix $u \in W^{1,p}( Z; \mathbb{V} )$.

Given $j \in \mathbb{N}$, we conclude from \Cref{lemm:densityoffinitedim} the existence of $u_j \in W^{1,p}( Z; \mathbb{V} )$ with
\begin{equation*}
    \| u - u_j \|_{ L^{p}( Z; \mathbb{V} ) }
    +
    \| \rho_u - \rho_{ u_j } \|_{ L^{p}( Z ) }
    \leq
    2^{-j},
\end{equation*}
and that $u_j$ has a representative in $\widetilde{N}^{1,p}( Z; \mathbb{V} )$, denoted the same way, with image in a finite-dimensional subspace, say $\mathbb{V}_j$, satisfying $\mathbb{V}_{j} \subset \mathbb{V}_{j+1}$ for every $j$.

By applying \Cref{lemm:finitedimension:approximation:weak}, we find Lipschitz functions $v_{j} \in W^{1,p}( Z; \mathbb{V}_j )$ with bounded support for which
\begin{equation*}
    \| u_j - v_j \|_{ L^{p}( Z; \mathbb{V}_j ) }
    +
    \| \rho_{ u_j } - \rho_{ v_j } \|_{ L^p(Z) }
    \leq
    2^{-j}.
\end{equation*}
We conclude from this fact that $v_j$ converges to $u$ in $L^{p}( Z; \mathbb{V} )$ and $\rho_{ v_j }$ to $\rho_u$ in $L^p( Z )$.
\end{proof}
We end this section by proving the following proposition closely related to \Cref{ques:startingpoint}, cf. \Cref{thm:V-extension:PI}.
\begin{proposition}\label{prop:coincidence}
Let $Z$ be a metric measure space and $1 < p < \infty$. If $\mathbb{V} \neq \left\{0\right\}$ is a Banach space with the bounded approximation property, then the following are equivalent.
\begin{enumerate}
    \item $M^{1,p}( Z; \mathbb{V} ) = W^{1,p}( Z; \mathbb{V} )$ as sets;
    \item the $4$-Lipschitz linear embedding $M^{1,p}( Z; \mathbb{V} ) \subset W^{1,p}( Z; \mathbb{V} )$ has a closed image and Lipschitz functions with bounded support are norm-dense in $W^{1,p}( Z )$;
    \item $M^{1,p}( Z; \mathbb{V} ) \subset W^{1,p}( Z; \mathbb{V} )$ is a closed subspace and $M^{1,p}( Z ) = W^{1,p}( Z )$ as sets.
\end{enumerate}
In particular, under any of these assumptions, Lipschitz functions are norm-dense in $W^{1,p}(Z)$.
\end{proposition}
\begin{proof}
Let $\lambda \geq 1$ be such that $\mathbb{V}$ has the $\lambda$-bounded approximation property.

Let us prove "(2) $\Rightarrow$ (1)". Fix $u \in W^{1,p}( Z; \mathbb{V} )$. We recall from \Cref{lemm:densityoffinitedim} the existence of a sequence $( u_j )_{j=1}^{\infty}$ with finite-dimensional image and $\sup \rho_{ u_j } \in L^{p}( Z )$ with
\begin{equation*}
    \| u - u_j \|_{ L^{p}( Z; \mathbb{V} ) }
    \leq
    2^{-j}
    \quad\text{and}\quad
    \limsup_{ j \rightarrow \infty } \rho_{ u_j } \leq \lambda \rho_u.
\end{equation*}
By arguing as in the proof of \Cref{thm:density:Lips} using \Cref{lemm:finitedimension:approximation:weak}, we may assume that each $u_j$ is Lipschitz with bounded support. In particular, $u_j \in M^{1,p}( Z; \mathbb{V} )$ and $( u_j )_{ j = 1 }^{ \infty }$ is a bounded sequence in $W^{1,p}( Z; \mathbb{V} )$. Since the $W^{1,p}$-norm and the $M^{1,p}$-norm are equivalent in the (closed) image of the embedding, by the bounded inverse theorem, we have that
\begin{equation*}
    ( u_j )_{ j = 1 }^{ \infty }
    \quad\text{is bounded in $M^{1,p}( Z; \mathbb{V} )$.}
\end{equation*}
So there exists a subsequence $( u_{j_i} )_{ i = 1 }^{ \infty }$ and a sequence $G_{ i }$ of Haj\l{}asz gradients of $u_{ j_i }$ converging weakly to some $G \in L^{p}( Z )$, by reflexivity of $L^{p}( Z )$. Since $u_{ j_i }$ converge strongly to $u$ in $L^{p}( Z; \mathbb{V} )$, a standard application of Mazur's lemma implies that $G$ is a Haj\l{}asz gradient of $u$. In particular, $u \in M^{1,p}( Z; \mathbb{V} )$. The claim follows.

Under the assumption (1), the equality $M^{1,p}( Z ) = W^{1,p}( Z )$ as sets follows by considering an isometric embedding $\iota \colon \mathbb{R} \to \mathbb{V}$, so (3) follows.

Under the assumption (3), the equivalences of the norms of $M^{1,p}( Z ) = W^{1,p}( Z )$ and the norm-density of Lipschitz functions with bounded support in $M^{1,p}( Z )$, cf. \cite{M=W,HKST2015}, shows (2).
\end{proof}

\section[Hajlasz--Sobolev extension sets]{Haj\l{}asz--Sobolev extension sets}\label{sec:proof:main}\
The goal of this section is to prove \Cref{th:Mexten}. 
\subsection{Statements}
In this section we fix a metric measure space $Z$. We also consider  sets $F \subset Z$ which satisfy the \emph{measure-density condition}:
\begin{equation}\label{eq:measuredensity}
    \mu( B(z,r) \cap F )
    \geq
    c_F \mu( B(z,r) )
    \quad\text{for every $z \in F$ and $0 < r \leq 1$}
\end{equation}
for some constant $c_F >0$. We also assume $\mu$ to be doubling with doubling constant $c_\mu$.

We show the following variant of \Cref{th:Mexten}.

\begin{theorem}\label{thm:intermediate}
Let $Z$ be a doubling metric measure space and $F \subset Z$ a closed set which satisfies \eqref{eq:measuredensity}. Then, for every $1 \leq p < \infty$ and Banach space $\mathbb{V}$, there exists a bounded linear extension operator $E_{\mathbb{V}} \colon M^{1,p}( F; \mathbb{V} ) \rightarrow M^{1,p}( Z; \mathbb{V} )$ satisfying the following:
\begin{itemize}
    \item $\| E_{\mathbb{V}} \| \leq C(c_\mu,c_F,p)$.
    \item If $\mathbb{V}, \mathbb{W}$ are Banach spaces and $T \colon \mathbb{V} \rightarrow \mathbb{W}$ is continuous and linear, then $E_{\mathbb{W}} \circ T = T \circ E_{\mathbb{V}}$.
\end{itemize}
\end{theorem}

We explain how \Cref{thm:intermediate} can be applied to all measurable sets $\Omega \subset Z$ which satisfy the measure-density condition. Indeed, if $\Omega$ satisfies \eqref{eq:measuredensity}, then
\begin{equation*}
    \mu( B(z,r) \cap \Omega )
    \geq
    c_\Omega \mu( B(z,r) )
    \quad\text{for every $z \in \overline{\Omega}$ and $0 < r \leq 1$}.
\end{equation*}
This fact and the Lebesgue differentiation theorem imply $\mu( \overline{ \Omega } \setminus \Omega ) = 0$. Therefore 
$$M^{1,p}( \overline{\Omega}; \mathbb{V} ) = M^{1,p}( \Omega; \mathbb{V} )$$
can be isometrically identified as Banach spaces (see also \Cref{lemm:closure})  and we may now apply \Cref{thm:intermediate} to the closed set $\overline{\Omega}$. In particular, \Cref{th:Mexten} follows from \Cref{thm:intermediate}.

\subsection{Some density estimates}
The measure-density condition \eqref{eq:measuredensity} can be quantitatively extended for larger $R > 1$.
\begin{lemma}\label{lemm:extendedmeasuredensity}
Given $R \geq 1$, there exists a constant $C = C( R, c_\mu, c_F )$ such that
\begin{equation}\label{eq:extendedmeasuredensity}
    \mu( B(z,r) \cap F ) \geq C( R, c_\mu, c_F ) \mu( B(z,r) )
\end{equation}
whenever $z \in F$ and $0 < r \leq R$.

Furthermore, whenever $( z, z_0 ) \in F \times F$ and $0 < r \leq r_0 \leq R$ satisfy $z \in B_0 = B( z_0, r_0 )$, then
\begin{equation*}
    \mu( B(z,r) \cap F )
    \geq
    C( R, c_\mu, c_F )
    \left(
        \frac{ r }{ r_0 }
    \right)^s
    \mu( B_0 \cap F ) ),
\end{equation*}
where $s = \log_2 c_\mu$.
\end{lemma}
\begin{proof}
The doubling property of $\mu$ yields that for $s = \log_2 c_\mu$ and $C = 4^{-s}$, whenever $( z, z_0 ) \in Z \times Z$ and $0 < r \leq r_0$ satisfy $z \in B( z_0, r_0 )$, then
\begin{equation}\label{eq:doublingproperty:balls}
    \mu( B(z,r) )
    \geq
    C\left(
        \frac{ r }{ r_0 }
    \right)^s
    \mu( B(z_0, r_0) );
    \quad\text{see, e.g. \cite[Lemma 8.1.13]{HKST2015}.}
\end{equation}
Then, if $z \in F$ and $0 < r \leq R$, then either $0 < r \leq 1$ and
\begin{equation*}
    \mu( B(z,r) \cap F ) \geq c_F \mu( B(z,r) ),
\end{equation*}
or $r > 1$ and
\begin{equation*}
    \mu( B(z,r) \cap F )
    \geq
    c_F
    \mu( B(z,1) ),
\end{equation*}
so by \eqref{eq:doublingproperty:balls},
\begin{equation*}
    \mu( B(z,1) )
    \geq
    C( c_\mu )
    \mu( B(z,r) )
    \left(
        \frac{ 1 }{ r }
    \right)^s
    \geq
    C( c_\mu, R )
    \mu( B(z,r) ).
\end{equation*}
In either case, we obtain \eqref{eq:extendedmeasuredensity}. The final claim follows by combining \eqref{eq:extendedmeasuredensity} and \eqref{eq:doublingproperty:balls}.
\end{proof}

\subsection{Preparations for the proofs}
The main idea of the proof of \Cref{thm:intermediate} comes from the proof of \cite[Theorem 6]{HKT2008:B}, where the authors proved the real-valued case. We include quite a few details in the section since we need them during the proof of \Cref{thm:extensionresults} as well.

For the purposes of the proof, we introduce some notations and lemmas. First, we define an open neighborhood of $F$ by setting
\begin{equation*}
    U \coloneqq B(F,8)= \left\{ z \in Z \colon d( z, F ) < 8 \right\}.
\end{equation*}

We assume from now on that $F \neq Z$, since otherwise we have nothing to do. Next, for every $z \in Z$, set $r(z) \coloneqq d( z, F )/10$. From the family $\left\{ B(z, r(z)/5) \right\}_{ z \in Z\setminus F }$, we select a maximal subfamily $\left\{ B(z_i,r(z_i)/5) \right\}_{ i \in I}$ of pairwise disjoint balls where $I\subset \N$.

For each $i \in I$, we refer to $B_{i} \coloneqq B( z_{i}, r(z_i) )$ as a \emph{Whitney ball} with center $z_{i}$ and radius $r_{i} \coloneqq r(z_i)$. We call $\mathcal{W} = \left\{ B_i \right\}_{ i \in I }$ a \emph{Whitney covering} of $Z \setminus F$. The following properties are readily verified, cf. \eqref{eq:doublingproperty:balls} or \cite[Section 4]{HKST2015}.
\begin{lemma}\label{lemm:Whitney:metric}
There exists $M=M(c_\mu) \in \mathbb{N}$ such that
\begin{itemize} 
    \item[(W1)] Each $5B_{i}$ is an open ball inside $Z \setminus F$.
    \item[(W2)] $Z \setminus F =\bigcup_{i\in I} B_i$ and $5^{-1}B_i \cap 5^{-1}B_j = \emptyset$ whenever $i \neq j$.
    \item[(W3)] If $z \in 5B_i$, then $5 r_i < 10 r(z) = d( z, F ) < 15 r_i$. 
    \item[(W4)] There is $z_{i}^{\star} \in F$ such that $d( z_{i}, z_{i}^{\star} ) < 15 r_i$.
    \item[(W5)] $\sum_{ i\in I} \chi_{ 5B_i }(x) \leq M$ for every $x \in Z\setminus F$.
    \item[(W6)] Whenever $5B_i \cap 5B_j \neq \emptyset$ with $r_i \leq r_j$, then $r_j \leq 75 r_i$.
\end{itemize}
\end{lemma}
Next, fix a smooth function $\psi \colon \mathbb{R} \rightarrow \mathbb{R}$ with $\psi\equiv 1$ in $\left[0,1\right]$, $\psi \equiv 0$ on $\left[3/2, \infty\right)$, $0 \leq \psi \leq 1$  elsewhere and $\|\psi'(t)\| \leq 2$ for all $t\in\R$. Define
\begin{equation*}
    \psi_{i}(z)
    =
    \psi\left( \frac{ d(z,z_i) }{ r_{i} } \right).
\end{equation*}
Then $\psi_{i} \equiv 1$ in $\overline{B}_{i}$ and zero in the complement of $(3/2)B_{i}$. Also, each $\psi_{i}$ is $2/r_i$-Lipschitz. We define
\begin{equation*}
    \phi_{i}(z) = \frac{ \psi_{i}(z) }{ \sum_{ j \in I} \psi_{j}(z) }.
\end{equation*}
We say that $\left\{ \phi_{i} \right\}_{ i \in I }$ is a partition of unity associated to the \emph{Whitney cover} $\mathcal{W}$. The lemma below states some basic properties of the partition of unity, cf. \cite[Section 4]{HKST2015}.
\begin{lemma}\label{lemm:Whitney:partition:metric}
Let $\left\{ \phi_{i} \right\}_{ i \in I }$ be the previous partition of unity associated to the Whitney cover $\mathcal{W}$ of $Z\setminus F$. Then there are constants $M = M( c_\mu )$ and $K = K( M, \|\psi'\|_\infty )$ for which
\begin{itemize}
    \item $\phi_{i} \equiv 0$ in the complement of $2B_{i}$; 
    \item $\phi_{i}(z) \geq 1/M$ for every $z \in B_i$;
    \item each $\phi_{i}$ is $(K/r_i )$-Lipschitz (One may take $K=2/M$ for instance);
    \item $\chi_{Z\setminus F}(z) = \sum_{ i\in I } \phi_{i}(z)$ for every $z \in Z\setminus F$.
\end{itemize}
\end{lemma}
We fix an arbitrary point $z^{\star}_{i}$ satisfying the conclusion of (W4) in \Cref{lemm:Whitney:metric} and define a ball $B^\star_i$ by setting 
\[B^\star_i:=B(z^\star_i, r_i).\]
Observe that if $z\in (3/2)\overline{B}_i$, then
\begin{equation}
    \label{eq:inclusion:mod}
    B_{i}^{\star} \subset \overline{B}( z, 25r(z) ) \eqqcolon B_{z}.
\end{equation}
The inclusion \eqref{eq:inclusion:mod}, together with (W3) and \Cref{lemm:extendedmeasuredensity}, yield the inequalities
\begin{equation}\label{eq:use_of_ahlfors:mod}
    \mu( B_{i}^{\star} \cap F ) \leq \mu( B_z ) \leq C(c_\mu,c_F) \mu( B_i^{\star} \cap F ) \quad\text{for every $i$ and $z\in (3/2)\overline{B}_i \cap U$.}
\end{equation}

For each $z\in U\setminus F$, denote by $I_z$ the collection of all $i\in I$ such that $z\in 2\overline{B}_i$. The cardinality of $I_{z}$ is bounded from above by the number $M$ from (W5). Denote $I_{1} := \left\{ i \in I \colon r_i < 1 \right\}$ and observe that if $i\in I\setminus I_1$, then $r_i \geq 1$ and hence $d( 2B_i, F) \geq 8r_i \geq 8$, so $2\overline{B}_i \cap U=\emptyset$. Consequently, for every $z\in U\setminus F$, we have $I_z\subset I_1$ and 
\[\sum_{i\in I_z}\phi_i(z)=\sum_{i\in I}\phi(z)=\sum_{i\in I_1}\phi_i(z)=1. \]

Let $\mathbb V$ be a Banach space and $u\in\mathcal L^{p}(F; \mathbb{V} )$ be arbitrary for $1 \leq p < \infty$. We define an extension function $\widetilde{E}_{\mathbb{V}}(u)$ on $U$ by setting 
\begin{equation}\label{eq:ext.op.}
\widetilde{E}_{\mathbb{V}}(u)(x):=\begin{cases}
u(x), & \ {\rm if}\ x\in F,\\
\sum_{i\in I_1}\phi_i(x)u_{B^\star_i\cap F}, & \ {\rm if}\ x\in U\setminus F.
\end{cases}
\end{equation}

The following lemma shows that the extension function is pointwise controlled from above by the maximal function.
\begin{lemma}\label{lemm:maximal}
Let $1 \leq p < \infty$ and $u \in\mathcal L^{p}( F; \mathbb{V} )$. Then
\begin{equation}
    \label{eq:extension:maximalfunction}
    |\widetilde{E}_{\mathbb{V}}(u)(x)|
    \leq
    \widetilde{E}_{\mathbb{R}}(|u|)(x)
    \leq
    C(c_\mu,c_F)
    \mathcal{M}( \widehat{u} )(x)
    \quad\text{for almost every $x \in U$},
\end{equation}
where $\widehat u$ is the zero extension of $u$. Moreover, $\widetilde{E}_{\mathbb{V}}(u)$ is measurable and when $p > 1$, it satisfies $$\| \widetilde{E}_{\mathbb{V}}(u) \|_{ L^{p}( U;\mathbb V ) } \leq  C(c_\mu,c_F,p) \| u \|_{ L^{p}( F;\mathbb V ) }.$$
\end{lemma}
\begin{proof}
We first show that $|\widetilde{E}_{\mathbb{V}}(u)(x)| \leq \widetilde{E}_{\mathbb{R}}( |u| )(x)$ for every $x \in U$. The inequality holds for $x\in F$ by definition. If $x \in U \setminus F$, we observe
\begin{equation*}
    |\widetilde{E}_{\mathbb{V}}(u)(x)|
    \leq
    \sum_{ i \in I_{1} }
        |\phi_{i}|(x)| u_{ B_{i}^{\star} \cap F} |
    \quad\text{and}\quad
    | u_{ B_{i}^{\star} \cap F } | \leq |u|_{ B_{i}^{\star} \cap F },
\end{equation*}
the latter fact following from \cite[Remark 3.2.8]{HKST2015}. Using these inequalities and  that $\phi_{i} \geq 0$, we conclude
\begin{equation*}
    \sum_{ i \in I_{1} }
        |\phi_{i}|(x)| u_{ B_{i}^{\star} \cap F } |
    \leq
    \sum_{ i \in I_1 }
        \phi_{i}(x) |u|_{ B_{i}^{\star} \cap F }
    =
    \widetilde{E}_{\mathbb{R}}( |u| )(x).
\end{equation*}
Next we claim $\widetilde{E}_{\mathbb{R}}(|u|)(x) \leq \mathcal{M}( \widehat{u} )(x)\ {\rm for\ almost\ every}\ x \in U$. Firstly, since $\widehat u\in\mathcal L^p(Z;\mathbb V)$, the Lebesgue differentiation theorem applied to $|\widehat{u}|$ yields that $$\widetilde{E}_{\mathbb{R}}(|u|)(x) = |u(x)| \leq \mathcal{M}( |\widehat{u}| )(x) = \mathcal{M}( \widehat{u} )(x)$$ for almost every $x \in F$. And secondly, in the case $x \in U \setminus F$, we have
\begin{equation*}
    \widetilde{E}_{\mathbb{R}}(|u|)(x)
    =
    \sum_{ i \in I_1 }
        \phi_{i}(x) |u|_{ B_{i}^{\star} \cap F }
    \leq
    \sum_{ i \in I_{x} } |u|_{ B_{i}^{\star} \cap F }
\end{equation*}
and the property (W5), \eqref{eq:inclusion:mod} and \eqref{eq:use_of_ahlfors:mod} imply
\begin{equation*}
    \sum_{ i \in I_{x}  } |u|_{ B_{i}^{\star} \cap F }
    \leq
    C(c_\mu,c_F)\aint{B_x}|\widehat{u}(y)|d\mu(y)
    \leq
    C(c_\mu,c_F)\mathcal{M}( \widehat{u} )(x).
\end{equation*}
Therefore $\widetilde{E}_{\mathbb{R}}(|u|)(x) \leq C(c_\mu,c_F) \mathcal{M}( \widehat{u} )(x)$ for almost every $x \in U$.

The measurability of $\widetilde{E}_{\mathbb{V}}(u)$ is clear on $F$. On $U\setminus F$, the measurability follows from the fact that the restriction of $\widetilde{E}_{\mathbb{V}}(u)$ to each Whitney ball is continuous. Hence $\widetilde{E}_{\mathbb{V}}(u)$ is measurable.

When $p > 1$, the inequality $\| \widetilde{E}_{\mathbb{V}}(u) \|_{ L^{p}( U;\mathbb V ) } \leq C(c_\mu,c_F,p) \| u \|_{ L^{p}( F;\mathbb V ) }$ follows from \eqref{eq:extension:maximalfunction} and the boundedness of the maximal operator, cf. \Cref{lemm:maximalfunction:bounded}.
\end{proof}
The following lemma shows that the extension operator commutes with continuous linear mappings between Banach spaces.
\begin{lemma}\label{lemm:commutes}
Let $\mathbb{V}$, $\mathbb{W}$ be Banach and $T \colon \mathbb{V} \rightarrow \mathbb{W}$ be  continuous and linear. If $u \in L^{p}( F; \mathbb{V} )$ for some $1\leq p<\fz$, then $T( \widetilde{E}_{\mathbb{V}}(u) )(z) = \widetilde{E}_{\mathbb{W}}( T(u) )(z)$ for every $z \in U$.
\end{lemma}
\begin{proof}
We recall that integral averages over integrable mappings commute with continuous linear mappings, see \cite[Remark 3.2.8]{HKST2015}, for example. Hence the linearity of $T$ yields, for every $z \in U \setminus F$,
\begin{equation*}
    T\left( \widetilde{E}_{\mathbb{V}}( u ) \right)(z)
    =
    T\left( \sum_{ i \in I_{1} } \phi_{i}(z) u_{B_{i}^{*}\cap F} \right)(z)
    =
    \sum_{ i \in I_{1} } \phi_{i}(z) T( u_{ B_{i}^{*} \cap F } )
    =
    \widetilde{E}_{\mathbb{W}}( T(u) )(z).
\end{equation*}
The commutation property is clearly true in $F$, so the claim follows.
\end{proof}

\subsection{Proof of Theorem \ref{thm:intermediate} for \texorpdfstring{$p > 1$}{Lg}}

We prove a local variant of \cite[Theorem 6]{HKT2008:B} in the case $p>1$.
The local version shows more than what is strictly needed for the proof of \Cref{thm:intermediate}. However, we apply the following formulation during one of the key steps in the proof of \Cref{thm:extensionresults}.
\begin{proposition}\label{proposition:local}
Given $u \in \mathcal{L}^p( F; \mathbb{V} )$, $x \in F$, $r >0$, and $g \in \mathcal{D}_{p}( u|_{ B( x, 4r ) \cap F } )$, there exists a negligible set $N \subset B( x, r )$ such that
\begin{equation}\label{eq:hajlasz}
    | \widetilde{E}_{\mathbb{V}}(u)(y) - \widetilde{E}_{\mathbb{V}}(u)(z) |
    \leq
    d(y,z) C(c_\mu,c_F)  (\mathcal{M}_{4r}( \widehat{g} )(y) + \mathcal{M}_{4r}( \widehat{g} )(z) )
\end{equation}
for every $y, z \in U \cap B( x, r ) \setminus N$. Here $\widehat{g}$ denotes the zero extension of $g$ to the complement of $B( x, 4r ) \cap F$.
\end{proposition}
\begin{proof}
First, we fix a negligible set $N$ such that $F \setminus N$ consists of the Lebesgue points of the zero extensions of $g$ and $\widetilde{E}_{\mathbb{V}}(u)$.

We claim that \eqref{eq:hajlasz} holds for every $y, z \in U \cap B( x, r ) \setminus N$. Now, based on the locations of $y, z$, we divide the following argument into four cases:
\begin{itemize}
    \item[i)] $y, z \in F \cap B( x, r ) \setminus N$;
    \item[ii)] $y \in F \cap B( x, r ) \setminus N$ and $z \in U \cap B( x, r ) \setminus (F\cup N)$ (or the roles of $y,z$ are reversed);
    \item[iii)] $y,z \in U \cap B( x, r ) \setminus F$ and $d(y,z) \geq 10\min\left\{ r(y), r(z) \right\}$;
    \item[iv)] $y,z \in U \cap B( x, r ) \setminus F$ and $d(y,z) < 10\min\left\{ r(y), r(z) \right\}$.
\end{itemize}
Case i) follows from the fact that
\begin{equation*}
    | u(y) - u(z) | \leq d(y,z) ( \widehat{g}(y) + \widehat{g}(z) )
    \quad\text{for every $y, z \in F \cap B( x, r ) \setminus N$}
\end{equation*}
and \[\widehat{g}(z) \leq \mathcal{M}_{4r}( \widehat{g} )(z)\] for every $z \in B( x, r ) \setminus N$, the inequalities valid by the Lebesgue differentiation theorem.

For the purpose of showing the remaining cases, observe that $10 r(w) \leq d( w, y )$ for every $w \in B( x, r )$ and every $y \in F$. Recalling the definition of $B_{w}$ from \eqref{eq:inclusion:mod}, we deduce that the radius of $B_w$ is at most $5r/2$ and $B_w \subset B( x, 4 r )$ for every $w\in B(x,r)\setminus F$. Consequently,
$$g|_{ ( B_{w} \cap F ) \cup ( B_{w'} \cap F ) } \in \mathcal{D}_{p}( u|_{ ( B_{w} \cap F ) \cup ( B_{w'} \cap F ) } ) \quad\text{for every $w, w' \in B( x, r ) \setminus F$.}$$

We proceed now to the proof of Case ii). Notice that if $y \in F \cap B(x,r) \setminus N$ and $z \in U \cap B(x,r) \setminus(F\cup N)$, then
\begin{equation*}
    | \widetilde{E}_{\mathbb{V}}(u)(y) - \widetilde{E}_{\mathbb{V}}(u)(z) |
    \leq
    | u(y) - u_{ B_{z} \cap F } |
    +
    | u_{ B_{z} \cap F } - \widetilde{E}_{\mathbb{V}}(u)(z) |.
\end{equation*}
We first estimate as follows
\begin{equation*}
    | u(y) - u_{ B_{z} \cap F } |
    \leq
    \aint{ B_{z} \cap F } | u(y) - u(w) | \,d\mu(w)
    \leq
    \aint{ B_{z} \cap F } d( y, w )( g(y) + g(w) ) \,d\mu(w).
\end{equation*}
Here, note that for each $w \in B_{z} \cap F $,
$$d(y, w ) \leq d(y,z)+ d(z,w)\leq d(y,z)+25r(z)\leq d(y,z)(1+5/2)\leq4d( y, z ).$$  
Hence, since $B_{z} \cap F \subset B(x,4r)$,
\begin{align*}
    | u(y) - u_{ B_{z} \cap F } |
    &\leq
    4d( y, z )\left( g(y) + \aint{ B_{z} \cap F } g(w) \,d\mu(w) \right)
    \\
    &\leq
    4 d( y, z )( \mathcal{M}_{4r}( \widehat{g} )(y) + \mathcal{M}_{4r}( \widehat{g} )(z) ).
\end{align*}
Next, we estimate $| u_{ B_{z} \cap F } - \widetilde E_{\mathbb V}(z) |$ as follows:
\begin{align}\label{eq:essentialestimate:one}
    | u_{ B_{z} \cap F } - \widetilde E_{\mathbb V}(z) |
    &=
    \left| \sum_{ i \in I_{z} } \phi_{i}(z) ( u_{ B_{z} \cap F } - u_{ B_{i}^{*} \cap F } u(z) ) \right|
    \\\notag
    &\leq
    \sum_{ i \in I_{z} } \phi_i(z)
    \aint{ B_{z} \cap F } \aint{ B_{i}^{*} \cap F } | u(w)-u(w') | \,d\mu(w)\,d\mu(w')
    \\\notag
    &\leq
    C(c_\mu,c_F)\aint{ B_{z} \cap F } \aint{ B_{z} \cap F } | u(w)-u(w') | \,d\mu(w)\,d\mu(w')
    \\\notag
    &\leq
    2C(c_\mu,c_F) \cdot 2 \cdot 25 r(z) \mathcal{M}_{4r}( \widehat{g} )(z)\\ \notag
    &\leq 10 d(y,z) C(c_\mu,c_F) \mathcal{M}_{4r}( \widehat{g} )(z).
\end{align}
The last three inequalities we conclude from the following facts: First, we apply \eqref{eq:inclusion:mod} and \eqref{eq:use_of_ahlfors:mod} for every $i \in I_z$, together with $\sum_{ i \in I_{z} } \phi_{i}(z) = 1$. Secondly, we apply that $d( w, w' ) \leq 2 \cdot 25 r(z)$ for all $w, w' \in B_{z} \cap F$ and that $ B_z\cap F\subset B(x,4r)$. Lastly, recall that $100 r(z) \leq 10 d( y, z )$. Case ii) follows.

In Case iii), up to relabeling $y$ and $z$, we may assume $r( y ) \leq r( z )$. We first estimate
\begin{equation*}
    | \widetilde{E}_{\mathbb{V}}(u)(y) - \widetilde{E}_{\mathbb{V}}(u)(z) |
    \leq
    |  \widetilde{E}_{\mathbb{V}}(u)(y) - u_{ B_{y} \cap F } |
    +
    | u_{ B_{y} \cap F } - u_{ B_{z} \cap F } |
    +
    | u_{ B_{z} \cap F } - \widetilde{E}_{\mathbb{V}}(u)(z) |.
\end{equation*}
The first and third terms can be estimated as in \eqref{eq:essentialestimate:one} by using the facts \[10 r(y) \leq d(y,z) \quad\text{and}\quad 10 r(z) \leq 10 r(y) + d(y,z).\] Next, we estimate the middle term. First, we have \[d( w, w' ) \leq d(w,y) + d(y,z) + d(z,w') \leq 25 r(y) + d(z,y) + 25r(z) \leq 8 d(z,y)\] for every $( w, w' ) \in ( B_{y} \cap F ) \times ( B_{z} \cap F )$. Then a simple computation shows that for some $C = C( c_\mu, c_F )>0$,
\begin{align*}
    | u_{ B_{y} \cap F } - u_{ B_{z} \cap F } |
    &\leq
    C\aint{ B_{z} \cap F } \aint{ B_{y} \cap F } | u(w)-u(w') | \,d\mu(w)\,d\mu(w')
    \\
    &\leq
    C\aint{ B_{z} \cap F } \aint{ B_{y} \cap F } d(w,w') ( g(w) + g(w') ) \,d\mu(w)\,d\mu(w')
    \\
    &\leq
    C d( y,z )\left(
        \mathcal{M}_{4r}( \widehat{g} ) )(y)
        +
        \mathcal{M}_{4r}( \widehat{g} )(z)
    \right).
\end{align*}
Combining the previous inequalities yields Case iii).

In the remaining Case iv), we once again assume $r( y ) \leq r(z)$. Since \[\sum_{ i \in I_{y} \cup I_{z} } ( \phi_{i}(y) -  \phi_{i}(z) ) = 0,\]
we obtain
\begin{align*}
    | \widetilde{E}_{\mathbb{V}}(u)(y) - \widetilde{E}_{\mathbb{V}}(u)(z) |
    =
    \left| \sum_{ i \in I_{y} \cup I_{z} }
    ( \phi_{i}(y) - \phi_{i}(z) )( u_{ B_{i}^{*} \cap F } - u_{ B_{y} \cap F } )  \right|. 
\end{align*}
We recall from \Cref{lemm:Whitney:partition:metric} that every $\phi_i$ is $K/r_i$-Lipschitz. Thus
\begin{align}\label{eq:casefour}
    | \widetilde{E}_{\mathbb{V}}(u)(y) - \widetilde{E}_{\mathbb{V}}(u)(z) |
    \leq
    \sum_{ i \in I_{y} \cup I_{z} }
    \frac{ K d(y,z) }{ r_i }
        \aint{ B_{i}^{*} \cap F }
        \aint{ B_{y} \cap F }
            |u(w)-u(w')|
        \,d\mu(w)\,d\mu(w').
\end{align}
If $i \in I_y$, then $C^{-1}r_i\leq r(y)\leq Cr_i$ and $d(y,z) \leq Cr_i$, e.g. for $C = 15$. Then, by arguing as in \eqref{eq:essentialestimate:one}, we obtain
\begin{align}\label{eq:penultimate_eq_prop_5.6}
    \frac{ K d(y,z) }{ r_i }
        \aint{ B_{i}^{*} \cap F }
        \aint{ B_{y} \cap F }
            |u(w)-u(w')|
        \,d\mu(w)\,d\mu(w')
    \leq
    C(c_\mu,c_F) K d( y, z ) \mathcal{M}_{4r}( \widehat{g} )(y).
\end{align}
If $i \in I_{z} \setminus I_y$, then for every $( w, w' ) \in ( B_y \cap F ) \times (  B_{i}^{*} \cap F )$, we have \[d( w, w' ) \leq 25 r(y) + d( y, z ) + d( z, w' ).\] Here $d( y, z ) \leq 10r(y) \leq10 r(z) \leq C r_i$ and $d( z, w' ) \leq C r_i$, e.g. with $C = 15$. Hence
\begin{align}\label{eq:last_estimate_prop5.6}
\dfrac{ K d(y,z) }{ r_i } &\aint{ B_{i}^{*} \cap F } \aint{ B_{y} \cap F }|u(w)-u(w')|\,d\mu(w)\,d\mu(w')&  \\
&\leq C K d(y,z)\aint{ B_{i}^{*} \cap F }\aint{ B_y \cap F } (g(w)+ g(w'))\,d\mu(w)\,d\mu(w')\nonumber \\
&\leq C(c_\mu,c_F) d(y,z) ( \mathcal{M}_{4r}( \widehat{g} )(y) +\mathcal{M}_{4r}( \widehat{g} )(z) ),\nonumber
\end{align}
the last inequality holding due to \eqref{eq:use_of_ahlfors:mod}. Now Case iv) follows by combining \eqref{eq:penultimate_eq_prop_5.6}, \eqref{eq:last_estimate_prop5.6},  and by using the cardinality upper bound $2M$ for $I_{y} \cup I_{z}$.
\end{proof}

Note that by letting $r$ tend to infinity in \Cref{proposition:local}, we obtain the following corollary. The dependence on $p$ in the constants below come from the operator norm of the maximal function.

\begin{corollary}\label{cor:HKT08}
Let $1 < p < \infty$. If $u \in M^{1,p}( F; \mathbb{V} )$ and $g \in L^{p}( F )$ is a Haj\l{}asz upper gradient of $u$, then $\widetilde{E}_{\mathbb{V}}(u) \in L^{p}( U; \mathbb{V} )$ with Haj\l{}asz upper gradient $C\mathcal{M}( \widehat{g} ) \in L^{p}( U )$ for the constant $C=C(c_\mu,c_F)$ as in \Cref{proposition:local}. Moreover, 
\begin{equation*}
    \| \widetilde{E}_{\mathbb{V}}( u ) \|_{ L^{p}( U; \mathbb{V} ) } \leq  C(c_\mu,c_F,p) \| u \|_{ L^{p}( F; \mathbb{V} ) }
    \quad\text{and}\quad
    \| \mathcal{M}( \widehat{g} ) \|_{ L^{p}( U ) } \leq C(c_\mu,c_F,p) \| g \|_{ L^{p}( F ) }.
\end{equation*}
\end{corollary}
By the argument above, we have obtained an extension $\widetilde{E}_{\mathbb{V}}(u)$ of $u$ to $U$ so that
$$\|\widetilde{E}_{\mathbb{V}}(u)\|_{M^{1,p}(U;\mathbb V)}\leq C(c_\mu,c_F,p) \|u\|_{M^{1,p}(F;\mathbb V)}.$$ 
By using a cut-off function, we now extend $u$ to the whole space $Z$. More precisely, consider next a Lipschitz function $\phi \colon Z \rightarrow \mathbb{R}$ with $\phi|_{F} = 1$, $\phi|_{ Z \setminus U } = 0$, and $0 \leq \phi \leq 1$ otherwise. For example, $\phi(z) = \max\left\{0, 1 - d( F, z ) \right\}$ has the required properties and is $1$-Lipschitz. We define the extension operator by setting
$$E_{\mathbb{V}}(u) \coloneqq \phi \widetilde{E}_{\mathbb{V}}(u).$$

\begin{corollary}\label{corollary:extension:p>1}
Let $1 < p < \infty$ and $\mathbb{V}$ be a Banach space. Then \[E_{\mathbb{V}} \colon M^{1,p}( F; \mathbb{V} ) \rightarrow M^{1,p}( Z; \mathbb{V} )\] is a linear extension operator with the operator norm bounded from above by $C = C( c_\mu, c_F, p ) > 0$. Moreover, the extension operator commutes with continuous linear maps $T \colon \mathbb{V} \rightarrow \mathbb{W}$, for every Banach $\mathbb{V}, \mathbb{W}$.
\end{corollary}
\begin{proof}
Given that $\phi$ is $1$-Lipschitz and $0 \leq \phi \leq 1$, \Cref{lemm:basic} implies that 
\[\| E_{\mathbb{V}}(u) \|_{ M^{1,p}( Z; \mathbb{V}  ) } \leq 2 \| \widetilde{E}_{\mathbb{V}}(u) \|_{ M^{1,p}( U; \mathbb{V} ) }.\] \Cref{lemm:maximal} and \Cref{cor:HKT08} show that 
\[\| \widetilde{E}_{\mathbb{V}}(u) \|_{ M^{1,p}( U; \mathbb{V} ) } \leq C(c_\mu,c_F,p) \| u \|_{ M^{1,p}( F; \mathbb{V} ) }.\] The quantitativeness of the statement follows from the quoted results. The commuting property follows from \Cref{lemm:commutes}.
\end{proof}
The case $p > 1$ in \Cref{thm:intermediate} is proved by \Cref{corollary:extension:p>1}.

\subsection{Proof of Theorem \ref{thm:intermediate} for \texorpdfstring{$p = 1$}{Lg}}
We denote $q = s/(s+1)$, where $s = \log_2 c_\mu$ appears in \Cref{lemm:extendedmeasuredensity}. The ideas of the proof, for the $\R$-valued case, are contained in  \cite[pp. 656-660]{HKT2008:B}. However, everything translates to Banach-valued maps quite straightforwardly. Therefore we just give a quick sketch of the proof.

One first shows the following.
\begin{lemma}\label{lemm:localvariation}
Let $u \in M^{1,1}( F; \mathbb{V} )$ with a Haj\l{}asz upper gradient $g \in L^{1}( F )$.
\begin{enumerate}
    \item Then there exists a negligible set $N \subset U$ such that, for every $x, y \in U \setminus N$ with $d(x,y) < 1$, we have
\begin{equation*}
    |\widetilde{E}_{\mathbb{V}}(u)(x) - \widetilde{E}_{\mathbb{V}}(u)(y)|
    \leq
    d(x,y)
    C( c_\mu, c_F )
    \left( 
        \left( \mathcal{M}( \widehat{g}^{q} ) \right)^{1/q}(x)  + \left( \mathcal{M}( \widehat{g}^{q} ) \right)^{1/q}(y)
    \right).
\end{equation*}
\item Let $B = B(x_0,r_0)$ with $x_0 \in F$ and $r_0 = 1/10$. Then we have the inequality
\begin{equation*}
    \| \widetilde{E}_{\mathbb{V}}(u) \|_{ L^{1}( 5B ) }
    \leq
    C( c_\mu, c_F )
    \left(
        \|u\|_{ L^{1}( 5B \cap F ) }
        +
        \| ( \mathcal{M}( \widehat{g}^{1/q} ) )^{q} \|_{ L^{1}( 5B \cap F ) }
    \right).
\end{equation*}
\end{enumerate}
\end{lemma}
Claim (1) can be proved directly by following the proof of \cite{HKT2008:B} or deduced from the real-valued statement by applying the equivalence of (1) and (2) in \Cref{lemm:projection}, together with the fact that $\widetilde{E}_\mathbb{V}$ commutes with linear maps. Next, claim (2) follows from (1) as in the real-valued case.

Now consider a maximal family of open balls $\left\{ B_{i} \right\}_{ i = 1 }^{ m }$ with non-overlapping interiors and radii $r_i =  1/10$ centered at points of $F$, with $m \in \mathbb{N} \cup \left\{\infty\right\}$. 

Since $\mu$ is a doubling measure, we have $\sum_{ i = 1 }^{ m } \chi_{ 5B_{i} } \leq \widetilde{M}$ for some $\widetilde{M} = \widetilde{M}( c_\mu ) >0$. Observe that
\begin{equation*}
    W \coloneqq \left\{ x \in Z \colon d( x, F ) < 1/10 \right\}
    \subset
    \bigcup_{ i = 1 }^{ m } 3B_{i}.
\end{equation*}
Next, consider Lipschitz mappings $\widetilde{\psi}_{i}$ equal to $1$ in $3B_{i}$, zero outside $5B_{i}$, being $5$-Lipschitz and with $0 \leq \widetilde{\psi}_{i} \leq 1$. Denote
\begin{equation*}
    \widetilde{\phi}_{i}(x)
    = 
    \begin{cases}
       \frac{ \widetilde{\psi}_{i}(x) }{ \sum_{ j = 1 }^{ m } \widetilde{\psi}_{i}(x) }, & \ {\rm if}\ \sum_{ j = 1 }^{ m } \widetilde{\psi}_{i}(x) > 0,\\
        0, & \ {\rm otherwise}.
    \end{cases}
\end{equation*}
Fix a Lipschitz function $\psi \colon Z \rightarrow \mathbb{R}$ with $\psi \equiv 1$ in $F$, $\psi \equiv 0$ outside $W$ and $0 \leq \psi \leq 1$.

Denote $\widetilde{\eta}_{i} = \psi \widetilde{\phi}_{i}$ for each $i$. Observe that each $\widetilde{\eta}_{i}$ is $L$-Lipschitz, for some uniform $L>0$ depending only on the doubling constant $c_\mu$, and $0 \leq \widetilde{\eta}_{i} \leq 1$ supported in $5B_{i}$. Setting $\widetilde{\eta}_{i} \widetilde{E}_{\mathbb{V}}(u)$ to be zero outside $U$, we have a measurable function on $Z$. Define finally
$$E_{\mathbb{V}}( u ) \coloneqq \sum_{ i = 1 }^{ m } \widetilde{\eta}_{i} \widetilde{E}_{\mathbb{V}}(u). $$
The following lemma is proved exactly as in \cite{HKT2008:B}.

\begin{lemma}\label{lemm:boundedoverlap}
Let $u \in M^{1,1}( F; \mathbb{V} )$ with Haj\l{}asz upper gradient $g \in L^{1}( F )$. Then 
\begin{equation}\label{eq:cauchy}
 \|E_{\mathbb{V}}( u )\|_{ M^{1,1}( Z; \mathbb{V} ) }
    \leq
    C(c_\mu,c_F)\left( \|u\|_{ L^{1}( F; \mathbb{V} ) } + \| g \|_{ L^{1}( F ) } \right).
\end{equation}
In particular, $E_{\mathbb{V}}( u )\in M^{1,1}( Z; \mathbb{V} )$.
\end{lemma}

\Cref{lemm:boundedoverlap} has the following immediate corollary which finishes the proof of Theorem \ref{thm:intermediate} for $p = 1$.
\begin{proposition}\label{prop:extension}
The $E_{\mathbb{V}} \colon M^{1,1}( F; \mathbb{V} ) \rightarrow M^{1,1}( Z; \mathbb{V} )$ is a linear extension operator with operator norm $C = C( c_\mu, c_F )$. Moreover, $E_{\mathbb{V}}$ commutes with linear maps $T \colon \mathbb{V} \rightarrow \mathbb{W}$ and the equality $E_{\mathbb{W}}( T(u) )(z) = T( E_{\mathbb{V}}(u) )(z)$ for every $z \in Z$.
\end{proposition}

\section[Newton-Sobolev extension domains]{Newton--Sobolev extension sets}\label{Section_Sob.ext.dom.}

In this section, we prove a stronger variant of \Cref{thm:extensionresults}, valid for a larger class of metric measure spaces $Z$. Namely $Z$ does not have to be complete. We also assume that $\mu$ is doubling with a doubling constant $c_\mu$,  and the existence of a measurable set $\Omega \subset Z$ satisfying the measure-density condition with constant $c_\Omega$. 

Recall that in this case we have $\mu( \overline{\Omega} \setminus \Omega ) = 0$ as a consequence of the Lebesgue differentiation theorem.

\begin{definition}\label{eq:qp-PI}
Given $1 \leq q < p$ and a metric measure space $Z$. We say $Z$ has the \emph{$(q,p)$-PI property} if there exists a constant $c_{PI} > 0$ such that, for every $h \in W^{1,p}( Z )$ and $p$-weak upper gradient $\rho$ of (a representative of) $h$, we have
\begin{equation*}
    g_{r}(z) \coloneqq c_{PI}\left( \mathcal{M}_{2r}( \rho^{q} ) \right)^{1/q}(z)
    \in \mathcal{D}_{p}( h|_{ B(z,r) } )
    \quad\text{for every $z \in Z$ and $r > 0$.}
\end{equation*}
In other words, $g_r \in \mathcal{D}_p( h, r )$ for every $0<r<\infty$ and $g_\infty = \lim_{r \rightarrow\infty }g_r  = c_{PI}\left( \mathcal{M}( \rho^{q} ) \right)^{1/q}\in \mathcal{D}_p( h )$.
\end{definition}
Recall the definition of the restricted maximal function $\mathcal{M}_{2r}$ from \eqref{eq:maximalfunction:restr}. 

\begin{remark}\label{rem:selfimprovement}
If $p > 1$, every $p$-PI space has the $(q,p)$-PI property for some $q < p$. Indeed, a deep result by Keith and Zhong shows that every $p$-PI space for $p > 1$, is $q$-PI space for some $1 \leq q < p$ \cite{KZ2008}. For their conclusion to hold, it is important to recall that $p$-PI spaces are complete. Moreover, if $1 \leq q < p$, every $q$-PI space has the $(q,p)$-PI property. We refer the interested reader to \cite{M=W,HK2000,HKST2015} for further details.
\end{remark}

The goal of this section is to prove the following result which directly implies \Cref{thm:extensionresults}.

\begin{theorem}\label{thm:W1p:metric}
Suppose that $Z$ is a doubling metric measure space satisfying the $(q,p)$-PI property, with $\Omega \subset Z$ being a measurable  set having the measure-density condition and $\mathbb{V}$ an arbitrary Banach space. Then a given function $u \colon \Omega \rightarrow \mathbb{V}$ satisfies $u = h|_{\Omega}$ for some $h \in W^{1,p}( Z; \mathbb{V} )$ if and only if $u\in m^{1,p}(\Omega;\mathbb V)$. Moreover, if such an $h$ exists, then $u \in M^{1,p}( \Omega; \mathbb{V} )$.
\end{theorem}

We recall from \Cref{lemm:localHaj} that under the assumptions of \Cref{thm:W1p:metric}, $u \in m^{1,p}( \Omega; \mathbb{V} )$ if and only if $u \in L^{p}( \Omega; \mathbb{V} )$ and $u_{s}^{\sharp} \in L^{p}( \Omega )$ for some $s = s(u) > 0$. In the $\rr$-valued case, Shvartsman proved a similar result for $s=\fz$, see \cite[Theorem 1.2]{S2007}. As we can see, $u^\sharp_s$ is an increasing function of $s$, so our result implies the one by Shvartsman.

Before going into the proof of \Cref{thm:W1p:metric}, we establish further properties of the Banach-valued local Haj\l{}asz--Sobolev space $m^{1,p}$ introduced in \Cref{sec:locHaj}.

When $Z$ is a doubling metric measure space satisfying the $(q,p)$-PI property, it is well known that the spaces $M^{1,p}(Z;\mathbb V)$ and $W^{1,p}(Z;\mathbb V)$ coincide \cite{HKST2001,HKST2015}. Since $$M^{1,p}(Z;\mathbb V)\subset m^{1,p}(Z;\mathbb V)\subset W^{1,p}(Z;\mathbb V)$$ and 
$$\|u\|_{W^{1, p}(Z, \mathbb V)}\leq 4\|u\|_{m^{1, p}(Z, \mathbb V)}\leq 4\|u\|_{M^{1, p}(Z, \mathbb V)}$$
for every $u\in M^{1, p}(Z, \mathbb V)$, the three spaces are isomorphically equivalent, again by using the bounded inverse theorem. In fact, the following holds.

\begin{lemma}\label{prop:extensiondomain:Sobolev}
Let $Z$ is a doubling metric measure space satisfying the $(q,p)$-PI property for some $1\leq q<p$. Then for every $h \in W^{1,p}( Z; \mathbb{V} )$ and a $p$-weak upper gradient $\rho$ of (a representative of) $h$ the following holds:
\begin{enumerate}
    \item $g_r\in\mathcal D_p(h, r)$ for every $0<r<\infty$;
    \item $g_{\infty} \coloneqq \lim_{ r \rightarrow \infty } g_{r}$ satisfies $g_{\infty} \in \mathcal{D}_{p}( h )$;
    \item Moreover,
\begin{equation}\label{eq:Lp:inequalities}
    c_{PI}
    \| \rho\|_{ L^{p}(Z) }
    =
    \lim_{ r \rightarrow 0^{+} }\| g_{r} \|_{ L^{p}(Z) }
    \leq
    \| g_{r} \|_{ L^{p}(Z) } 
    \leq
    \| g_{\infty} \|_{ L^{p}(Z) }
    \leq
    C(c_\mu,p,q) c_{PI}\|  \rho\|_{ L^{p}(Z) }.
\end{equation}
\end{enumerate}
In particular, $M^{1,p}( Z; \mathbb{V} ) = m^{1,p}( Z; \mathbb{V} ) = W^{1,p}( Z; \mathbb{V} )$ as sets with bi-Lipschitz comparable norms.
\end{lemma}
\begin{proof}
The inclusion mapping $\iota_M$ from $M^{1,p}( Z; \mathbb{V} )$ into $m^{1,p}( Z; \mathbb{V} )$ is a $1$-Lipschitz linear embedding. Also, \Cref{lemma:m:into:N} implies that $m^{1,p}( Z; \mathbb{V} )$ $4$-Lipschitz linearly embedds into $W^{1,p}( Z; \mathbb{V} )$, the embedding denoted by $\iota_m$.

Next, we claim that every $h \in W^{1,p}( Z; \mathbb{V} )$ defines an element of $M^{1,p}( Z; \mathbb{V} )$. We also claim that $g_{r} \in \mathcal{D}_{p}( h, r )$ for every $0<r<\fz$ and $g_{\infty} \in \mathcal{D}_{p}( h )$. If this fact and \eqref{eq:Lp:inequalities} are proved, then \eqref{eq:Lp:inequalities} implies that $\iota_{M}$ is $\max\left\{1,C(c_\mu,p,q)\right\}$-bi-Lipschitz and $\iota_{m}$ is $\max\left\{4,c_{PI}\right\}$-bi-Lipschitz.

To this end, fix $h \in W^{1,p}( Z; \mathbb{V} )$ and $\rho \in \mathcal{D}_{N,p}( h )$ (for some representative of $h$ in the Newton--Sobolev space).  The first inequality in \eqref{eq:Lp:inequalities} is a consequence of the dominated convergence and the Lebesgue differentiation theorem, while the last one  follows from the boundedness of the maximal operator.

It remains to prove $g_{r} \in \mathcal{D}_{p}( h, r )$ and $g_{\infty} \in \mathcal{D}_{p}( h )$. It follows from \Cref{lemm:projection:new} that $ \rho \in \mathcal{D}_{N,p}( w(h) )$ for every $w \in \mathbb{V}^{*}$ with $|w| = 1$. Then the defining property of the constant $c_{PI}$ yields $g_{r} \in \mathcal{D}_{p}( w(h), r )$ and $g_{\infty} \in \mathcal{D}_{p}( w(h) )$ for every such $w$. The implication "(2) $\Rightarrow$ (1)" in \Cref{lemm:projection} yields $g_{r}\in \mathcal{D}_{p}(h, r)$ and $g_\fz \in \mathcal{D}_{p}( h )$. The claim follows.
\end{proof}

Another important point is that whenever $Z$ is a doubling metric measure space and $\Omega$ has the measure-density condition, the spaces $M^{1,p}(\Omega;\mathbb V)$ and $M^{1,p}( \overline{\Omega}; \mathbb{V} )$ (resp. $m^{1,p}(\Omega;\mathbb V)$ and  $m^{1,p}(\overline\Omega;\mathbb V)$) are isometrically equal, respectively, because $\mu( \overline{\Omega} \setminus \Omega )=0$.
\begin{lemma}\label{lemm:closure}
Let $Z$ be a doubling metric measure space, $\Omega\subset Z$ be a measurable set which satisfies the measure-density condition and $\mathbb V$ a Banach space. Then the restriction mapping $R \colon m^{1,p}( \overline \Omega; \mathbb{V} ) \rightarrow m^{1,p}( \Omega; \mathbb{V} )$ defined by $\overline{u} \mapsto \overline{u}|_{ \Omega }$ is an isometric isomorphism. Similarly, the restriction map is an isometric isomorphism from $M^{1,p}( \overline \Omega; \mathbb{V} )$ onto $M^{1,p}( \Omega; \mathbb{V} )$.
\end{lemma}
The proof of \Cref{lemm:closure} is a simple exercise and is left to the reader. In fact, $\mu( \overline{\Omega} \setminus \Omega ) = 0$ is sufficient for the conclusion to hold.

The following lemma connects the local and global Haj\l{}asz--Sobolev spaces for measurable subsets satisfying the measure-density condition.
\begin{lemma}\label{lemm:restriction}
Let $Z$ be a doubling metric measure space satisfying the $(q,p)$-PI property, $\Omega\subset Z$ be a measurable set with the measure-density condition and $\mathbb V$ a Banach space. Then $m^{1,p}(  \Omega; \mathbb{V} ) = M^{1,p}(  \Omega; \mathbb{V} )$ as sets.
\end{lemma}
\begin{proof}
Observe from \Cref{lemm:closure} that $m^{1,p}( \Omega; \mathbb{V} ) = m^{1,p}( \overline \Omega; \mathbb{V} )$ and $M^{1,p}( \Omega; \mathbb{V} ) = M^{1,p}( \overline \Omega; \mathbb{V} )$ as sets (using the identification from \Cref{lemm:closure}). Therefore it is enough to show that $m^{1,p}( \overline \Omega; \mathbb{V} ) = M^{1,p}( \overline \Omega; \mathbb{V} )$ as sets. In fact, it is enough to verify that each $u \in m^{1,p}( \overline{\Omega}; \mathbb{V} )$ is the restriction of some $h \in M^{1,p}( Z; \mathbb{V} )$.

Fix an element $u \in m^{1,p}( \overline \Omega; \mathbb{V} )$ and $g \in \mathcal{D}_{p}( u, r )$. We may assume that $r < 1$ and denote $R = r/8$ for convenience.

Consider the set $W = \bigcup_{ z \in \overline \Omega } B( z, R )\subset B (\overline \Omega,8)$. We denote $h = \widetilde{E}_{\mathbb{V}}(u)|_{W} \colon W\to \mathbb V$, where $\widetilde{E}_{\mathbb{V}}$ is the extension operator defined as in \eqref{eq:ext.op.}. \Cref{proposition:local} yields the existence of a constant $C=C(c_\mu,c_{\overline \Omega}) > 0$ for which $G \colon B( \overline{\Omega}, 8 ) \to \R$ defined as 
$$G \coloneqq C\mathcal{M}_{8R}(  \widehat{g} )$$
satisfies $G|_{W} \in \mathcal{D}_{p}(h, R)$. Indeed, since every $w \in W$ satisfies $B( w, R ) \subset B( w', 2R )$ for some $w' \in \overline \Omega$, the claim follows directly from \Cref{proposition:local}.

We consider next a Lipschitz function $\psi$ equal to one in $\overline \Omega$ and zero outside $W$. Then the zero extension $h'$ of $\psi h$ defines an element of $m^{1,p}( Z; \mathbb{V} )$ as a consequence of \Cref{lemm:basic}. \Cref{prop:extensiondomain:Sobolev} implies that $h' \in M^{1,p}( Z; \mathbb{V} )$ so $u = h'|_{\overline \Omega} \in M^{1,p}( \overline \Omega; \mathbb{V} )$. The claim follows.
\end{proof}

\begin{proof}[Proof of \Cref{thm:W1p:metric}]
Recall that $W^{1,p}( Z; \mathbb{V} ) = M^{1,p}( Z; \mathbb{V} )$ as sets by \Cref{prop:extensiondomain:Sobolev}.

Suppose that $u = h|_{\Omega}$ for some $h \in W^{1,p}( Z; \mathbb{V} )$. Then $u \in M^{1,p}( \Omega; \mathbb{V} ) \subset m^{1,p}( \Omega; \mathbb{V} )$ due to \Cref{prop:extensiondomain:Sobolev}.

Conversely, if $u \in m^{1,p}( \Omega; \mathbb{V} )$, then $u \in M^{1,p}( \Omega; \mathbb{V} )$ by \Cref{lemm:restriction}. Therefore $u$ admits an extension due to \Cref{th:Mexten}.
\end{proof}

\section[Banach property of the local Hajlasz-Sobolev space and some examples]{Banach property of the local Haj{\l}asz--Sobolev space and some examples}\label{Section_LAST} 
In this section we treat three topics. First we study the density of Lipschitz functions with bounded support in Sobolev spaces $W^{1,p}(\Omega)$, thereby proving \Cref{prop:density}. Then we investigate the completeness property of the local Haj{\l}asz--Sobolev space $m^{1,p}(\Omega;\mathbb V)$, and its relation to the density of Lipschitz functions and to Sobolev extension sets. Finally, we give some examples in the Euclidean space that clarify and highlight some subtleties of these properties.

\subsection{Density of Lipschitz functions}
\begin{definition}
Assume that $Z$ is a metric measure space and $\Omega \subset Z$ is measurable. Given $1 \leq p < \infty$, we say that $\Omega$ is \emph{$p$-path almost open} if, outside a $p$-negligible family of paths $\gamma \colon I_\gamma \to Z$, every nonconstant rectifiable $\gamma \colon [a,b] \rightarrow Z$ is such that $\gamma^{-1}( \Omega )$ is a union of an open set and a set negligible for the length measure of $\gamma$. 
\end{definition}
The main point for us about working with $p$-path almost open lies in the next lemma.

\begin{lemma}[Proposition 3.5, \cite{Bj:Bj:2015}]\label{lemm:stronglocality}
Let $Z$ be a metric measure space and $\Omega \subset Z$ a measurable $p$-path almost open set. If $h \in \widetilde{N}^{1,p}( Z; \mathbb{V} )$, then $u = h|_{\Omega} \in \widetilde{N}^{1,p}( \Omega; \mathbb{V} )$ and $\rho_{u} = \rho_{h}|_\Omega$.
\end{lemma}
The density of Lipschitz functions in $W^{1,p}(\Omega)$ is related to the existence of extensions from $W^{1,p}(\Omega)$ to $W^{1,p}(\overline \Omega)$, especially when one works with $p$-path almost open sets. As a related result, we establish now \Cref{prop:density}.

\begin{proof}[Proof of \Cref{prop:density}]
Recall that we are assuming that $Z$ is a $p$-PI space, in particular, a complete metric measure space. We are also assuming that $\Omega \subset Z$ is measurable and satisfies $\mu( \overline{\Omega} \setminus \Omega ) = 0$.

Claim (1) follows from the $1$-Lipschitz property of the restriction map from $W^{1,p}( \overline{\Omega} )$ into $W^{1,p}( \Omega )$ and the fact that Lipschitz functions with bounded support are dense in $W^{1,p}( \overline{\Omega} )$: the latter fact can be deduced in many ways from \cite{EB:Sou:21}. For example, $\overline{\Omega}$ has finite Hausdorff dimension, so \cite[Theorem 1.5]{EB:Sou:21} implies that $\overline{\Omega}$ has what the authors call a $p$-weak differentiable structure. Moreover, for every complete metric measure space, the existence of such a $p$-weak differentiable structure implies the claimed norm-density of Lipschitz functions (with bounded support) \cite[Theorem 1.9]{EB:Sou:21}.

For Claim (2), we know that every Lipschitz $u \colon \Omega \rightarrow \mathbb{R}$ with bounded support is the restriction of a Lipschitz function $h \colon \overline{\Omega} \to \mathbb{R}$ with bounded support, by uniform continuity. Then $\mu( \overline{\Omega} \setminus \Omega ) = 0$ and \Cref{lemm:stronglocality} guarantee $\| h \|_{ W^{1,p}( \overline{\Omega} ) } = \| u \|_{ W^{1,p}( \Omega ) }$. Thus, by the assumed energy-density and \Cref{lemm:stronglocality}, for every $u \in \widetilde{N}^{1,p}( \Omega )$, there exists a sequence $( h_n )_{ n = 1 }^{ \infty }$ of Lipschitz functions with bounded support defined in $\overline{\Omega}$ such that
\begin{equation*}
    \| h_n|_{\Omega} - u \|_{ L^{p}(\Omega) }
    +
    \| \rho_{ h_{n} }|_{ \Omega } - \rho_{ u } \|_{ L^{p}(\Omega) }
    \leq
    2^{-n}
    \quad\text{for every $n \in \mathbb{N}$.}
\end{equation*}
This implies that $( h_n )_{ n = 1 }^{ \infty }$ is a bounded sequence in $\widetilde{N}^{1,p}( \overline{\Omega} )$ converging to the zero extension of $u$ in $L^{p}( \overline{\Omega} )$, and the weak upper gradients $\rho_{ h_n }$ converging to the zero extension of $\rho_{u}$ in $L^{p}( \overline{\Omega} )$. Then \cite[Proposition 7.3.7]{HKST2015} guarantees that the zero extension of $u$ has a Newton--Sobolev representative $h$ such that a Borel representative of the zero extension of $\rho_u$ is a $p$-weak upper gradient of $h$. \Cref{lemm:stronglocality} readily yields that $h$ and $u$ have equal norm. The claim follows.
\end{proof}

\subsection[Banach property]{Banach property }\label{sec:Banach:extension}
In this section, we work with $p$-PI spaces for $p > 1$. Such spaces have the $(q,p)$-property, cf. \Cref{rem:selfimprovement}, so the results of \Cref{Section_Sob.ext.dom.} are applicable.

We recall from \Cref{thm:W1p:metric} that the set of functions $u\in W^{1,p}(\Omega;\mathbb V)$ that can be extended to $W^{1,p}(Z;\mathbb V)$ is indeed a subspace isomorphic to $m^{1,p}(\Omega;\mathbb V)$ when $\Omega$ satisfies the measure-density condition and $Z$ is $p$-PI. If we add the further assumption of $p$-path almost openness, we obtain the following theorem, whose proof we leave for the end of this subsection.

\begin{theorem}\label{thm:V-extension:PI}
Assume that $Z$ is $p$-PI space for $\infty > p > 1$. Suppose that $\Omega \subset Z$ is a measurable $p$-path almost open set satisfying the measure-density condition. The following are equivalent:
\begin{enumerate}
    \item the $4$-Lipschitz linear embedding $M^{1,p}( \Omega; \mathbb{V} ) \subset W^{1,p}( \Omega; \mathbb{V} )$ has closed image.
    \item $m^{1,p}( \Omega; \mathbb{V} )$ is Banach.
\end{enumerate}
Moreover, if $\mathbb{V}$ has the metric approximation property, the following are equivalent:
\begin{enumerate}
    \setcounter{enumi}{2}
    \item $\Omega$ is a $\mathbb{V}$-valued $W^{1,p}$-extension set;
    \item $m^{1,p}( \Omega; \mathbb{V} )$ is Banach and Lipschitz functions with bounded support are dense in energy in $W^{1,p}( \Omega )$.
\end{enumerate}
Under assumption (3), $M^{1,p}(\Omega;\mathbb V)=m^{1,p}( \Omega; \mathbb{V} ) = W^{1,p}( \Omega; \mathbb{V} )$ as sets with comparable norms.
\end{theorem}
\begin{remark}\label{rem:necessaryandsufficient}
By applying \Cref{prop:density}, the energy-density of Lipschitz functions with bounded support in (4) can be equivalently replaced by density in norm. Furthermore, the norm-density of Lipschitz functions with bounded support is true in $M^{1,p}( \Omega )$, cf. \cite{M=W,HKST2015}, and if $M^{1,p}( \Omega ) = W^{1,p}( \Omega )$ as sets, then the bounded inverse theorem yields the norm-density in $W^{1,p}(\Omega)$. Therefore, to solve \Cref{ques:startingpoint} positively in the setting of \Cref{thm:V-extension:PI}, it is necessary and sufficient to prove that $m^{1,p}( \Omega; c_0 )$ is Banach whenever $\Omega$ is a $W^{1,p}$-extension set.
\end{remark}
\begin{corollary}\label{cor:completion:PI}
Assume that $Z$ is $p$-PI space for $p > 1$. Suppose that $\Omega \subset Z$ is a measurable $p$-path almost open set satisfying the measure-density condition.

Then $\Omega$ is a $W^{1,p}$-extension set if and only if $m^{1,p}( \Omega )$ is a Banach space and every $u \in W^{1,p}( \Omega )$ is the restriction of some $h \in W^{1,p}( \overline{\Omega} )$.
\end{corollary}
\begin{proof}
In both directions of the claim, we have that every $u \in W^{1,p}( \Omega )$ is the restriction of some $h \in W^{1,p}( \overline{\Omega} )$. Equivalently, Lipschitz functions with bounded support are norm-dense by \Cref{prop:density}. Thus the $4$-Lipschitz inclusion $m^{1,p}( \Omega ) \subset W^{1,p}( \Omega )$ has a norm-dense image. The equivalence between (3) and (4) in \Cref{thm:V-extension:PI} shows the claim.
\end{proof}

We apply the following lemma during the proof of \Cref{thm:V-extension:PI}.
\begin{lemma}\label{prop:extensiondomain:PI}
Assume that $Z$ is $p$-PI space for $p > 1$. Suppose that $\Omega \subset Z$ is a measurable $p$-path almost open set satisfying the measure-density condition.

Then every $u \in m^{1,p}( \Omega; \mathbb{V} )$ satisfies
\begin{equation*}
    4^{-1}\| u \|_{ W^{1,p}( \Omega; \mathbb{V} ) }
    \leq
    \| u \|_{ m^{1,p}( \Omega; \mathbb{V} ) }
    \leq
    c\| u \|_{ W^{1,p}( \Omega; \mathbb{V} ) },
\end{equation*}
for the constant $c = \max\left\{ 1, c_{PI} \right\}$.
\end{lemma}
\begin{proof}
We recall from \Cref{thm:W1p:metric} that $m^{1,p}( \Omega; \mathbb{V} )$ are precisely those functions that admit an extension to $h \in W^{1,p}( Z; \mathbb{V} )$.

Let $h$ be such an extension and $g \in \mathcal{D}_{N,p}( h )$. We consider the restricted maximal function of $g$ multiplied by a suitable factor $c_{PI}$, which we used to define $g_r$ in \Cref{prop:extensiondomain:Sobolev}. We recall from \Cref{prop:extensiondomain:Sobolev} that $g_r$ is a local Haj{\l}asz gradient of $h$ up to scale $r$. Then \Cref{lemm:restriction} yields the lower bound $4^{-1}\| u \|_{ W^{1,p}( \Omega; \mathbb{V} ) } \leq \| u \|_{ m^{1,p}( \Omega; \mathbb{V} ) }$.

The Lebesgue differentiation theorem implies that $\chi_{ \Omega }g_{r} \rightarrow \chi_{\Omega} c_{PI} g = c_{PI} \chi_{ \Omega } g$, monotonically decreasingly pointwise almost everywhere, as $r \rightarrow 0^{+}$. The dominated convergence theorem implies that the pointwise convergence improves to $L^{p}$-convergence. Hence 
\begin{equation*}
    \inf_{ g' \in  \mathcal{D}_{\loc, p}( u ) }
    \| g' \|_{ L^{p}( \Omega ) }
    \leq
    \lim_{ r \rightarrow 0^{+} }
    \| g_{r}|_{\Omega} \|_{ L^{p}( \Omega ) }
    =
    c_{PI}
    \| g|_{\Omega} \|_{ L^{p}( \Omega ) }.
\end{equation*}
Next we set $g = \rho_{h}$. Here $\rho_{h}|_{\Omega} = \rho_{u}$ follows from \Cref{lemm:stronglocality}. The claim follows from this observation.
\end{proof}
\begin{proof}[Proof of \Cref{thm:V-extension:PI}]
\Cref{lemm:restriction} shows that when $\Omega$ is measurable and satisfies the measure-density condition, the $1$-Lipschitz linear inclusion $M^{1,p}( \Omega; \mathbb{V} ) \subset m^{1,p}( \Omega; \mathbb{V} )$ is onto. On the other hand, \Cref{prop:extensiondomain:PI} shows that the embedding of $m^{1,p}( \Omega; \mathbb{V} )$ into $W^{1,p}( \Omega; \mathbb{V} )$ is bi-Lipschitz. Consequently, $M^{1,p}( \Omega; \mathbb{V} ) \subset W^{1,p}( \Omega; \mathbb{V} )$ has a closed image if and only if $m^{1,p}( \Omega; \mathbb{V} )$ is Banach, so "(1) $\Leftrightarrow$ (2)".

Towards proving "(3) $\Leftrightarrow$ (4)", we make three observations: First, by \Cref{prop:density}, the energy-density of Lipschitz functions with bounded support is equivalent to norm-density. Second, \Cref{prop:coincidence} shows that $M^{1,p}( \Omega; \mathbb{V} ) = W^{1,p}( \Omega; \mathbb{V} )$ as sets if and only if Lipschitz functions are norm-dense in $W^{1,p}( \Omega )$ and the linear $4$-embedding $M^{1,p}( \Omega; \mathbb{V} ) \subset W^{1,p}( \Omega; \mathbb{V} )$ has closed image, the latter fact being equivalent to $m^{1,p}( \Omega; \mathbb{V} )$ being Banach. Third, \Cref{cor:measuredensity:Sob} shows that $\Omega$ is a $\mathbb{V}$-valued $W^{1,p}$-extension set if and only if $M^{1,p}( \Omega; \mathbb{V} ) = W^{1,p}( \Omega; \mathbb{V} )$ as sets. These three observations show the equivalence.
\end{proof}

\subsection{Some examples}

Related to the discussion of this section and also to the topics of this paper, it is clear that it would be interesting to know the relations between the following properties, stated for $\R$-valued functions, a domain $\Omega\subset\R^n$ and $1 < p<\infty$. Note that domains are always $p$-path almost open sets.
\begin{enumerate}
    \item $\Omega$ is a $W^{1,p}$-extension domain.
    \item $\Omega$ has the measure-density condition.
    \item $M^{1,p}(\Omega)=m^{1,p}(\Omega)=W^{1,p}(\Omega)$ as sets and with comparable norms.
    \item $m^{1,p}(\Omega)$ is a Banach space.
     \item Every $u \in W^{1,p}(\Omega)$ is the restriction of some $h \in W^{1,p}( \overline{\Omega} )$.
    \item Density of Lipschitz functions with bounded support in $W^{1,p}(\Omega)$.
    \item Density of $m^{1,p}(\Omega)$ in $W^{1,p}(\Omega)$.
\end{enumerate}
In this section, we provide examples to show the relations between the aforementioned properties and at the same time bring out the sharpness of the previous results of this section. The following implications are clear using our previous results.

\begin{align*}
&(1)\Rightarrow (2),(3),(4),(5),(6),(7); \quad \quad &&(5)\Rightarrow (6) \Rightarrow (7);  \\
&(7) \text{ and }\mu( \partial \Omega ) = 0 \Rightarrow (5); \quad\quad
&&(2) + (3) \Leftrightarrow (1); &&&(2)+(4)+(7) \Leftrightarrow (1).
\end{align*}
Our first example shows (2) is not sufficient to guarantee $W^{1, p}$-extensions, the result being well-known among specialists in the field. We recall that by \cite[Theorem 5]{HKT2008:B}, a domain has the measure-density condition if and only if it is a $M^{1, p}$-extension domain for some (every) $1\leq p<\infty$.
\begin{example}\label{ex:slitdisk}
The slit disk
$$  \Omega := B(0,1)\setminus [0,1] \times \{0\} .$$
This is well-known to not be a $W^{1,p}$-extension domain for any $1< p<\infty$. For example, we may consider $u \in C( \Omega ) \cap W^{1,p}( \Omega )$ such that $u \equiv 1$ above $[1/2,1] \times \left\{0\right\}$ and $u \equiv 0$ below it. Then $u^{\sharp}_s \not\in L^{p}(\Omega)$ for every $p > 1$, which is necessary and sufficient for $u$ to admit an extension to $W^{1,p}( \Omega )$, recall \Cref{thm:W1p:metric}.

Note that $\Omega$ satisfies the measure-density condition. Moreover, since $m^{1,p}(\Omega)=m^{1,p}(\overline{\Omega})$ as sets by \Cref{lemm:closure} and also $\overline{\Omega}=\overline B (0,1)$ is a $W^{1,p}$-extension domain for every $p$, the space $m^{1,p}(\Omega)$ is Banach. On the other hand, neither Lipschitz functions nor $m^{1,p}( \Omega )$-functions can be (energy or norm) dense in $W^{1,p}(\Omega)$ for any $1 < p < \infty$, due to \Cref{cor:completion:PI}.
\end{example}
Our second example shows that without the measure-density condition (2), none of the conditions (3) to (7) are enough to guarantee (1).
\begin{example}
The outward cusp
$$ \Omega= B((2,0),\sqrt{2})\cup \{(x,y)\in (0,1]\times \R:\, |y|\leq x^2  \}$$
is another example of a domain that is not $W^{1,p}$-extension domain for any $1< p<\infty$. Observe that the measure-density condition fails at the cusp $(2,0)$. Nevertheless, in \cite{BKZheng,Zheng}, it was shown that $M^{1,p}(\Omega)=W^{1,p}(\Omega)$, so automatically we have $M^{1,p}(\Omega)=m^{1,p}(\Omega)=W^{1,p}(\Omega) $ as sets. Thus $m^{1,p}(\Omega)$ is Banach and Lipschitz functions, as well as the space $m^{1,p}(\Omega)$, are dense in $W^{1,p}(\Omega)$.
\end{example}

Our next example shows that a domain can satisfy (2) and (6) (thus also (7) and (5)) without (1) or (3) holding.
\begin{example}\label{ex:3}
Consider the unit square where we remove two central squares touching in one vertex. That is,
$$\Omega= ((-1,1)\times (-1,1)) \setminus \left( ([-1/2,0]\times [-1/2,0]) \cup ([0,1/2]\times [0,1/2] ) \right).$$
This is based on Example 2.5 \cite{Koskela}, see also \cite[Remark 5.2]{Bjo:Sha:07}. It turns out that $\Omega$ is a $W^{1,p}$-extension domain for every $1 < p < 2$ but not for any $p \geq 2$. For the proof of the case $1 < p < 2$ we refer the interested reader to \cite[Example 2.5]{Koskela}. On the other hand, a minor modification of the function $u$ in \Cref{ex:slitdisk} yields an example of $u \in W^{1,p}( \Omega )$ for which $u_{s}^{\sharp}$ is not $L^p(\Omega)$-integrable for any $(s,p) \in (0,\infty) \times [2, \infty]$. Thus $\Omega$ is not a $W^{1,p}$-extension set for any $p \geq 2$.

We consider next an arbitrary $1 < p \leq 2$. We claim that every $u \in W^{1,p}( \Omega )$ extends to some $h \in W^{1,p}( \overline{\Omega} )$; the point is that \begin{equation*}
    h(z) = \limsup_{ r \rightarrow 0^{+} } \aint{ B(z,r) \cap \Omega } u(y) \,dy
\end{equation*}
defines an extension to $\widetilde{N}^{1,p}( \overline{\Omega} \setminus \left\{0\right\} )$. Since $\left\{0\right\}$ has negligible measure and the collection of nonconstant paths intersecting the point is $p$-negligible (in $\mathbb{R}^2$ so also in $\overline{\Omega}$), extending $h$ as $0$ to $\overline{\Omega}$ defines an element of $\widetilde{N}^{1,p}( \overline{\Omega} )$. Then recalling \Cref{prop:density}, we have that Lipschitz functions are dense in $W^{1,p}( \Omega )$. Since $\Omega$ is not a $W^{1,2}$-extension domain, \Cref{thm:V-extension:PI} shows that the space $m^{1,2}( \Omega )$ cannot be Banach. Equivalently, $M^{1,2}( \Omega ) \neq W^{1,2}( \Omega )$. In fact, there exists a sequence of Lipschitz functions in $\Omega$ bounded in $W^{1,2}( \Omega )$ but unbounded in $M^{1,2}( \Omega )$, e.g., we may approximate any $u \in W^{1,2}( \Omega )$ for which $u^{\sharp}_1 \not\in L^{2}(\Omega)$ to obtain such a sequence.
\end{example}

In the following example, we recall results that guarantee the validity of (6) for many domains in the plane.
\begin{example}\label{ex:Jordan}
We consider an arbitrary Jordan domain $\Omega \subset \rr^2$, i.e., a simply connected domain whose boundary is homeomorphic to the unit circle. Then $W^{1,p}( \Omega )$ has a dense subspace consisting of those $u = h|_{\Omega}$, where $h \in C^{\infty}( \rr^2 )$; see \cite{Lew:87,Kos:Zh:16}. In particular, Lipschitz functions, so, in particular, $m^{1,p}( \Omega )$ is dense in $W^{1,p}( \Omega )$. 

Interestingly, whenever $\Omega \subset \rr^n$ is any bounded domain $\delta$-hyperbolic with respect to the quasihyperbolic metric, then $W^{1,\infty}( \Omega )$ is dense in $W^{1,p}( \Omega )$, see \cite{Ko:Ra:Zh:17} for the terminology and proof. In particular, when $\Omega$ is also quasiconvex, then Lipschitz functions with bounded support are dense in $W^{1,p}( \Omega )$.
\end{example}

\section[Weak Poincaré inequalities]{Weak Poincaré inequalities}\label{section:Corollary}
Our goal is to verify \Cref{thm:positive} and \Cref{thm:extension:topoincare} in this section, which connect the validity of the weak Poincaré inequality up to some scale to Banach-valued Sobolev extension domains. Proving \Cref{thm:positive} is equivalent to showing that $M^{1,p}(\Omega;\mathbb V)=m^{1,p}(\Omega;\mathbb V)=W^{1,p}(\Omega;\mathbb V)$ as sets under the measure-density condition by \Cref{cor:measuredensity:Sob}.

Related to \Cref{thm:positive}, we begin by commenting on some previous known results. First, for $p$-PI spaces, it is known that $M^{1,p}( Z; \mathbb{V} ) = m^{1,p}(Z;\mathbb V)=W^{1,p}(Z;\mathbb V)$ as sets, cf. \Cref{Section_Sob.ext.dom.}. Many of the proofs showing this fact involve the self-improvement of weak Poincaré inequalities on (complete) $p$-PI spaces. Similar improvement is quite delicate when we relax the completeness of $Z$, cf. \cite{Kos:99}, and this is one of the main subtleties when addressing \Cref{thm:positive}. Secondly, by \cite[Proposition 5.1]{Bjo:Sha:07} every domain  $\Omega\subset Z$ with the measure-density condition and that satisfies a global weak $(1,p)$-Poincaré inequality must be a $W^{1,p}$-extension set. So considering just weak Poincaré inequalities up to some scale is the main novelty here, apart from working with general measurable sets. We also note that this section is partly inspired by \cite{Bj:Bj:Lah:21}.

In this section, we consider an arbitrary  $p$-PI space $Z$ and a measurable subset $\Omega \subset Z$ satisfying the measure-density condition. To simplify notation we denote $C_{\mu,\Omega}=\max\{c_\mu,1/c_\Omega\}$.

The following self-improvement phenomenom of Poincaré inequalities is an immediate consequence of \cite[Theorem 5.1]{Bj:Bj:2019} and will be needed for the proof of \Cref{thm:positive}.
\begin{theorem}\label{thm:selfimprovement}
Let $1 < p < \infty$. The following are equivalent under the assumptions of this section.
\begin{enumerate}
    \item There exist constants $C_1 > 0$, $\lambda_1 \geq 1$ and $2/3 > \delta_1 > 0$ such that for every ball $B$ centered at $\Omega$ and radius at most $\delta_1$, 
    \begin{equation}\label{eq:PI-relative}
        \aint{ \Omega \cap B } |u-u_{\Omega \cap B} |\,d\mu
        \leq
        C_1\diam (\Omega \cap B)
        \left(\aint{\Omega \cap \lambda_1 B} \rho^p\,d\mu  \right)^{1/p}
    \end{equation}
    for every integrable function $u \colon \Omega \cap \lambda_{1} B \to \R$ and every $L^{p}$-integrable $p$-weak upper gradient $\rho \colon \Omega \cap \lambda_{1} B \rightarrow [0,\infty]$ of $u$.
    \item There exist constants $1 \leq q < p$, $C_2 > 0$, $\lambda_2 \geq 1$ and $2/3 > \delta_2 > 0$ such that for every ball $B$ centered at $\Omega$ and radius at most $\delta_2$, 
    \begin{equation*}
        \aint{ \Omega \cap B } |u-u_{\Omega \cap B} |\,d\mu
        \leq
        C_2\diam (\Omega \cap B)
        \left(\aint{\Omega \cap \lambda_2 B} \rho^q\,d\mu  \right)^{1/q}
    \end{equation*}
    for every integrable function $u \colon \Omega \cap \lambda_{2} B \to \R$ and every $L^{p}$-integrable $p$-weak upper gradient $\rho \colon \Omega \cap \lambda_{2} B \rightarrow [0,\infty]$ of $u$. 
\end{enumerate}
Moreover, the constants $q$, $C_2$, $\lambda_2$ and $\delta_2$ depend only on $p$ and $C_{\mu,\Omega}$, together with $C_1$, $\lambda_1$ and $\delta_1$, and vice versa.
\end{theorem}
We emphasize that the validity of \Cref{thm:selfimprovement} (2) only for $L^{p}$-integrable $p$-weak upper gradients is important and is the main place where the noncompleteness of $\Omega$ plays a significant role, cf. \cite[Corollary D]{Kos:99}.

\begin{proof}[Proof of \Cref{thm:positive}]
We recall that we are given a $p$-PI space $Z$ and a measurable set $\Omega \subset Z$ satisfying the measure-density condition and a weak $(1,p)$-Poincaré inequality up to some scale $r_0$. That is, we are in the setting of \Cref{thm:selfimprovement} and assumption (1) for some constants $ C_1>0$, $\lambda_1\geq 1$ and $0<\delta_1<\min\{r_0,2/3\}$. We claim that $\Omega$ is a $\mathbb{V}$-valued $W^{1,p}$-extension set.

We fix a real-valued $f \in W^{1,p}( \Omega)$ and a Newton--Sobolev representative of the function, which we continue denoting by $f$. Let also $\rho\in \mathcal D_{N,p}(f)$. \Cref{thm:selfimprovement} (2) yields the existence of $1 \leq q < p$, $\delta_2 > 0$, $C_2 > 0$ and $\lambda_2 \geq  1$ such that 
\begin{equation*}
    \aint{ \Omega \cap B } |f-f_{\Omega \cap B} |\,d\mu
    \leq
    C_2\diam (\Omega \cap B)
    \left(\aint{\Omega \cap \lambda_2 B} \rho^q\,d\mu  \right)^{1/q}
\end{equation*}
for every ball centered at $\Omega$ and of radius $r < \delta_2$. Then, by fixing $0 < s < \delta_2/\lambda_2$ (hence $\lambda_2 s<\delta_2<1$), we may take the supremum over $0< r < s$ and apply \Cref{lemm:maximalfunction:bounded} to conclude
\begin{equation}\label{eq:prop_1,9}
    f_{s}^{\sharp}
    \leq
    4 C_2 C_{\mu,\Omega}^{1/p}
    \left(
        \mathcal{M}_{ \lambda_2 s }( \widehat{\rho}^q )
    \right)^{1/q}
    \in
    L^{p}( \Omega ).
\end{equation}

Now we use observation \eqref{eq:prop_1,9} to conclude the proof by verifying that every $u \in W^{1,p}( \Omega; \mathbb{V} )$ is an element of $m^{1,p}( \Omega; \mathbb{V} )$. Note that for every $w \in \mathbb{V}^{*}$ with $|w| \leq 1$, we have $\rho_{u}\in \mathcal D_{N,p}(w(u))$ by \Cref{lemm:projection:new}. Therefore, an application of \eqref{eq:prop_1,9} implies that for every $0 < s < \delta_2/\lambda_2$,
\begin{equation}\label{eq:almostconclusion}
    ( w( u ) )_{s}^{\sharp}
    \leq
    4 C_2 C_{\mu,\Omega}
    \left(
        \mathcal{M}_{ \lambda_2 s }( \widehat{\rho_u}^q )
    \right)^{1/q}
    \in
    L^{p}( \Omega ).
\end{equation}
Recall from \Cref{lemm:localHaj} that a constant multiple of $( w( u ) )_{s}^{\sharp}$ is a Haj\l{}asz gradient of $w(u)$ up to scale $s/2$. Then  \Cref{lemm:projection} implies that the right-hand side of \eqref{eq:almostconclusion} is a Haj\l{}asz gradient of $u$ up to scale $s/2$ and up to the aforementioned constant. Therefore $u \in m^{1,p}( \Omega; \mathbb{V} )$. Since $M^{1,p}( \Omega; \mathbb{V} ) = m^{1,p}( \Omega; \mathbb{V} ) \subset W^{1,p}( \Omega; \mathbb{V} )$ by \Cref{lemm:restriction}, we conclude $M^{1,p}( \Omega; \mathbb{V} ) = W^{1,p}( \Omega; \mathbb{V} )$. Finally, given this equality, the fact that $\Omega$ is a $\mathbb{V}$-valued extension set follows from \Cref{cor:measuredensity:Sob}. 
\end{proof}
Towards proving \Cref{thm:extension:topoincare}, we establish the following proposition, closely related to \Cref{thm:selfimprovement}.
\begin{proposition}\label{prop:Poincare:sharpfunctional}
Let $1 < p < \infty$. The following are equivalent under the assumptions of this section.
\begin{enumerate}
    \item There exist constants $C_1 > 0$, $\lambda_1 \geq 1$ and $2/3 > \delta_1 > 0$ such that for every ball $B$ centered at $\Omega$ and radius at most $\delta_1$, 
    \begin{equation}\label{eq:PI-relative:prop}
        \aint{ \Omega \cap B } |u-u_{\Omega \cap B} |\,d\mu
        \leq
        C_1\diam (\Omega \cap B)
        \left(\aint{\Omega \cap \lambda_1 B} \rho^p\,d\mu  \right)^{1/p}
    \end{equation}
    for every integrable function $u \colon \Omega \cap \lambda_{1} B \to \R$ and every $L^{p}$-integrable $p$-weak upper gradient $\rho \colon \Omega \cap \lambda_{1} B \rightarrow [0,\infty]$ of $u$.
    \item There exist constants $C_2 > 0$, $\lambda_2 \geq 1$ and $2/3 > \delta_2 > 0$ such that every $u \in L^{1}( \Omega )$ with an $L^{p}$-integrable $p$-weak upper gradient $\rho \colon \Omega \rightarrow [0, \infty]$ satisfies
    \begin{equation}
        u_{s}^{\sharp}
        \leq
        C_2\left( \mathcal{M}_{ \lambda_2 s }( \widehat{\rho}^{p} ) \right)^{1/p}
        \quad\text{almost everywhere in $\Omega$ for every $0 < s < \delta_2$.}
    \end{equation}
    Here $\widehat \rho$ refers to the zero extension of $\rho$ into $Z$.
\end{enumerate}
The statements are quantitatively equivalent in the sense that $C_1$, $\lambda_1$ and $\delta_1$ determine $C_2$, $\lambda_2$ and $\delta_2$ up to a constant depending only on $C_{\mu,\Omega}$ and $p$ and vice versa.
\end{proposition}
\begin{proof}
We assume (1) and consider $u \in L^{1}( \Omega )$ with an $L^{p}$-integrable $p$-weak upper gradient $\rho$. Fix $0 < s < \delta_1$ and $z \in \Omega$. Then, by taking the supremum in \eqref{eq:PI-relative:prop} over $B( z, r )$ and $0 < r < s$, the definition of $u_{s}^{\sharp}$ yields
\begin{equation*}
    u_{s}^{\sharp}(z)
    \leq
    4 C_1 C_{\mu,\Omega}^{1/p}\left( \mathcal{M}_{ \lambda_1 s }( \widehat{\rho}^{p} ) \right)^{1/p}(z).
\end{equation*}
The conclusion (2) follows by setting $C_2 = 4 C_1 C_{\mu,\Omega}^{1/p}$, $\lambda_2 = \lambda_1$ and $\delta_2 = \delta_1$.

We assume now (2). Fix constants $C_1>0$, $2/3 >\delta_1 > 0$, $\lambda_1\geq 1 $ yet to be determined. Fix now $z \in \Omega$ and $0 < r < \delta_1$ and set $B = B( z, r )$. 

Suppose for a moment that we consider $u \in L^1(\Omega\cap  B)$ and a Lipschitz function $\phi(y) = f( d( y, z ) )$ for a Lipschitz $f \colon [0,\infty) \rightarrow [0,1]$ equal to one in $[0, r/4]$, zero in $[r/2, \infty)$ and affine in $( r/4, r/2)$. In particular, $\phi$ is $( r /4  )^{-1}$-Lipschitz. Consider next $\phi u: \Omega \cap B\to \R $ and let $\widetilde{u} = \widehat{\phi u}$ be the zero extension to $\Omega$. Then $\widetilde{u} \in L^{1}( \Omega )$. We conclude from \Cref{lemm:localHaj} that there exists a constant $C \geq 1$, depending only on $C_{\mu,\Omega}$, for which
\begin{equation}\label{eq:conclusion2}
    C \widetilde{u}^{\sharp}_{2s} \in \mathcal{D}( \widetilde{u}, s )
    \quad\text{whenever $0 < s < 1/2$.}
\end{equation}
In order to prove (1), take $u\in L^1(\Omega\cap B)$ with an $L^{p}$-integrable $p$-weak upper gradient $\rho \colon \Omega\cap B\to [0,\infty]$. Then $u_k = \max\left\{ -k, \min\left\{k, u\right\} \right\}\in L^p(\Omega\cap B)$ has $\rho$ as an $L^{p}$-integrable $p$-weak upper gradient for every $k \in \mathbb{N}$ by \Cref{lemm:basic:new} (1). Consequently, $\phi u_k$ has a $p$-weak upper gradient
\begin{equation*}
    \rho_{k}=\left(|u_k| \frac{2}{r} + \rho \right)\chi_{B(z,r/2)} \in L^{p}( \Omega \cap B ),
    \quad\text{cf. \Cref{lemm:basic:new} (3).}
\end{equation*}
Also, the zero extension $\widetilde{u}_k=\widehat{\phi u_k}\in L^1(\Omega)$ has an $L^p$-integrable $p$-weak upper gradient coinciding with the zero extension of $\rho_{k} $ to $\Omega$. We may therefore apply assumption (2) for each $\widetilde{u}_k$ and their minimal $p$-weak upper gradients $\rho_{ \widetilde{u}_k }$. We may also apply \eqref{eq:conclusion2}. Therefore there exists a negligible set $N \subset \Omega \cap B$ such that, for every $x, y \in \Omega \cap B \setminus N$ with $d(x,y) < s$ and every $k \in \mathbb{N}$,
\begin{equation*}
    | \widetilde{u}_k(x) - \widetilde{u}_k(y) |
    \leq
    d( x, y )(
        C C_2 \left( \mathcal{M}_{ 2\lambda_2 s }( \widehat{\rho_{ \widetilde{u}_k }}^{p} ) \right)^{1/p}(x)
        +  
        C C_2 \left( \mathcal{M}_{ 2\lambda_2 s }( \widehat{\rho_{ \widetilde{u}_k }}^{p} ) \right)^{1/p}(y)
    ).
\end{equation*}
Here we need the further assumption that $2 s < \delta_2$.

We set $\lambda_1 = 4( 2 \lambda_2 + 1 )$ and $s = r/\lambda_1$. Whenever $x, y \in \Omega \cap B( z, s ) \setminus N$, we have $\Omega \cap B( x, 2 \lambda_2 s ) \subset \Omega \cap B( z, r/4 )$. For this reason, the inequality
\begin{equation*}
    \mathcal{M}_{ 2\lambda_2 s }( \widehat{\rho_{ \widetilde{u}_k }}^{p} )(x)
    \leq
    \mathcal{M}_{ 2\lambda_2 s }( \widehat{\rho_{ \widetilde{u} }}^{p} )(x)
    =
    \mathcal{M}_{ 2\lambda_2 s }( \widehat{\rho_{ u }}^{p} )(x),
\end{equation*}
holds for $x$, with the corresponding inequality valid for $y$ as well; the last equality applies the locality, namely \Cref{lemm:basic:new} (4), of minimal $p$-weak upper gradients. Hence, as $\lim_{k\to \infty}\widetilde u_k(x)=u(x)$ for almost every $x\in B(z,s)$, we obtain
\begin{equation*}
     C C_2 \left( \mathcal{M}_{ 2\lambda_2 s }( \widehat{\rho_{ u }}^{p} ) \right)^{1/p}
     \in
     \mathcal{D}( u|_{ B( z, s ) } ).
\end{equation*}
Consequently,
\begin{align*}
    \aint{ \Omega \cap B( z, s ) } |u-u_{\Omega \cap B(z,s)} |\,d\mu
    &\leq
    4 C C_2 
    \diam( \Omega \cap B( z, s ) ) 
    \aint{ \Omega \cap B( z, s) }
        \left( \mathcal{M}_{ 2\lambda_2 s }( \widehat{\rho_{ u }}^{p} ) \right)^{1/p}
    \,d\mu \\
    &\leq
    4 C C_2  C( C_{\mu,\Omega}, \lambda_2 )
    \diam( \Omega \cap B( z, s ) ) 
    \aint{ B(z, \lambda_2 s) }
        \left( \mathcal{M}_{ 2\lambda_2 s }( \widehat{\rho_{ u }}^{p} ) \right)^{1/p}
    \,d\mu.
\end{align*}
 Since $\mu$ is doubling on $Z$, a weak-type estimate \cite[Lemma 3.5.10]{HKST2015} for the maximal operator yields that
\begin{align*}
    \aint{ B(z, \lambda_2 s) }
        \left( \mathcal{M}_{ 2\lambda_2 s }( \widehat{\rho_{ u }}^{p} ) \right)^{1/p}
    \,d\mu
    &\leq
    C( C_{\mu,\Omega}, p )
    \left(
        \aint{ B(z,3 \lambda_2 s) }
          \widehat{\rho_{u}}^{p}
        \,d\mu
    \right)^{1/p}.
\end{align*}
Since $3 \lambda_2 s < 1$, we obtain \eqref{eq:PI-relative:prop} for each $0 < s < \delta_1$ for a constant $C_1 = C( C_{\mu,\Omega}, \lambda_2, p, C_2 )$, $\lambda_1 = 4( 2 \lambda_2 + 1 )$, and $\delta_1 = \delta_2 / \lambda_1$.
\end{proof}

\begin{proof}[Proof of \Cref{thm:extension:topoincare}]
We first consider the case $n<p<\infty $. Let us verify that the Sobolev embedding theorem is satisfied in $W^{1,p}(\Omega)$. We know that every $u \in W^{1,p}( \Omega )$ satisfies $u = h|_{\Omega}$ for some $h \in W^{1,p}( \mathbb{R}^n )$. By the Morrey's embedding theorem, there exists a continuous representative of $h$, still denoted the same way, so that for $\alpha = 1 - n/p$,
\begin{equation*}
    \| h \|_{ C^{0,\alpha}( \mathbb{R}^n ) }
    \coloneqq
    \| h \|_{L^{\infty}(\R^n) }
    +
    \sup_{ x \neq y }
        \frac{ |h(x)-h(y)| }{ |x-y|^{\alpha} }
    \leq
    C( p, n )\| h \|_{ W^{1,p}( \mathbb{R}^n ) }.
\end{equation*}
Hence $u \in W^{1,p}( \Omega )$ has a representative $h|_{\Omega}$, that we identify with $u$, for which
\begin{equation*}
    \| u \|_{ C^{0,\alpha}( \Omega ) }\leq \|h\|_{ C^{0,\alpha}( \Omega )}
    \leq C(p,n)\| h \|_{ W^{1,p}( \mathbb{R}^n ) }\leq
    C( p, n, \Omega ) \| u \|_{ W^{1,p}( \Omega ) }.
\end{equation*}
This implies the existence of a linear and bounded embedding of $W^{1,p}( \Omega )$ into $C^{0,\alpha}( \Omega )$. The latter space can be identified with $C^{0,\alpha}( \overline{\Omega} )$ by uniform continuity. With this embedding result, \cite[Theorem 2.4]{K1998JFA} directly yields the existence of $C_1 > 0$, $\lambda_1 \geq 1$ and $2/3 > \delta_1 > 0$ such that every $u \in W^{1,p}( \Omega )$ satisfies
\begin{equation*}
    | u(x) - u(y) |
    \leq
    d(x,y)\left( C_1\mathcal{M}_{ \lambda_1d(x,y) }( \rho_u^{p} )(x) + C_1\mathcal{M}_{ \lambda_1d(x,y) }( \rho_u^{p} )(y) \right)^{1/p},
\end{equation*}
whenever $x,y \in \Omega$ and $d(x,y) < \delta_1$. Indeed,
\begin{align*}
    u^{\sharp}_{s}(z)
    &=
    \sup_{0<r<s}
        \frac{1}{t}
        \aint{ \Omega \cap B(z,t) }\aint{ \Omega \cap B(z,t) }| u(y)-u(x) |\,d\mu(y) \,d\mu(x)
    \\
    &\leq
    \sup_{0<t<s}
        C_1\frac{2t}{t}\left(\aint{ \Omega \cap B(z,t) }({M}_{ \lambda_1 2t }( \widehat{\rho}_u^{p} )(y))^{1/p}  \,d\mu(y) +\aint{ \Omega \cap B(z,t) }   ({M}_{ \lambda_1 2t }( \widehat{\rho}_u^{p} )(x))^{1/p}  \,d\mu(x)\right)
    \\
    &\leq 4C_1 C(C_\mu, \lambda_2)\aint{ \Omega \cap B(z,\lambda_1 s) } ({M}_{ \lambda_1 2s }( \widehat{\rho}_u^{p} )(y))^{1/p}  \,d\mu(y) 
    \\
    &\leq 4C_1 C(C_\mu, \lambda_2) C( C_\mu, p )
    \left(
        \aint{ B(z,3 \lambda_1 s) }
            \widehat{\rho}_{u}^{p}
        \,d\mu
    \right)^{1/p},
\end{align*}
where once again, we apply the weak-type estimate \cite[Lemma 3.5.10]{HKST2015} for the maximal operator. Hence we obtain the conclusion \Cref{prop:Poincare:sharpfunctional} (1), so the proof is complete under the assumption $p > n$.

We next consider the case $p = n$. Here we need to recall an inequality obtained during the proof of \cite[Theorem B]{K1998JFA}. There Koskela proved the following statement: There exist constants $2/3 > \delta_1 > 0$, $C_1 > 0$ and $\lambda_1 \geq 1$ for which
\begin{equation*}
    | u(x) - u(y) |
    \leq
    d(x,y)\left( C_1\mathcal{M}_{ \lambda_1d(x,y) }( \widehat{\rho}_u^{n} )(x) + C_1\mathcal{M}_{ \lambda_1d(x,y) }( \widehat{\rho}_u^{n} )(y) \right)^{1/n},
\end{equation*}
whenever $x, y \in \Omega$, $d(x,y) < \delta_1$ and $u \in W^{1,p}( \Omega )$ for some $p > n$. Therefore, by modifying the argument above, there exist constants $1 > \delta_2 > 0$, $C_2 > 0$ and $\lambda_2 \geq 1$ for which for every $0 < s < \delta_1/( 2 \lambda_2 )$ 
\begin{equation}\label{eq:crucialinequality}
    u_{ s }^{\sharp}
    \leq
    C_2\left( \mathcal{M}_{ 2\lambda_2 s }( \widehat{\rho}_{u}^{n} ) \right)^{1/n}
    \quad\text{for every $u \in \bigcup_{ p > n } W^{1,p}( \Omega )$.}
\end{equation}
On the other hand, we claim that since Lipschitz functions with bounded supports are elements of $W^{1,p}( \Omega )$ for every $p > n$ and also dense in $W^{1,n}( \Omega )$-norm in $W^{1,n}( \Omega )$ (recall \Cref{prop:density}), \eqref{eq:crucialinequality} holds for every $u \in W^{1,n}( \Omega )$. To this end, there exists a sequence of Lipschitz functions $( u_k )_{ k = 1 }^{ \infty }$ satisfying
\begin{align}\label{eq:crucialinequality:step0}
    \lim_{ k \rightarrow \infty }
    \| u - u_k \|_{ W^{1,n}( \Omega ) }
    &=
    0
    \quad\text{and}\quad
    \lim_{ k \rightarrow \infty }
    \| u^{\sharp}_s - (u_k)^{\sharp}_s \|_{ L^n( \Omega ) }
    =
    0.
\end{align}
Observe that the first fact in \eqref{eq:crucialinequality:step0} follows from density of Lipschitz functions with bounded support and the second one from the subadditivity of the sharp functional together with the equivalence of the $M^{1,p}( \Omega )$- and $W^{1,p}( \Omega )$-convergences (due to the equality $M^{1,p}( \Omega ) = W^{1,p}( \Omega )$).

Next, we claim that up to passing to a subsequence and relabeling, we have
\begin{equation}\label{eq:crucialinequality:step1}
    \left( \mathcal{M}_{ 2\lambda_2 s }( \widehat{\rho}_{u}^{n} ) \right)^{1/n}
    =
    \lim_{ k \rightarrow \infty }
    \left( \mathcal{M}_{ 2\lambda_2 s }( \widehat{\rho}_{u_k}^{n} ) \right)^{1/n}
    \quad\text{almost everywhere}.
\end{equation}
Observe that if this can be proved, then \eqref{eq:crucialinequality} follows from \eqref{eq:crucialinequality:step0} and \eqref{eq:crucialinequality:step1} for every $u \in W^{1,n}( \Omega )$. 

Towards proving \eqref{eq:crucialinequality:step1}, we observe that, by the subadditivity of the restricted maximal function and the weak-type estimate \cite[Lemma 3.5.10]{HKST2015} of the maximal operator,
\begin{gather}\notag
    \aint{ B(z, \lambda_2 s) }
        \left(
            \left| \mathcal{M}_{ 2\lambda_2 s }( \widehat{\rho}_{ u }^{n} )
            - 
            \mathcal{M}_{ 2\lambda_2 s }( \widehat{\rho}_{ u_k }^{n} )
            \right|
        \right)^{1/n}
    \,d\mu
    \\\label{eq:controlfromabove}
    \leq\,
    \aint{ B(z, \lambda_2 s) }
        \left( \mathcal{M}_{ 2\lambda_2 s }( \widehat{\rho}_{ u }^{n} - \widehat{\rho}_{ u_k }^{n} ) \right)^{1/n}
    \,d\mu
    \leq
    C(c_\mu,c_\Omega,n)
    \left(
        \aint{ B(z,3 \lambda_2 s) }
            | \widehat{\rho}_{u}^{n} - \widehat{\rho}_{ u_k }^n |
        \,d\mu
    \right)^{1/n}.
\end{gather}
We may now pass to a subsequence such that $\rho_{ u_k }$ converges to $\rho_u$ pointwise almost everywhere. Then \eqref{eq:controlfromabove} implies that, up to passing to a subsequence, $\mathcal{M}_{ 2\lambda_2 s }( \rho_{ u }^{n} )$ is the pointwise limit of $\mathcal{M}_{ 2\lambda_2 s }( \rho_{ u_k }^{n} )$ for almost every point in $\Omega$. Thus the conclusion \eqref{eq:crucialinequality} holds for the limiting function $u \in W^{1,n}( \Omega )$. Now a simple cut-off argument implies \Cref{prop:Poincare:sharpfunctional} (2), so the conclusion \Cref{prop:Poincare:sharpfunctional} (1) follows as claimed. 
\end{proof}

\section{Planar Jordan domains}\label{section:LAST_LAST}
In this section, we consider the two-dimensional sphere identified with the extended plane $\widehat{\rr}^2$ using stereographic projection; we endow the sphere with the Riemannian area measure.
\begin{definition}
Let $\Omega \subset \widehat{\rr}^2$ be a Jordan domain. A homeomorphism $\mathcal{R} \colon \widehat{\rr}^2 \to \widehat{\rr}^2$ is a \emph{reflection} over the boundary $\partial\boz$, if $\mathcal{R}(\boz)=\widehat{\rr^2}\setminus\overline{\boz}$ and $\mathcal R(z)=z$ for every $z\in\partial\boz$.

We say that $\mathcal{R}$ is a \emph{$p$-morphism} if $\mathcal{R} \in W^{1,p}( \widehat{\rr}^2; \widehat{\rr}^2 )$ and 
\begin{equation}\label{eq:Jacobian-morphism}
    | D\mathcal{R} |^p(z) \leq K | J_{\mathcal{R}}|(z) \quad\text{for almost every $z \in\widehat{\rr}^2$}
\end{equation}
for some $K \geq 1$. Here $| D\mathcal{R} |$ and $J_{\mathcal{R}}$, respectively, are the operator norm and the Jacobian of the differential of $\mathcal{R}$.
\end{definition}
We recall that the minimal $p$-weak upper gradient of $\mathcal{R}$ coincides with $| D\mathcal{R}|$ almost everywhere in $\widehat{\rr}^2$.

Associated to each reflection, we obtain the following \emph{folding map}:
\begin{equation*}
    \mathcal{F} \colon \widehat{\rr}^2 \to \widehat{\rr}^2, 
\end{equation*}
where $z \mapsto z$ if $z \in \overline{\Omega}$ and $z \mapsto \mathcal{R}(z)$ otherwise. We say that $\mathcal{R}$ is a \emph{$p$-reflection over $\partial \Omega$} if $\mathcal{R}$ is a $p$-morphism and for any $u \in W^{1,p}( \boz )$, there exists some $\widehat{h} \colon \widehat{\rr}^2 \to \widehat{\rr}^2$ with a $L^{p}( \widehat{\rr}^2 )$-integrable upper gradient with $\widehat{h}|_{ \widehat{\rr}^2 \setminus \partial \Omega } = u \circ \mathcal{F}|_{ \widehat{\rr}^2 \setminus \partial \Omega }$ almost everywhere. 
\begin{theorem}[Corollary 1.4, \cite{PPZ}]\label{le:s6:ref}
Let $\boz\subset \widehat{\rr}^2$ be a Jordan domain. If $\boz$ is a $W^{1, p}$-extension domain, then there exists a $p$-reflection $\mathcal R$ over $\partial \Omega$.
\end{theorem}
We state the main theorem in this section, corresponding to \Cref{prop:planarjordan}. Indeed, \Cref{prop:planarjordan} follows from \Cref{thm:s6:exten} by observing two facts: First, the stereographic projection restricted to the complement of a small neighbourhood of the infinity point is bi-Lipschitz. Second, such a map preserves the $c_0$-valued $W^{1,p}$-extension property of Jordan domains, whose closure does not touch the infinity point. 
\begin{theorem}\label{thm:s6:exten}
    Let $\boz\subset\widehat{\rr}^2$ be an $\mathbb R$-valued $W^{1, p}$-extension Jordan domain for some $1<p<\fz$. Then $\Omega$ is a $\mathbb V$-valued $W^{1, p}$-extension domain for every Banach space $\mathbb V$.  
\end{theorem}
\begin{proof}
Fix a $p$-reflection $\mathcal R$ over $\partial\boz$ obtained from \Cref{le:s6:ref}. Since $\Omega$ is a $W^{1,p}$-extension domain, we recall from \cite{HKT2008:B} that it satisfies the measure-density condition. Then \Cref{prop:density} (and \Cref{lemm:stronglocality}) implies that $W^{1,p}( \Omega ) = W^{1,p}( \overline{\Omega} )$ are isometrically isomorphic.

Consider next an arbitrary Lipschitz $u \in W^{1,p}( \overline{\Omega} )$. Then $h \coloneqq u \circ \mathcal{F}$ is uniquely defined in $\widehat{\rr}^2$ and continuous. Since $\mathcal{R}$ is a $p$-reflection over $\partial \Omega$, there exists some $\widehat{h} \colon \widehat{\rr}^2 \to \widehat{\rr}^2$ with a $L^{p}( \widehat{\rr}^2 )$-integrable upper gradient and $\widehat{h}|_{ \widehat{\rr}^2 \setminus \partial \Omega } = u \circ \mathcal{F}|_{ \widehat{\rr}^2 \setminus \partial \Omega } = h|_{ \widehat{\rr}^2 \setminus \partial \Omega }$. Therefore $\widehat{h} = h$ almost everywhere in $\widehat{\rr}^2$. Since $h$ is continuous, we have that $\widehat{h} = h$ at every Lebesgue  point of $\widehat{h}$. In other words, $h \equiv \widehat{h}$ up to a set negligible for $p$-Sobolev capacity in $\widehat{\rr}^2$ \cite[Theorem 9.2.8]{HKST2015}. Thus $h$ has an $L^{p}( \widehat{\rr}^2 )$-integrable $p$-weak upper gradient \cite[Proposition 7.1.31]{HKST2015}. Since $\partial \Omega$ has negligible measure, we conclude from \eqref{eq:Jacobian-morphism} that
\begin{equation}\label{eq:pointwiseinequality}
    \rho_{ u \circ \mathcal{F} }^p
    \leq
    ( \rho_{u} \circ \mathcal{F} )^p | D\mathcal{F} |^p
    \leq
    K ( \rho_u \circ \mathcal{F} )^p |J_{ \mathcal{F} }|
    \quad\text{almost every $x \in \widehat{\rr}^2$, see e.g. \cite{Wil:12}.}
\end{equation}
We also conclude from \eqref{eq:pointwiseinequality} and $\partial \Omega$ having zero measure that, for $C = K+1$,
\begin{align} \notag
    \int_{ \mathcal{F}^{-1}(E) }
        \rho_{ u \circ \mathcal{F} }^p(y)
    \,dy
    &=
    \int_{ E \cap \Omega }
        \rho_{ u }^p(y)
    \,dy
    +
    \int_{ \mathcal{F}^{-1}( E ) \setminus \overline{ \Omega } }
        \rho_{ u \circ \mathcal{F} }^p(y)
    \,dy
    \\ \label{eq:integralinequality}
    &\leq
    C\int_{ E \cap \Omega }
        \rho_{ u }^p(y)
    \,dy
    \quad\text{for every Borel set $E \subset \overline{\Omega}$.}
\end{align}
Since Lipschitz functions are dense in $W^{1,p}( \overline{ \Omega } )$ (recall \Cref{prop:density}, cf. \Cref{ex:Jordan}), we obtain \eqref{eq:integralinequality} for every $u \in W^{1,p}( \overline{ \Omega } )$. Also \eqref{eq:pointwiseinequality} holds for every $u \in W^{1,p}( \overline{\Omega} )$.

Next, consider $u = ( u_i )_{ i = 1 }^{ \infty } \in W^{1,p}( \Omega; c_0 )$. Let $\overline{u}_i \in W^{1,p}( \overline{\Omega} )$ denote the unique extension of $u_i$ to $\overline{\Omega}$. We denote $u \circ \mathcal{F} \coloneqq ( \overline{u}_i \circ \mathcal{F} )_{ i = 1 }^{ \infty }$. It follows from \eqref{eq:integralinequality} and  \Cref{lemm:projection:new} that
\begin{equation*}
    \int_{ \mathcal{F}^{-1}(E) }
        \rho_{ \overline{u}_i \circ \mathcal{F} }^p(y)
    \,dy
    \leq
    C\int_{ E \cap \Omega }
        \rho_{ u }^p(y)
    \,dy
    \quad\text{for every Borel set $E \subset \overline{\Omega}$.}
\end{equation*}
Since this inequality holds for every Borel set and $i\in\N$, we conclude that
\begin{equation*}
    \int_{ \mathcal{F}^{-1}(E) }
        \sup_{ i \in \mathbb{N} }\rho_{ \overline{u}_i \circ \mathcal{F} }^p(y)
    \,dy
    \leq
    C\int_{ E \cap \Omega }
        \rho_{ u }^p(y)
    \,dy
    \quad\text{for every Borel set $E \subset \overline{\Omega}$.}
\end{equation*}
Thus $\sup_{ i }\rho_{ \overline{u}_i \circ \mathcal{F} }$ is $L^{p}( \widehat{\rr}^2 )$-integrable, so the supremum is also a $p$-weak upper gradient of (a Newton--Sobolev representative of) $u \circ \mathcal{F}$ by \cite[Theorem 7.1.20]{HKST2015}. The $(1,p)$-Poincaré inequality and compactness of $\widehat{\rr}^2$ yields $u \circ \mathcal{F} \in W^{1,p}( \widehat{\rr}^2; c_0 )$ for which $( u \circ \mathcal{F} )|_{ \Omega } = u$. Therefore $\Omega$ is a $c_0$-valued $W^{1,p}$-extension domain. Recall that this implies the corresponding property for every Banach space due to \Cref{cor:c_0:all}.
\end{proof}

\section*{Acknowledgements}
We thank Kai Rajala and Tapio Rajala for comments on an earlier version of this manuscript. We also thank Pekka Koskela for discussions regarding \Cref{thm:extension:topoincare}.

\bibliographystyle{alpha}
\bibliography{Bibliography}

\end{document}